\documentclass{amsart}
  
\usepackage{graphicx}
\usepackage{amscd,amsmath,amsopn,amssymb, amsthm,multicol, wasysym} 
\usepackage{hyperref}
\usepackage[all]{xy}
\usepackage{amscd}
\usepackage{lscape}
\usepackage{slashed}
\usepackage{graphicx}
\usepackage{lscape}
 \usepackage{cancel}
  
\newtheorem{theorem}{Theorem}[section]
\newtheorem{lemma}[theorem]{Lemma}
\newtheorem{prop}[theorem]{Proposition}
\newtheorem{corol}[theorem]{Corollary}

\theoremstyle{definition}
\newtheorem{definition}[theorem]{Definition}

\theoremstyle{remark}
\newtheorem{remark}[theorem]{Remark}

\numberwithin{equation}{section}

 \input ulem.sty
 
\DeclareMathOperator{\Ad}{Ad}
\DeclareMathOperator{\ad}{ad}

\DeclareMathOperator{\Hgt}{ht}

\DeclareMathOperator{\Ric}{Ric}

\newcommand{\fr}{\mathfrak}
\newcommand{\al}{\alpha}
\newcommand{\be}{\beta}

\newcommand{\bb}{\mathbb}
\newcommand{\cal}{\mathcal}

\DeclareMathOperator{\SO}{SO}
\DeclareMathOperator{\Sp}{Sp}
 \DeclareMathOperator{\SU}{SU}
\DeclareMathOperator{\U}{U}
\DeclareMathOperator{\G}{G}
\DeclareMathOperator{\F}{F}

\DeclareMathOperator{\E}{E}
\DeclareMathOperator{\Ss}{S}

\newcommand{\thickline}{\noalign{\hrule height 1pt}}

\begin{document}   

 \title
{Non-naturally reductive Einstein metrics on exceptional Lie groups}
\author{Ioannis Chrysikos and Yusuke Sakane}
 \address{Masaryk University, Faculty of Science, Department of Mathematics and Statistics, Kotl\'a\v{r}sk\'a 2, Brno  61137, Czech Republic}
  \email{chrysikosi@math.muni.cz}
    \address{Osaka University, Department of Pure and Applied Mathematics, Graduate School of Information Science and Technology, Suita, 
Osaka 565-0871, Japan}
 \email{sakane@math.sci.osaka-u.ac.jp}




 
 \medskip

\begin{abstract}
 Given  an exceptional compact   simple Lie group $G$  we  describe new left-invariant Einstein metrics which are not naturally reductive.  In particular,  we consider fibrations of $G$ over  flag manifolds with a certain kind of isotropy representation and we  construct    the Einstein equation with respect to the induced   left-invariant metrics.  Then  we apply a technique based on  Gr\"obner bases and classify  the real solutions of the  associated algebraic systems.   For  the  Lie group  $\G_2$  we obtain the first known example of a left-invariant Einstein metric, which is not naturally reductive.  Moreover,  for the  Lie groups $\E_7$ and  $\E_8$, we conclude  that there exist non-isometric non-naturally reductive   Einstein metrics, which are   $\Ad(K)$-invariant by different Lie subgroups $K$.  
 
 
   \end{abstract}

\maketitle  

 \section{Introduction}
 In 1979 D'Atri and Ziller \cite{DZ} studied naturally reductive metrics on compact semi-simple Lie groups. They gave a complete classification of such metrics on compact simple Lie groups and described many naturally reductive Einstein metrics. They also asked the following  question (cf. \cite{DZ} Remark p.~62):

 
 \smallskip
 \noindent {\bf Question.} {\it Given a compact   simple Lie group $G$, do  there exist   left-invariant Einstein metrics which are not naturally reductive?}

\smallskip
  The first left-invariant Einstein metrics on a compact  simple  Lie group  which are non-naturally reductive, were discovered for $\SU_{n}$ $(n\geq 6)$ by K.~Mori  in 1994 \cite{Mori}.  He considered the Lie group $\SU_{n}$ as a principal bundle over the generalized flag manifold $\SU_{n}/\Ss(\U_{\ell}\times\U_{m}\times\U_{k})$ $(\ell+m+k=n\geq 2)$  and then he used the reverse of Kaluza-Klein ansatz to describe new left-invariant Einstein metrics.  In 2008,   Arvanitoyeorgos,  Mori and  the second author  proved the existence of new non-naturally reductive  Einstein metrics   for $\SO_{n}$ $(n\geq 11)$, $\Sp_{n}$ $(n\geq 3)$, $\E_6$, $\E_7$ and $\E_8$, using fibrations of a   compact simple Lie group  over a     flag manifold  (K\"ahler C-space) with two isotropy summands  (see \cite{AMS}). More recently, Chen and  Liang \cite{Chen} proved that there is a non-naturally reductive left-invariant Einstein metric on the exceptional Lie group $\F_4$. 

 In this paper  we  describe new non-naturally reductive Einstein metrics on compact simple Lie groups $G$, which can be viewed as principal bundles  over    flag manifolds $M=G/K$ with three isotropy summands and second Betti number $b_{2}(M)=1$.  Hence,   the painted Dynkin diagram of $M$ is   defined  by a pair  $(\Pi, \Pi_{K})$ such that $\Pi\backslash\Pi_{K}=\{\al_{i_{o}}\}$ with  $\Hgt(\al_{i_{o}})=3$  for some  simple root $\al_{i_{o}}$.  Here,   $\Pi=\{\al_1, \ldots, \al_{\ell}\}$  is a basis of simple roots  and  $\Hgt(\al_{j})$   is the height (Dynkin mark) of a simple root $\al_{j}$.  Because the heights of a classical compact simple Lie group   are bounded by $1\leq \Hgt(\al_{i})\leq 2$ for any $i=1, \ldots, \ell$ (c.f. \cite{GOV}),  the examined  Lie groups are necessarily   {\it exceptional},  see  Table 2.  From now on, we shall denote such a  Lie group $G$ by $G(\al_{i_{o}})\equiv G(i_{o})$; then  one can  immediately encode    the isotropy subgroup $K$ via the corresponding  painted Dynkin diagram.  Moreover, the related  reductive decomposition $\fr{g}=\fr{k}\oplus\fr{p}$   induces the left-invariant metrics $\langle \ , \ \rangle$ that we are interested in.  
 
  By extending the notation of \cite{AMS}, and since in our case the painted black simple root $\al_{i_{o}}$ is never connected with  the vertex corresponding to the negative  of  the maximal root $\tilde{\al}:=\Hgt(\al_{1})\al_{1}+\ldots+\Hgt(\al_{\ell})\al_{\ell}$, we agree to say that   $G\equiv G(i_{o})$ is  of   Type  $I_{b}$, $II_{b}$, or $III_{b}$,  if  after deleting the black vertex  the Dynkin diagram splits into {\it one}, {\it two}, or {\it three} components (subdiagrams), respectively.  In \cite{AMS} and for compact simple Lie groups $G$ associated to flag manifolds $M=G/K$ with {\it two} isotropy summands,  it was shown that the {\it new} non-naturally reductive Einstein metrics appear only for the corresponding  classes of Type  $I_{b}$ and $II_{b}$ (for the Types    $I_{a}, II_{a}, III_{a}$ the painted black simple root is connected to $-\tilde{\al}$). In particular,  for such flag manifolds, there are still Lie groups of Types    $III_{a}, III_{b}$  (related to $\SO_{2\ell}$, see Theorem \ref{class}) but these cases   have not been examined yet.   In this study,  we  focus on exceptional flag manifolds and  provide   the existence of {\it new} left-invariant non-naturally reductive Einstein metrics  on simple Lie groups of  all 3 types  $I_{b}$, $II_{b}$ and $III_{b}$.
  
  For convenience,  in   Table 1 and for   any exceptional  compact simple Lie group $G\equiv G(i_{o})$,  we list  the number of    {\it non-naturally reductive} left-invariant Einstein metrics found in \cite{AMS}, including our new Einstein metrics and the Einstein metric constructed by  Chen and  Liang \cite{Chen} (although it does not fit into our types). We   denote this number  by
 $\cal{E}(G)_{\rm non-nn}$ and  also state  the  isotropy subgroup $K$,  the numbers $p, q$ appearing in  the reductive decomposition $\fr{g}=\fr{k}\oplus\fr{p}=\fr{k}_{0}\oplus\cdots\oplus\fr{k}_{p}\oplus\fr{p}_{1}\oplus\cdots\oplus\fr{p}_{q}$  and the type of $G(i_{o})$.
 \smallskip
   {\small{ \begin{center}
{\bf{Table 1.}} Number of   non-naturally reductive left-invariant Einstein metrics on $G$ exceptional.
\end{center}}}
         \begin{center}
   \begin{tabular}{c|l|l|c|c|l}
   $G$ &  notation $G\equiv G(i_{o})$ & isotropy group    $K$ &   $p+q$ & \text{Type} &  $\cal{E}(G)_{\rm non-nn}$  \\
   \thickline
   $\G_2$ & $\G_2(2)$ & $\U_2^{l}\cong\SU_2\times\U_1$ & $2+3$ & $I_{b}$ & ${ 1}$  ${\rm new}$\\
   \hline
   $\F_4$ & $\F_4(2)$     & $\SU_3\times\SU_2\times\U_1$ & $3+3$ &  $II_{b}$ & $  5$ ${\rm new}$\\
     & - &   $\SO_{8}$  & - & - &   $1$     \cite{Chen} \\
    \hline
     $\E_6$ & $\E_6(2)\cong\E_6(4)$ & $\SU_5\times\SU_2\times\U_1$ & $3+2$ & $II_{b}$ & $  4$ \cite{AMS}\\
    & $\E_6(3)$ & $\SU_3\times\SU_2\times\SU_2\times\U_1$ & $4+3$ &   $III_{b}$ &  $9$  ${\rm new}$ \\
     \hline
    $\E_7$ & $\E_7(7)$ & $\SU_7\times\U_1$ & $2+2$ & $I_{b}$ & $  2$   \cite{AMS}\\
    & $\E_7(2)$ & $\SO_{10}\times\SU_2\times\U_1$ & $3+2$ & $II_{b}$ & $  4$ \cite{AMS}   \\
    & $\E_7(5)$ & $\SU_6\times\SU_2 \times\U_1$ & $3+3$ & $II_{b}$ & $  5$ ${\rm new}$ \\
    & $\E_7(3)$ & $\SU_5\times\SU_3\times\U_1$ & $3+3$ & $II_{b}$ & $  7$  ${\rm new}$ \\
    \hline
    $\E_8$ & $\E_8(7)$ &  $\SO_{14}\times\U_1$ & $2+2$ & $I_{b}$   & $  2$ \cite{AMS} \\
   & $\E_8(8)$ & $\SU_8\times\U_1$ & $2+3$ & $I_{b}$ & $  3$ ${\rm new}$ \\
    & $\E_8(2)$ & $\E_6\times\SU_2\times\U_1$ & $3+3$ & $II_{b}$ & $  5$ ${\rm new}$ \\
    \thickline
          \end{tabular}   
   \end{center}

\smallskip        
         
For the   description of the Ricci tensor associated to   a left-invariant metric on a  Lie group $G(i_{o})$ of Type  $I_{b}$, $II_{b}$ and $III_{b}$, we exploit the Lie theoretic description of flag manifolds $M=G(i_{o})/K$ and apply a basic method of Riemannian submersions, see for example Lemma \ref{ric}.   Notice that the induced  left-invariant metrics are given in terms of  5, 6, or 7 parameters, depending on the explicit type of $G=G(i_{o})$.  Hence, the Ricci tensor is complicated and    we often need to   consider new fibrations,  inducing  new left-invariant Riemannian metrics.   Then, a comparison of these metrics with the initial one $\langle \ , \ \rangle$ provides   all the necessary information for an explicitly description of the associated Einstein equation, see for example Lemma  \ref{gg2} or Lemma \ref{A345}. We mention  that this approach   leads   to a uniform description of the Ricci tensor and it has been     successfully   applied  in a series of   works   \cite{AMS, Chry2, CS}.  

After this step, we  proceed to a systematic examination of  the algebraic systems defined by the corresponding 
Einstein equation with respect to $\langle \ , \ \rangle$.  
With the aid of computer, 
we can describe  Gr\"obner bases for algebraic systems and    classify  real  solutions of the corresponding  
Einstein equation.   Finally  based on  Propositions \ref{prop4.6}, \ref{prop5.5}, \ref{prop6.4}, we deduce  which of these solutions induce  {\it new} non-naturally reductive left-invariant Einstein metrics.    We summarize  our results (up to isometry) in the following theorem  (for useful details on the examination of the isometry  problem we refer to \cite{CS, Chry2, Chry}).
  \begin{theorem}\label{THM1}
 (a) The compact  simple Lie group $\G_2$ admits at least  one non-naturally reductive  left-invariant Einstein metric. In particular, this metric  is $\Ad(\U_{2}^{l})$-invariant.
  
 (b)  The compact  simple Lie group $\F_4$ admits at least five new non-naturally reductive and non-isometric  left-invariant Einstein metrics.  These metrics are $\Ad(\SU_{3}\times\SU_2\times\U_1)$-invariant. 
 
(c) The compact  simple Lie group $\E_6$ admits at least nine new non-naturally reductive and non-isometric  left-invariant Einstein metrics. These metrics are $\Ad(\SU_{3}\times\SU_{2}\times\SU_{2}\times\U_{1})$-invariant. 

(d) The compact  simple Lie group $\E_7$ admits at least  twelve new non-naturally reductive and non-isometric left-invariant Einstein metrics.  Five of these metrics are $\Ad(\SU_{6}\times\SU_{2}\times\U_{1})$-invariant and the other seven are $\Ad(\SU_{5}\times\SU_{3}\times\U_{1}))$-invariant.

(e) The compact  simple Lie group $\E_8$ admits at least  eight  new non-naturally reductive and non-isometric left-invariant Einstein metrics. Three of these metrics are $\Ad(\SU_{8}\times\U_{1})$-invariant and  the  other five are $\Ad(\E_{6}\times\SU_{2}\times\U_{1})$-invariant. 
   \end{theorem}
A careful  analysis of the Einstein metrics described by Gibbons, L\"u and Pope \cite{Pope} on the compact  simple Lie group $\G_2$,  shows  that they  are naturally reductive, 
which is 
contrary to the claim  appearing  in the introduction of \cite{Chen}.  In particular,  the metric  described in Theorem \ref{THM1} is, 
to the best of our   knowledge,  the first known example of a left-invariant Einstein metric  on $\G_2$, which is {\it not} naturally reductive. Moreover,  another direct conclusion of the present work  is the existence of Lie groups $G$, that is, $\E_7, \E_8$, for which  one can construct non-isometric, non-naturally reductive left-invariant Einstein metrics which are $\Ad(K)$-invariant for different Lie subgroups $K\subset G$.  More examples in this direction occur  now in  combination with the  results  of \cite{AMS, Chen}, that is,  for the Lie groups $\F_4, \E_6$.  We remark, however,  that there are also subgroups $K\subset G$ for which the associated reductive decompositions does not induce any new left-invariant Einstein metric, for $(G=\G_2, K=\U_{2}^{s})$ and   $(G=\F_4, K=\Sp_{3}\times\U_1)$.  Eventually,  the results stated in Table 1   leads us to  conjecture that   Lie groups which can be viewed as principal bundles  over flag manifolds $M=G/K$ with $b_{2}(M)=1$ and $q\geq 4$, admit more non-naturally reductive left-invariant Einstein metrics.  By Theorem \ref{class}, this conjecture may apply for the Lie groups  $\F_4(3)$, $\E_8(3)$, $\E_8(6)$, $\E_7(4)$, $\E_8(4)$, $\E_{8}(5)$   and in the special case $q=2$, for the classical Lie group $\SO_{2\ell}(\ell-2)$ with $\ell\geq 5$ (we mean the case that the painted black root is not connected to $-\tilde{\al}$).

\smallskip
        \noindent{\bf Acknowledgements.}
          The first author gratefully acknowledges support by Grant Agency of Czech Republic (GA\v{C}R),  post-doctoral grant GP14-24642P.   He  also  thanks  Department of  Mathematics at Osaka University  for its hospitality, during a research stay in October 2014, where  a part of this work  written. 
The work was supported by JSPS KAKENHI Grant Number 25400071.

   \section{Invariant metrics and the Ricci tensor}  \label{prel}
 

  Let  $(M=G/K, g)$ be a compact homogeneous Riemannian manifold,  where $G\subset I(M)$ is a  closed subgroup of the isometry group   and $K$ is the isotropy subgroup at a fixed point $o=eK\in M$. Through this paper we shall denote by $B\equiv B_{G}$  the negative of the Killing form of $\fr{g}=T_{e}G$.  Without loss of generality we can assume that  $G$ acts (almost) effectively on $M$ and moreover that it is connected and semi-simple (see \cite{Bes}).  Fix a $B$-orthogonal $\Ad(K)$-invariant   complement  $\fr{m}\perp\fr{k}$ of $\fr{k}$ in $\fr{g}$ such that   $\fr{g}=\fr{k}\oplus\fr{m}$ and $\Ad(K)\fr{m}\subset\fr{m}$.  Then,      $\fr{m}$ is identified with the tangent space $T_{o}(G/K)$  $(o=eK\in G/K)$, and  the  isotropy  representation $\chi : K\to \SO(\fr{m})$ of $K$ coincides with the restriction of the adjoint representation $\Ad_{G}|_{K}$ on $\fr{m}$. Thus we may identify $g$  with an $\Ad(K)$-invariant inner product $( \ , \ ) :  \fr{m}\times\fr{m}\to\bb{R}$ on $\fr{m}$.   Traditionally, we call $(M=G/K, g)$  naturally reductive if there exist $G$ and $\fr{m}$ as above, such that  the endomorphism $\ad(X) : \fr{m}\to\fr{m}$ be skew-symmetric with respect to $( \ , \ )$ for any $X\in\fr{m}$.
  
     In \cite{DZ}, D'Atri and Ziller examined naturally reductive metrics among left invariant metrics on compact Lie groups, in particular in the simple case they presented the complete classification of such metrics. Let us recall some details.      Consider   a compact connected semi-simple Lie group $G$ and let $H$ be a closed subgroup.  We shall write $\fr{h}=\fr{h}_{0}\oplus\fr{h}_{1}\oplus\cdots\oplus\fr{h}_{p}$ for a decomposition of the   Lie algebra  $\fr{h}=T_{e}H$   into its centre $\fr{h}_{0}:=Z(\fr{h})$ and simple ideals $\fr{h}_{i}$, for $i=1, \ldots, p$.  Let $\fr{p}$  be  the orthogonal complement  of $\fr{h}\subset\fr{g}$ with respect to $B$. Then, $\fr{g}=\fr{h}\oplus\fr{p}$ and $\Ad(H)\fr{p}\subset\fr{p}$. 
     Now, the Lie group $G\times H$ acts almost effectively on $G$ with isotropy group the diagonal group $\Delta(H)=\{(h, h) : h\in H\}$.  Thus, $G$ can be viewed as coset $(G\times H)/\Delta(H)$ with  $\fr{g}\oplus\fr{h}=\Delta(\fr{h})\oplus \fr{G}$, where we identify  $\fr{G}\cong  T_{e}(G\times H/\Delta{H})\cong \fr{g}$ via the linear map $\fr{G}\ni (X, Y)\mapsto X-Y\in\fr{g}$. 
              
       \begin{theorem} \label{ziller} \textnormal{(\cite[Thm.~1, Thm.~3]{DZ})}  For any inner product  $b$ on the centre $\fr{h}_{0}$ of $\fr{h}$, the following left-invariant metric on $G$ is  naturally  reductive   with respect to the action $(g, h)y=gyh^{-1}$ of  $G\times H$:
   \[
   \langle \ , \ \rangle=u_{0}\cdot b|_{\fr{h}_{0}}+u_{1}\cdot B|_{\fr{h}_{1}}+\cdots+u_{p}\cdot B|_{\fr{h}_{p}}+x\cdot B|_{\fr{p}}, \quad (u_{0}, u_{1}, \ldots, u_{p}, x\in\bb{R}_{+}).
   \]
    Conversely, if a left invariant metric $\langle \ , \  \rangle$ on a compact {\it simple} Lie group $G$ is naturally reductive, then there exist a closed subgroup $H\subset G$ such that  $\langle \ , \ \rangle$ can be written as above.      
    \end{theorem}
The $\Ad(H)$-invariant orthogonal complement  $\fr{p}$ coincides with the tangent space of  the homogeneous space $G/H$,  i.e.  $\fr{p}\cong T_{o}G/H$. From now on  we assume that  ${\frak p} = {\frak p}_1 \oplus \cdots \oplus {\frak p}_q$ defines an orthogonal  decomposition
of   $\fr{p}=T_{o}(G/K)$   into $q$ irreducible mutually non-equivalent $\mbox{Ad}(H)$-modules 
${\frak p}_j$ $( j = 1, \cdots, q )$.  We also assume  that the ideals $\fr{h}_{i}$ $(i=1, \ldots, p)$ are mutually non-isomorphic with  $\dim {\frak h}_0 \leq 1$.
   With the aim to give a unified expression for the Ricci tensor of a left invariant metric  on $G$ and a $G$-invariant metric on $G/H$, it is useful  to  express  the decomposition  of $\fr{g}$ (and hence also of $\fr{h}$ and $\fr{p}$) as follows:  
\begin{equation}\label{go}
{\frak g} =(\fr{h}_{0}\oplus\fr{h}_{1}\oplus\cdots\oplus\fr{h}_{p})\oplus ({\frak p}_1 \oplus \cdots \oplus {\frak p}_q)=({\frak m}_0
\oplus {\frak m}_1 \oplus \cdots \oplus {\frak m}_p) \oplus ({\frak
m}_{p+1} \oplus \cdots \oplus {\frak m}_{p+q})
\end{equation}
with $\fr{h}=\fr{h}_{0}\oplus\fr{k}_{1}\oplus\cdots\oplus\fr{h}_{p}\cong\fr{m}_{0}\oplus\fr{m}_{1}\oplus\cdots\oplus\fr{m}_{p}$ and ${\frak p} = {\frak p}_1 \oplus \cdots \oplus {\frak p}_{q}\cong{\frak
m}_{p+1} \oplus \cdots \oplus {\frak m}_{p+q}$, respectively.  Then,  the following left-invariant metric on $G$  is in fact
$\mbox{Ad}(H)$-invariant: 
\begin{eqnarray}
 \langle \ , \ \rangle &=&  
u_0\cdot B|_{\mbox{\footnotesize$ \frak h$}_0} +
 u_1\cdot B|_{\mbox{\footnotesize$ \frak h$}_1} + \cdots + 
 u_p\cdot B|_{\mbox{\footnotesize$ \frak h$}_p} + 
x_1\cdot B|_{\mbox{\footnotesize$ \frak p$}_1} + \cdots + 
 x_q\cdot B|_{\mbox{\footnotesize$ \frak p$}_q},\nonumber \\
 &=& y_0\cdot B|_{\mbox{\footnotesize$ \frak m$}_0} +
 y_1\cdot B|_{\mbox{\footnotesize$ \frak m$}_1} + \cdots + 
 y_p\cdot B|_{\mbox{\footnotesize$ \frak m$}_p} + 
y_{p+1}\cdot B|_{\mbox{\footnotesize$ \frak m$}_{p+1}} + \cdots + 
 y_{p+q}\cdot B|_{\mbox{\footnotesize$ \frak m$}_{p+q}}\label{left}
\end{eqnarray}
where 
$u_{a}, x_{b}, y_{c}\in{\mathbb R}_+$ for any $0\leq a\leq p$, $1\leq b\leq q$, $0\leq y_{c}\leq  p+q$. Similarly,   any $G$-invariant Riemannian metric  on $G/H$ is given by 
\begin{eqnarray}
 ( \  , \ )=  
x_1\cdot B|_{\mbox{\footnotesize$ \frak p$}_1} + \cdots + 
 x_q\cdot B|_{\mbox{\footnotesize$ \frak p$}_q}=y_{p+1}\cdot B|_{\mbox{\footnotesize$ \frak m$}_{p+1}} + \cdots + 
 y_{p+q}\cdot B|_{\mbox{\footnotesize$ \frak m$}_{p+q}}.  \label{inv}
\end{eqnarray}

Let up present now a formula for the Ricci tensor associated to the  left-invariant metric $\langle \ , \ \rangle$ on $G$ defined by (\ref{left}), and describe the same time the Ricci tensor of $M=G/H$ with respect to the $G$-invariant metric $( \ , \ )$ given by (\ref{inv}) (see also \cite{Wa2, PS, Chry2}).     Set from now on $d_{i}:=\dim_{\bb{R}}\fr{m}_{i}$ and let $\lbrace e_{\mu}^{i} \rbrace_{\mu=1}^{d_{i}}$ be a $B$-orthonormal basis 
adapted to the decomposition of $\frak g$, i.e., 
$e_{\mu}^{i} \in {\frak m}_i$ for some $i$, 
$\mu < \nu$ if $i<j$  with $e_{\mu}^{i} \in {\frak m}_i$ and
$e_{\nu}^{j} \in {\frak m}_j $, for any $0\leq i, j\leq p+q$.  
 Consider the numbers ${A^\xi_{\mu
\nu}}=B(\left[e_{\mu}^{i},e_{\nu}^{j}\right],e_{\xi}^{k})$ such that
$\left[e_{\mu}^{i},e_{\nu}^{j}\right] 
= {\sum_{\xi}
A^\xi_{\mu \nu} e_{\xi}^{k}}$, and set 
\[
A_{ijk}:=\displaystyle{k \brack {ij}}=\sum (A^\xi_{\mu \nu})^2 
\]
 where the sum is
taken over all indices $\mu, \nu, \xi$ with $e_\mu^{i} \in
{\frak m}_i,\ e_\nu^{j} \in {\frak m}_j,\ e_\xi^{k}\in {\frak m}_k$. 
Then, $A_{ijk}$ are independent of the 
$B$-orthonormal bases chosen for ${\frak m}_i, {\frak m}_j, {\frak m}_k$,
and symmetric in all three indices, i.e. $A_{ijk}=A_{jik}=A_{kij}$.

Note that  $\Ric_{\langle \ , \ \rangle}$ is also an  $\Ad(H)$-invariant symmetric (covariant) tensor. In particular, since $\fr{m}\ncong\fr{m}_{j}$ for any $0\leq i\neq j\leq p+q$, it is $\Ric_{\langle \ , \ \rangle}(\fr{m}_{i}, \fr{m}_{j})=0$ for $i\neq j$, that is,  the Ricci tensor is diagonal.  Now,    for  any $k=0,\ldots, p+q$, the set 
$\{e_{\mu}^{k} / \sqrt{y_{k}}\}_{\mu=1}^{d_{k}}$ 
   forms an $\langle \ , \ \rangle$-orthonormal basis of $\fr{m}_{k}$.  Consider real numbers $r_{k}$  defined by $r_{k}:=\Ric_{\langle \ , \ \rangle}(e_{\mu}^{k} / \sqrt{y_{k}}, e_{\mu}^{k} / \sqrt{y_{k}})=(1/y_{k})\Ric_{\langle \ , \ \rangle}(e_{\mu}^{k}, e_{\mu}^{k})$.  
   Then,  we have $\Ric_{\langle \ , \ \rangle}=\sum_{k=0}^{p+q}y_{k}r_{k}B|_{\fr{m}_{k}}$.

\begin{lemma}\label{ric}\textnormal{(\cite{PS, AMS})}
Let $G$ be a compact connected semi-simple Lie group endowed with the left-invariant metric $\langle \ , \ \rangle$  given  by
  (\ref{left}). Then the components $r_0, r_1, \cdots, r_{p+q}$ 
of the Ricci tensor $\Ric_{\langle  \ , \ \rangle}$ associated to $\langle \ , \ \rangle$, are expressed as follows:
\[
r_k = \frac{1}{2y_k}+\frac{1}{4d_k}\sum_{j,i}
\frac{y_k}{y_j y_i} {k \brack {ji}}
 -\frac{1}{2d_k}\sum_{j,i}\frac{y_j}{y_k y_i} {j \brack {ki}}
 \quad (k= 0, 1,\ \cdots,\ p+q).
\]
Here,   the sums are taken over all $i, j = 0, 1,\cdots, p+q$.  In particular, for each $k$ it holds that
\[
 \sum_{i,j}^{p+q}{j \brack {ki}} =\sum_{i,j}^{p+q}A_{kij}= d_k:=\dim_{\bb{R}}\fr{m}_{k}.
 \]
  For  $k=p+1, \ldots, p+q$ and by considering the sums appearing  in the expression of $r_{k}$ only for $i, j$ with $p+1\leq i, j, \leq p+q$,  one obtains the components ${\bar r}_{p+1}, \cdots, {\bar r}_{p+q}$  of the Ricci tensor ${\bar r}$ of the $G$-invariant  metric $( \ , \ )$ on $G/H$  defined by (\ref{inv}).
\end{lemma}

  \section{Flag manifolds with  second Betti number $b_{2}(M)=1$ and  Lie groups os Type I, II, or III}\label{algebraic}
 Let $G$ be a compact connected simple Lie group with finite centre and  Dynkin diagram $\Gamma(\Pi)$, where $\Pi$ denotes a  basis of simple roots.  We  are interested in $K$-principal fibrations of $G$ over flag manifolds $M=G/K$  with the aim to   build  left-invariant metrics on $G$ via a metric on  the base $G/K$ and a metric on the fiber $K$.  For a Lie theoretic description of flag manifolds in terms of painted Dynkin diagrams we refer to \cite{Ale, Chry2, Chry, CS}.

 \subsection{Lie groups of Type $I$, $II$ or $III$}It is well-known  (see \cite{Chry, PhD, CS})  that by painting  black on $\Gamma(\Pi)$ the vertex of   a simple root, say $\al_{i_{o}}$  for some $1\leq i_{o}\leq \ell$,   we define  a flag manifold $M=G/K$ whose isotropy representation decomposes into $q:=\Hgt(\al_{i_{o}})\in\bb{Z}_{+}$ mutually inequivalent, irreducible $\Ad(K)$-submodules, i.e.
 \[
 \fr{p}=T_{o}M=\fr{p}_{1}\oplus\cdots\oplus\fr{p}_{q}.
 \]
  In fact, these are the flag manifolds $M=G/K$ with $b_{2}(M)=1$ and their classification can be found for example in \cite[Table 1]{CS}.  
  Because $M=G/K$ is just defined   by fixing a simple root $\al_{i_{o}}$, we will denote the corresponding Lie group $G$ by $G(i_{o})\equiv G(\al_{i_{o}})$.  We shall  write $G(i)\cong G(j)$ when  the flag manifolds obtained by the subsets $\Pi_{M}=\{\al_{i}\}$ and $\Pi'_{M}=\{\al_{j}\}$, are isomorphic.

Given the   Dynkin diagram $\Gamma(\Pi)$ of $G$,     by deleting a vertex we obtain  at most three components (subdiagrams), with the 3-component  case appearing only for the $G\cong \E_6, \E_7, \E_8$ or $\SO_{2\ell}$. These components  correspond to the  Dynkin  diagrams which define the semi-simple part of $K$.  Hence, we have the same number of simple  ideals in $\fr{k}$ and  components in $\Gamma(\Pi)$ after deleting $\al_{i_{o}}$, and as long as this number increases it brakes up the symmetry of $G(i_{o})$.    The centre $\fr{k}_{0}\cong\fr{u}_{1}$ corresponds to the black vertex and we shall use also a double circle \ $\begin{picture}(0, 2)(0, 0)
     \put(2,2.8){\circle{6}}
\put(2,2.8){\circle{3}}\end{picture}$ \  \  to  denote the negative of the maximal root $\tilde{\al}=\sum_{i=1}^{\ell}\Hgt(\al_{i})\al_{i}$, with respect to the fixed basis $\Pi$. Finally,  in the splitting $\fr{k}=\fr{k}_{0}\oplus\fr{k}_{1}\oplus\cdots\fr{k}_{p}$ of the isotropy subalgebra $\fr{k}$ into its centre and simple idelas, we agree to denote by $\fr{k}_{1}$ the  subalgebra  whose    Dynkin  diagram is  connected with $-\tilde{\al}$.   
  \begin{definition}\label{types}
   We separate the compact simple Lie groups $G\equiv G(i_{o})\equiv G(\al_{i_{o}})$ $(1\leq i_{o}\leq \ell)$ into three types (and similarly   the corresponding flag  manifolds $M=G/K$ with $b_{2}(M)=1$), namely Type $I$, $II$, or $III$, depending on the number of  components (namely one, two or three components) on  the Dynkin diagram $\Gamma(\Pi)$ of $G(i_{o})$,  after deleting the black vertex corresponding to $\al_{i_{o}}$.
  In particular, we shall write $I(q), II(q)$, or $III(q)$, where $q=\Hgt(\al_{i_{o}})$ coincides with the number of the isotropy summands of $M=G/K$. We further divide each of these  three types, into two  subclasses by inserting a subscript $a$ or  $b$, e.g.   Type $I_{a}(q)$ or Type $I_{b}(q)$, depending whether the black vertex is connected to  $-\tilde{\al}$, or not.
  \end{definition}  

Let us classify now  the compact simple Lie groups    with respect  to  the above 6 types. We mention that  maybe  not all  the types appear for any $q$ with $1\leq q\leq 6$, or for some of them just a few Lie groups exist. For example 
\begin{prop}
The unique compact simple Lie group $G=G(i_{o})=G(\al_{i_{o}})$  of Type $III_{a}(q)$ for some $q$ $(1\leq q\leq 6)$, is the Lie group $\SO_{8}(2)$ with $q=2$.
\end{prop}
\begin{proof}
By definition, a Lie group  $G(i_{o})$ of Type $III(q)$    $(1\leq q\leq 6)$ can only be  isometric   to one of $\SO_{2\ell}$, $\E_6, \E_7, \E_8$, namely: $\SO_{2\ell}(\ell-2)$, $\E_{6}(3)$, $\E_{7}(4)$ and $\E_8(5)$. In fact, due to  the form of extended Dynkin diagrams (see  \cite{AMS, GOV}),   these groups  are always of Type $III_{b}(q)$, i.e. in any case  the painted black simple root   is not connected with $-\tilde{\al}$, except $\SO_{8}(2)$ which is of Type $III_{a}(q)$ with $q=2$; the corresponding flag manifold is $\SO_{8}/(\SU_{2}\times\SU_{2}\times\SU_{2}\times\U_{1})$.  For the general case  $\SO_{2\ell}(\ell-2)$ with $\ell\geq 5$, the associated flag manifold $\SO_{2\ell}/(\SU_{\ell-2}\times\SU_{2}\times\SU_{2}\times\U_{1})$  is of Type $III_{b}(q)$ with $q=2$. For $\ell=4$ it is still $q=2$ but as we explained before, in this case  the type changes.
For the exceptional Lie groups of type $III(q)$,  it is $q=3, 4$ and $6$, respectively, with  $\E_6/(\SU_{3}\times\SU_{3}\times\SU_{2}\times\U_{1})$, $ \E_7/(\SU_{4}\times\SU_{3}\times\SU_{2}\times\U_{1})$, and $\E_8/(\SU_{5}\times\SU_{3}\times\SU_{2}\times\U_{1})$  being the corresponding K\"ahler C-spaces, see also \cite[Table 1]{CS}.  Their exact  type is $III_{b}$. The Lie group $\E_6(3)$ will be examined  in Section \ref{TypeIII}.
\end{proof}

\begin{theorem}\label{class}
 Let $M=G/K$ be a  generalized flag manifold with $b_{2}(M)=1$, where $G$ is a compact, connected, simple Lie group.  The classification of   $M=G/K$ (and hence of $G$) with respect to the Types $I_{a}(q)$, $I_{b}(q)$, $II_{a}(q)$, $II_{b}(q)$, $III_{a}(q)$, and $III_{b}(q)$, for $q\geq 1$, is given as follows:
\end{theorem}
{\small{\begin{center}
\begin{tabular}{l|l}
 Type  & compact simple Lie group $G\equiv G(i_{o})\equiv G(\al_{i_{o}})$\\
\thickline  
  Type $I_{b}(1)$ & $\SO_{2\ell+1}(1)$, $\Sp_{\ell}(\ell)$, $\SO_{2\ell}(1)$, $\SO_{2\ell}(\ell-1)\cong\SO_{2\ell}(\ell)$, $\E_{6}(1)\cong\E_{6}(5)$, $\E_7(1)$\\
 
    Type $II_{b}(1)$  & $\SU_{\ell}(\rho)$ $(1\leq \rho\leq \ell-1)$\\
\hline
Type $I_{a}(2)$ & $\Sp_{\ell}(1)$, $\E_6(6)$, $\E_7(6)$, $\E_8(1)$, $\F_{4}(1)$, $\G_2(1)$\\
 
Type $I_{b}(2)$ & $\E_7(7)$, $\E_8(7)$, $\F_4(4)$\\
 
Type $II_{a}(2)$ & $\SO_{2\ell+1}(2)$, $\SO_{2\ell}(2)$\\
 
Type $II_{b}(2)$ & $\SO_{2\ell+1}(\rho)$ $(2\leq \rho\leq \ell-1)$, $\Sp_{\ell}(\rho)$ $(2\leq \rho\leq\ell-1)$, $\SO_{2\ell}(\rho)$ $(3\leq \rho\leq \ell-3)$, $\E_{6}(2)\cong\E_{6}(4)$, $\E_7(2)$\\
 
Type $III_{a}(2)$ & $\SO_{8}(2)$ \\
Type $III_{b}(2)$ & $\SO_{2\ell}(\ell-2)$ $(\ell\geq 5)$\\
\hline
Type $I_{b}(3)$ & $\G_2(2)$, $\E_8(8)$\\

Type $II_{b}(3)$ & $\F_4(2)$, $\E_7(3)$, $\E_7(5)$, $\E_8(2)$\\

Type $III_{b}(3)$ & $\E_6(3)$\\
\hline
Type $II_{b}(4)$ & $\F_4(3)$, $\E_8(3)$, $\E_8(6)$\\

Type $III_{b}(4)$ & $\E_7(4)$\\
\hline
Type $II_{b}(5)$ & $\E_8(4)$\\
\hline
Type $III_{b}(6)$ & $\E_{8}(5)$\\
\thickline
 \end{tabular}
\end{center}}}
\begin{proof}
The case  $q=2$ has  been  already examined in \cite{AMS}, except the Types $III_{a}(2)$ and   $III_{b}(2)$  which we include  here.  These result  exhaust all possible types with $q=2$.  Now, due to the form of the maximal root $\tilde{\al}$ we need to examine the cases $q=1, 3, 4, 5, 6$.  For $q=1$, $M=G/K$ is isometric to a compact  isotropy irreducible Hermitian symmetric space.   The classification of flag manifolds with three isotropy summands  and $b_{2}(M)=1$ was given in \cite{Kim} and this with four isotropy summands was obtained in \cite{Chry2, PhD}.  The cases $q=5, 6$ appear only for $\E_8$, see  \cite{CS}.  Now, the presented results are a combination  of the Definition \ref{types} and the  (extended) painted Dynkin diagrams associated to flag manifolds with $b_{2}(M)=1$.
\end{proof}

 \subsection{Flag manifolds with three isotropy summands and $b_{2}(M)=1$}
From now on we  focus on flag manifolds $M=G/K$ with three isotropy summands $\fr{p}=\fr{p}_{1}\oplus\fr{p}_{2}\oplus\fr{p}_{3}$ and second Betti number $b_{2}(M)=1$. Hence, $M=G/K$ is defined by a pair $(\Pi, \Pi_{K})$ such that $\Pi_{M}:=\Pi\backslash\Pi_{K}=\{\al_{i_{o}}\}$ with $\Hgt(\al_{i_{o}})=3$, see also \cite{Kim, stauros}.
 As one can read from Theorem \ref{class},  such  pairs $(\Pi, \Pi_{K})$  exist  only for an  exceptional   simple Lie group, in particular   $G(i_{o})\equiv G(\al_{i_{o}})$ must be  isometric to one of   the following Lie groups: $\G_2(2)$, $\E_8(8)$, $\F_4(2)$, $\E_7(3)$, $\E_7(5)$,  $\E_8(2)$, and  $\E_6(3)$.   For these  groups  we present in Table 2 the  associated  flag manifolds  $G(i_{o})/K$ (via their painted  Dynkin diagrams) and we state   the necessary dimensions $D_{i}:=\dim_{\bb{R}}\fr{p}_{i}$ for $i=1, 2, 3$.   By $\U^{\it l}_{2}$ (resp. $\U^{\it s}_{2}$) we  denote  the Lie group isomorphic to $\U_{2}\cong \SU_{2}\times\U_{1}$, generated by the long (resp. short) root of the basis $\Pi=\{\al_1=e_2-e_3, \al_2=-e_2\}$ associated to the root system of $\G_2$.  For this case, recall  that     $\|\al_1\|=\sqrt{3}\|\al_2\|$ and  $\tilde{\al}=2\al_1+3\al_2$.  The highest roots of the  exceptional groups $\F_4, \E_6, \E_7$ and $\E_8$ with respect to the used fundamental bases (see \cite{PhD, GOV}) are given as follows (we write only the heights  $(\Hgt(\al_{1}), \ldots, \Hgt(\al_{\ell}))$):
  \[
\F_{4} : (2, 3, 4, 2), \quad \E_6 : (1, 2, 3, 2, 1, 2), \quad \E_7 :  (1, 2, 3, 4, 3, 2, 2) , \quad \E_8 : (2, 3, 4, 5, 6, 2, 4, 3).
\] 
\begin{remark}\textnormal{
   An explicit computation of the triples $A_{ijk}$ associated to the reductive decomposition of the Lie algebra $\fr{g}$ of   $G=G(i_{o})$, is possible  via their  definition. This method   access a description   in terms of the structure constants  of the corresponding Lie algebra  (see also \cite{Sak}). In general, we avoid this technique, since we need to repeat it separately for any Lie group of each type  and it  becomes combinatorial.   The same time,  Theorem \ref{class} and the classification of the simple Lie groups into different types $I(q)$, $II(q)$, $III(q)$, $(1\leq q\leq 6)$ allows us to compute   $A_{ijk}$ explicitly, in a generalized  way. This means   that often   Lie groups of exactly  the same type   can be  treated   simultaneously; in this situation one can  describe the Ricci tensor and construct the Einstein equation for the fixed left-invariant metric only once, and not for any Lie group separately; we refer to   \cite{AMS} for the  Types $I(q), II(q)$ and $q=2$ and see Sections \ref{TypeI}, \ref{TypeII} and \ref{TypeIII}  for the  Types $I(q), II(q)$ and $III(q)$ with $q=3$, respectively.}
 \end{remark}

 {\small{ \begin{center}
{\bf{Table 2.}} Painted Dynkin diagrams of flag manifolds $M=G/K$ with  $\fr{p}=\fr{p}_{1}\oplus\fr{p}_{2}\oplus\fr{p}_{3}$   and $b_{2}(M)=1$.
\end{center}
\begin{center}
\begin{tabular}{r|l|c|l|c}
      & $G$ & $(\Pi, \Pi_{K})$ & $K$ &  $(D_{1}, D_{2}, D_{3})$   \\
     \thickline
     Type $I_{b}(3)$ &  & &  \\
     \hline
     & $\G_2(2)$ &  $\begin{picture}(160, 20)(30, 6)
     \put(88,10){\circle{6}}
\put(88,10){\circle{3}}
\put(91, 10){\line(1,0){11.8}}
\put(105,10){\circle{4.7}}
\put(105, 18){\makebox(0,0){$\al_1$}}
\put(106.8, 8.1){\line(1,0){13.6}}
\put(107.2,10){\line(1,0){16.6}}
\put(106.8,11.9){\line(1,0){13.6}}
\put(111, 8.25){\scriptsize $>$}
\put(122,10){\circle*{4.8}}
\put(122, 18){\makebox(0,0){$\al_2$}}
\end{picture}$ & $\U^{\it l}_{2}$ & $(4, 2, 4)$  
\\
\hline
& $\E_8(8)$ & $\begin{picture}(160,40)(-10,-10)
\put(-3,9.5){\circle{6}}
\put(-3,9.5){\circle{3}}
\put(-0, 10){\line(1,0){12.9}}
\put(15, 9.5){\circle{4.7 }}
\put(15, 18){\makebox(0,0){$\al_1$}}
\put(17, 10){\line(1,0){14}}
\put(33.5, 9.5){\circle{4.7 }}
\put(33.5, 18){\makebox(0,0){$\al_2$}}
\put(35.5, 10){\line(1,0){13.4}}
 \put(51, 9.5){\circle{4.7 }}
 \put(51, 18){\makebox(0,0){$\al_3$}}
\put(53,10){\line(1,0){13.5}}
\put(69,9.5){\circle{4.7 }}
\put(69, 18){\makebox(0,0){$\al_4$}}
\put(89,-8){\circle*{4.7}}
\put(99, -9.5){\makebox(0,0){$\al_8$}}
\put(89,-6){\line(0,1){13}}
\put(71.2,10){\line(1,0){15.4}}
\put(89,9.5){\circle{4.7 }}
\put(89, 18){\makebox(0,0){$\al_5$}}
\put(91,10){\line(1,0){15.8}}
\put(109,9.5){\circle{4.7 }}
\put(109, 18){\makebox(0,0){$\al_6$}}
\put(111,10){\line(1,0){15.8}}
\put(129,9.5){\circle{4.7 }}
\put(129, 18){\makebox(0,0){$\al_7$}}
\end{picture}
$ & $\SU_{8}\times \U_{1}$ & $(112, 56, 16)$  \\
\thickline 
Type $II_{b}(3)$  & & &   \\
\hline 
& $\F_4(2)$ &  $\begin{picture}(160,25)(30, 0)
  \put(70,10){\circle{6}}
\put(70,10){\circle{3}}
\put(73, 10){\line(1,0){11.8}}
 \put(87,10){\circle{4.7}}
\put(87, 18){\makebox(0,0){$\al_1$}}
\put(89.2,10){\line(1,0){14}}
\put(105,10){\circle*{4.7}}
\put(105, 18){\makebox(0,0){$\al_2$}}
\put(107, 11.1){\line(1,0){14.3}}
\put(107, 9){\line(1,0){14.3}}
\put(112, 8.2){\scriptsize $>$}
\put(124,10){\circle{4.7}}
\put(124, 18){\makebox(0,0){$\al_3$}}
\put(126.2,10){\line(1,0){13.7}}
\put(142.2,10){\circle{4.7}}
\put(142, 18){\makebox(0,0){$\al_4$}}
\end{picture}$ & $\SU_{3}\times\SU_{2}\times\U_{1}$& $(22, 12, 4)$  \\
\hline
& $\E_7(3)$ & $\begin{picture}(160, 35)(-5,-10)
\put(15, 9.5){\circle{4.7 }}
\put(15, 18){\makebox(0,0){$\al_1$}}
\put(17, 10){\line(1,0){14}}
\put(33.5, 9.5){\circle{4.7 }}
\put(33.5, 18){\makebox(0,0){$\al_2$}}
\put(35.5, 10){\line(1,0){13}}
 \put(51, 9.5){\circle*{4.7 }}
 \put(51, 18){\makebox(0,0){$\al_3$}}
\put(68.9,-6){\line(0,1){13}}
\put(53,10){\line(1,0){13}}
\put(69,9.5){\circle{4.7 }}
\put(69.9, 18){\makebox(0,0){$\al_4$}}
\put(69,-8.2){\circle{4.7}}
\put(79, -10){\makebox(0,0){$\al_7$}}
\put(71.2,10){\line(1,0){15.4}}
\put(89,9.5){\circle{4.7 }}
\put(89, 18){\makebox(0,0){$\al_5$}}
\put(91,10){\line(1,0){15.8}}
\put(109,9.5){\circle{4.7 }}
\put(109, 18){\makebox(0,0){$\al_6$}}
\put(111, 10){\line(1,0){11.8}}
\put(126,10){\circle{6}}
\put(126,10){\circle{3}}
\end{picture}
$ & $\SU_{5}\times\SU_{3}\times\U_{1}$ & $(60, 30, 10)$    \\
\hline
& $\E_7(5)$ & $\begin{picture}(160, 35)(-5,-10)
\put(15, 9.5){\circle{4.7 }}
\put(15, 18){\makebox(0,0){$\al_1$}}
\put(17, 10){\line(1,0){14}}
\put(33.5, 9.5){\circle{4.7 }}
\put(33.5, 18){\makebox(0,0){$\al_2$}}
\put(35.5, 10){\line(1,0){13}}
 \put(51, 9.5){\circle{4.7 }}
 \put(51, 18){\makebox(0,0){$\al_3$}}
\put(68.9,-6){\line(0,1){13}}
\put(53,10){\line(1,0){13}}
\put(69,9.5){\circle{4.7 }}
\put(69.9, 18){\makebox(0,0){$\al_4$}}
\put(69,-8.2){\circle{4.7}}
\put(79, -10){\makebox(0,0){$\al_7$}}
\put(71.2,10){\line(1,0){15.4}}
\put(89,9.5){\circle*{4.7 }}
\put(89, 18){\makebox(0,0){$\al_5$}}
\put(91,10){\line(1,0){15.8}}
\put(109,9.5){\circle{4.7 }}
\put(109, 18){\makebox(0,0){$\al_6$}}
\put(111, 10){\line(1,0){11.8}}
\put(126,10){\circle{6}}
\put(126,10){\circle{3}}
\end{picture}
$ & $\SU_{6}\times\SU_{2}\times\U_{1}$ & $(60, 30, 4)$   \\
\hline
& $\E_8(2)$ & $\begin{picture}(160,40)(-10,-10)
\put(-3,9.5){\circle{6}}
\put(-3,9.5){\circle{3}}
\put(-0, 10){\line(1,0){12.9}}
\put(15, 9.5){\circle{4.7 }}
\put(15, 18){\makebox(0,0){$\al_1$}}
\put(17, 10){\line(1,0){14}}
\put(33.5, 9.5){\circle*{4.7 }}
\put(33.5, 18){\makebox(0,0){$\al_2$}}
\put(35.5, 10){\line(1,0){13.4}}
 \put(51, 9.5){\circle{4.7 }}
 \put(51, 18){\makebox(0,0){$\al_3$}}
\put(53,10){\line(1,0){13.5}}
\put(69,9.5){\circle{4.7 }}
\put(69, 18){\makebox(0,0){$\al_4$}}
\put(89,-8){\circle{4.7}}
\put(99, -9.5){\makebox(0,0){$\al_8$}}
\put(89,-6){\line(0,1){13}}
\put(71.2,10){\line(1,0){15.4}}
\put(89,9.5){\circle{4.7 }}
\put(89, 18){\makebox(0,0){$\al_5$}}
\put(91,10){\line(1,0){15.8}}
\put(109,9.5){\circle{4.7 }}
\put(109, 18){\makebox(0,0){$\al_6$}}
\put(111,10){\line(1,0){15.8}}
\put(129,9.5){\circle{4.7 }}
\put(129, 18){\makebox(0,0){$\al_7$}}
\end{picture}
$ & $\E_6\times\SU_{2}\times \U_{1}$ & $(108, 54, 4)$   \\
\thickline 
Type $III_{b}(3)$ & & &   \\
\hline
& $\E_6(3)$ & $\begin{picture}(150,55)(-10, -30)
\put(15, 9.5){\circle{4.7 }}
\put(15,17){\makebox(0,0){$\al_1$}}
\put(17, 10){\line(1,0){14}}
\put(33.5, 9.5){\circle{4.7 }}
\put(33.5,17){\makebox(0,0){$\al_2$}}
\put(35.6, 10){\line(1,0){13.6}}
 \put(51, 9.5){\circle*{4.7 }}
 \put(51,17){\makebox(0,0){$\al_3$}}
\put(51,-5.8){\line(0,1){13}}
\put(53,10){\line(1,0){13.8}}
\put(69,9.5){\circle{4.7 }}
\put(69, 17){\makebox(0,0){$\al_4$}}
\put(51,-8){\circle{4.7}}
\put(51, -22.6){\line(0,1){11.8}}
\put(51,-25.8){\circle{6}}
\put(51,-25.8){\circle{3}}
\put(60, -8.5){\makebox(0,0){$\al_6$}}
\put(71.2,10){\line(1,0){15.2}}
\put(89,9.5){\circle{4.7 }}
\put(89, 17){\makebox(0,0){$\al_5$}}
\end{picture}$ & $\SU_{3}\times\SU_{3}\times\SU_{2}\times\U_{1}$ & $(36, 18, 4)$   \\
\thickline 
\end{tabular} 
\end{center}}}
 
  
  \section{Left-invariant non-naturally reductive Einstein metrics on Lie groups of Type $I_{b}(3)$}\label{TypeI}

Let $G\cong G(i_{o})$ be a compact connected Lie groups of Type $I_{b}(3)$. Then  $G$ is isometric to $\G_2(2)$ or $\E_8(8)$. From now on we shall write  $\fr{g}=T_{e}G$ for  the corresponding Lie algebra.

\subsection{The Ricci tensor}
For a Lie group $G\cong G(i_{o})$ of Type $I_{b}(3)$  consider the  orthogonal decomposition 
\begin{equation}\label{I3}
\fr{g}=\fr{k}_{0}\oplus\fr{k}_{1}\oplus\fr{p}_{1}\oplus\fr{p}_{2}\oplus\fr{p}_{3}=\fr{m}_{0}\oplus\fr{m}_{1}\oplus\fr{m}_{2}\oplus\fr{m}_{3}\oplus\fr{m}_{4}.
\end{equation}
This is a reductive decomposition of $\fr{g}$ of the form  (\ref{go})  and due to (\ref{left}), a left-invariant metric  on $\G\cong G(i_{o})$ is given by
 \begin{equation}\label{invI3}
 \langle \ , \  \rangle=u_{0}\cdot B|_{\fr{k}_{0}}+u_{1}\cdot B|_{\fr{k}_{1}}+x_{1}\cdot B|_{\fr{p}_{1}}+x_{2}\cdot B|_{\fr{p}_{2}}+x_{3}\cdot B|_{\fr{p}_{3}}=y_{0}\cdot B|_{\fr{m}_{0}}+y_{1}\cdot B|_{\fr{m}_{1}}+y_{2}\cdot B|_{\fr{m}_{2}}+y_{3}\cdot B|_{\fr{m}_{3}}+y_{4}\cdot B|_{\fr{m}_{4}}
 \end{equation}
 for some positive numbers $u_{0}, u_{1}, x_{i}, y_{j}\in\bb{R}_{+}$. This metric is also $\Ad(K)$-invariant and since $\fr{m}_{i}\ncong\fr{m}_{j}$ for all $2\leq i\neq j\leq 4$, any  $G$-invariant metric on the base space $M=G/K$ is of the form
 \[
 ( \ , \ )=x_{1}\cdot B|_{\fr{p}_{1}}+x_{2}\cdot B|_{\fr{p}_{2}}+x_{3}\cdot B|_{\fr{p}_{3}}=y_{2}\cdot B|_{\fr{m}_{2}}+y_{3}\cdot B|_{\fr{m}_{3}}+y_{4}\cdot B|_{\fr{m}_{4}}.
 \]

 \begin{remark}\label{dimg2}
 \textnormal{For $\G_{2}(2)$ it is $\fr{k}_{0}\cong\fr{u}_{1}$, $\fr{k}_{1}\cong\fr{su}_{2}$ and
 $d_{0}=1$, $d_{1}=3$, $d_{2}=4$, $d_{3}=2$, $d_{4}=4$. For $\E_8(8)$ it  is $\fr{k}_{0}\cong\fr{u}_{1}$, $\fr{k}_{1}\cong\fr{su}_{8}$  and  $d_{0}=1$, $d_{1}=63$, $d_{2}=112$, $d_{3}=56$, $d_{4}=16$.}
 \end{remark}
 
\begin{prop}\label{G22}
   For the reductive decomposition (\ref{I3}) associated to  the compact simple Lie group  $\G_{2}(2)$  and for the left-invariant metric given by (\ref{invI3}), the non-zero structure constants $A_{ijk}$ $(0\leq i, j, k\leq 4)$ are the following (and their symmetries): $A_{022}$, $A_{033}$,  $A_{044}$, $A_{111}$,  $A_{122}$,  $A_{144}$, $A_{223}$,  $A_{234}$. This hold also for   $\E_{8}(8)$,  but in this case one has in addition $A_{133}\neq 0$.
\end{prop}
\begin{proof}
The proof is based on Lie theoretic arguments and it is similar with the one  which we shall present for  Proposition \ref{F42}. Since the latter case is a bit more complicated, we  state  here only a few details and we refer to this proof for an extensive description of the different techniques that can be applied.
  First notice that 
  \[
[\fr{m}_{2}, \fr{m}_{2}]\subset\fr{m}_{3}\oplus\fr{k}, \quad [\fr{m}_{2}, \fr{m}_{3}]\subset\fr{m}_{2}\oplus\fr{m}_{4}, \quad [\fr{m}_{2}, \fr{m}_{4}]\subset\fr{m}_{3}, \quad [\fr{m}_{3}, \fr{m}_{3}]\subset\fr{k}, \quad [\fr{m}_{3}, \fr{m}_{4}]\subset\fr{m}_{2}, \quad [\fr{m}_{4}, \fr{m}_{4}]\subset\fr{k}.
\]
These inclusions occur since the base space of the $K$-principal bundle $G\to M=G/K$ is a flag manifold with $\fr{p}=\fr{p}_{1}\oplus\fr{p}_{2}\oplus\fr{p}_{3}(=\fr{m}_{2}\oplus\fr{m}_{3}\oplus\fr{m}_{4})$ and $b_{2}(M)=1$  (see \cite[p.~1593]{stauros} and \cite[p.~674]{CS} for the general case). 
Notice also  that $[\fr{m}_{0}, \fr{m}_{j}]\subset\fr{m}_{j}$ for  $j=2, 3, 4$. Moreover,  for   $\G_2(2)$ one can show that  
\[
[\fr{m}_{1}, \fr{m}_{1}]\subset\fr{m}_{1}, \ \ [\fr{m}_{1}, \fr{m}_{2}]\subset\fr{m}_{2}, \ \ [\fr{m}_{1}, \fr{m}_{4}]\subset\fr{m}_{4},
\]
 but $[\fr{m}_{1}, \fr{m}_{3}]=0$. Hence  $A_{113}=0$.  Indeed,  the highest weight of $\fr{k}_{1}=\fr{su}_{2}=\fr{m}_{1}$ is $\al_1$ and the highest weight of $\fr{m}_{3}$ is $\al_1+2\al_2$.  By using the corresponding root vectors we see that  $[E_{\pm \al_{1}}, E_{\pm(\al_{1}+2\al_{2})}]=0$, because $\pm \al_{1}\pm (\al_{1}+2\al_{2})$ is not a root of $\G_2$. The highest weight of $\fr{m}_{4}$ is  the maximal root $\tilde{\al}$. Then $E_{\tilde{\al}-\al_1}=E_{\al_1+3\al_2}\in\fr{m}_{4}$ and we finally conclude that  $[\fr{m}_{1}, \fr{m}_{4}]\subset\fr{m}_{4}$.  For $\E_8(8)$ one gets that $[\fr{m}_{1}, \fr{m}_{j}]\subset\fr{m}_{j}$, for any $j=2, 3, 4$, so $A_{133}$ is non-zero in this case. Another   approach, appropriate for the examination of $A_{133}$, is based on the orthogonality of roots, see also Proposition \ref{F42}.
\end{proof}

By applying  now  Lemma \ref{ric} we get that
\begin{corol}\label{ricg2}
On $(\E_{8}(8), \langle \ , \ \rangle)$, the components ${r}_{i}$ of  the Ricci tensor $\Ric_{\langle \ , \ \rangle}$ associated to the left-invariant metric $\langle \ , \ \rangle$ given by (\ref{invI3}),  are described as follows
\begin{equation*}
\left\{\begin{array}{ll} 
r_0 &=  \displaystyle{\frac{u_{0}}{4d_0}\biggl(\frac{A_{022}}{{x_{1}}^{2}}+\frac{A_{033}}{{x_{2}}^{2}}+\frac{A_{044}}{{x_{3}}^{2}}
\biggr), }\quad r_1  =  \displaystyle{\frac{A_{111}}{4d_{1}}\cdot\frac{1}{u_{1}} + 
\frac{u_{1}}{4d_1} \biggl( \frac{A_{122}}{{x_{1}}^{2}} +\frac{A_{133}}{{x_{2}}^{2}} +\frac{A_{144}}{{x_{3}}^{2}}\biggr),}
  \\ & \\
r_2 &=  \displaystyle{\frac{1}{2 x_{1}} -
\frac{1}{2d_2}\cdot\frac{1}{{x_{1}}^{2}} 
\biggl(u_{0}\cdot A_{022}+u_{1}\cdot A_{122}+x_{2}\cdot A_{223} \biggr) 
 +  \frac{A_{234}}{2d_2}  
\biggl(\frac{x_1}{x_{2} x_{3}} - \frac{x_{2}}{x_1 x_{3}}- \frac{x_{3}}{x_1 x_{2}}
\biggr),  }
\\ & \\
r_3 & = \displaystyle{\frac{1}{2 x_{2}}   -\frac{1}{2d_{3}{x_{2}}^{2}}\bigg(u_{0}\cdot A_{033}+u_{1}\cdot A_{133}\biggr)
+\frac{A_{223}}{4d_{3}}\biggl( \frac{x_{2}}{{x_{1}}^{2}} - \frac{2}{x_{2}}\biggr)
  +\frac{A_{234}}{2d_{3}}
\biggl(\frac{x_{2}}{x_1 x_{3}} - \frac{x_1}{x_{2} x_{3}}- \frac{x_{3}}{x_1 x_{2}}
\biggr), }
\\ & \\
r_4 & = \displaystyle{\frac{1}{2 x_{3}} -  
\frac{1}{2d_4} \cdot\frac{1}{{x_{3}}^{2}}\biggl(u_{0}\cdot A_{044}+u_{1}\cdot A_{144}\biggr)+\frac{A_{234}}{2d_{4}}
\biggl(\frac{x_{3}}{x_1 x_{2}} - \frac{x_1}{x_{2} x_{3}}- \frac{x_{2}}{x_1 x_{3}}
\biggr). } 
\end{array}
\right.
\end{equation*}
The corresponding Ricci components  of $\G_{2}(2)$ occur by the same  expressions, by setting however $A_{133}=0$.
\end{corol}
\subsection{The structure constants}
 We proceed with the computation of the non-zero structure constants. Two of them, namely $A_{234}$ and $A_{223}$ can be    directly obtained  using the unique K\"ahler-Einstein metric  that any flag manifold $M=G/K$ with $b_{2}(M)=1$ and $\fr{p}=\fr{p}_{1}\oplus\fr{p}_{2}\oplus\fr{p}_{3}$ admits. By inserting the  values $x_{1}=1$, $x_{2}=2$, $x_{3}=3$ in the homogeneous Einstein equation  $\{\bar{r}_{2}-\bar{r}_{3}=0, \bar{r}_{3}-\bar{r}_{4}=0\}$ where $\bar{r}_{i}$ $(i=2, 3, 4)$ are the components of Ricci tensor $\Ric_{( \ , \ )}$ associated to the K\"ahler C-space $(M=G/K, ( \ , \ ))$ and solving this system with respect to $A_{223}, A_{234}$, one computes that (see   \cite{stauros})
\begin{equation}\label{A234}
 A_{223}=\displaystyle\frac{d_2 d_3 + 2 d_2 d_4 - d_3 d_4}{d_2 + 4 d_3 + 9 d_4}, \qquad A_{234}=\displaystyle\frac{(d_2 + d_3) d_4}{d_2 + 4 d_3 + 9 d_4}.
 \end{equation}
   For $\G_2(2)$ this gives $A_{223}=2/3$ and $A_{234}=1/2$, and for $\E_8(8)$ we get $A_{223}=56/3$, $A_{234}=28/5$. Now, by  Lemma \ref{ric} we also compute that   
    \begin{equation}\label{dimen}
  \G_{2}(2):  \left\{  \begin{array}{ll}
  d_{0}=1 & = \ A_{022}+A_{033}+A_{044},   \\
    d_{1}=3&= \ A_{111}+A_{122}+A_{144}, \\
    d_{2}=4&= \ 2(A_{022}+A_{122}+A_{223}+A_{234}), \\
    d_{3}=2&= \ 2A_{033}+A_{223}+2A_{234},  \\
    d_{4}=4&= \ 2(A_{044}+A_{144}+A_{234}).
    \end{array}
\right. \quad   \E_8(8): \left\{  \begin{array}{ll}
    1 & = \ A_{022}+A_{033}+A_{044},   \\
    63&= \ A_{111}+A_{122}+A_{133}+A_{144}, \\
    112&= \ 2(A_{022}+A_{122}+A_{223}+A_{234}), \\
    56&= \ 2A_{033}+2A_{133}+A_{223}+2A_{234},  \\
    16&= \ 2(A_{044}+A_{144}+A_{234}).
    \end{array}
\right.
    \end{equation}
  \begin{lemma}\label{gg2}
  For the reductive decomposition (\ref{I3}) and for the left-invariant metric $\langle \ , \ \rangle$ on $\G_2(2)$ given by (\ref{invI3}), the non-zero $A_{ijk}$ are given explicitly  as follows:
  \[
  A_{022}=1/12, \  \  A_{033}=1/6, \  \  A_{044}=A_{122}=A_{144}=3/4, \  \  A_{111}=3/2, \  \      A_{223}=2/3, \  \ A_{234}=1/2.
\]
For  $\E_8(8)$, the corresponding non-zero triples $A_{ijk}$ attain the values
  \[
  \begin{tabular}{lllll}
  $A_{022}=7/30$, &  $A_{044}=3/10$, &  $A_{111}=84/5$, &  $A_{122}=63/2$, & $A_{144}=21/10$, \\
  $A_{033}=7/15$, & $ A_{133}=63/5$,  & $A_{223}=56/3$, &  $A_{234}=28/5$.
  \end{tabular}
\]
  \end{lemma}
\begin{proof}
\noindent{\bf Case of $\G_2(2)$.}  The  fourth equation in (\ref{dimen})  implies that   $A_{033}=(1/2)(d_{3}-A_{223}-2A_{234})=1/6$.  Therefore,  one gets a  system with four equations and five unknowns, namely $A_{022}, A_{044}, A_{111}, A_{122}, A_{144}$.  Consider the Killing metric defined by $y_{i}=1$ for any $i=0, \ldots, 4$.  This is a  bi-invariant  Einstein metric on $\G_2$ (hence also left-invariant), and thus it satisfies  the system of equations $\{r_{0}-r_{1}=0, r_{1}-r_{2}=0, r_{2}-r_{3}=0, r_{3}-r_{4}=0\}$.  Solutions of this system are given by $A_{111}=3/2$ and   
\begin{equation}\label{A044}
 A_{044}=5/6-A_{022}, \quad A_{122}=1/96 (80 - 96 A_{022}), \quad A_{144} ={2/3 + A_{022}}.
 \end{equation}
  Notice that the first equation coincides with the first equation in (\ref{dimen}), after replacing the previous value of $A_{033}$,  the second equation with the third equation  in (\ref{dimen}) after inserting  the values of $A_{223}, A_{234}$ and finally the last one is the same with the second or the firth equation in (\ref{dimen}), after substituting the values of $A_{122}$ and $A_{044}$, respectively, given above.
Hence,   it suffices to compute some of the triples appearing in (\ref{A044}). 
 Set
\[
 \fr{g}=\fr{h}_{1}\oplus\fr{h}_{2}\oplus\fr{n}, \quad \fr{h}_{1}:=\fr{k}_{0}\oplus\fr{p}_{2}, \quad \fr{h}_{2}:=\fr{k}_{1}, \quad \fr{n}:=\fr{p}_{1}\oplus\fr{p}_{3}.
\]
The space $\fr{h}_{1}=\fr{k}_{0}\oplus\fr{p}_{2}\cong\fr{m}_{0}\oplus\fr{m}_{3}$ is a Lie subalgebra of $\fr{g}=\fr{g}_{2}$ isomorphic to $\fr{su}_{2}$ and the same is true for $\fr{h}_{2}=\fr{k}_{1}\cong\fr{m}_{1}$. Thus $\fr{h}:=\fr{h}_{1}\oplus\fr{h}_{2}$ is a Lie subalgebra of $\fr{g}$ isomorphic to $\fr{su}_{2}\oplus\fr{su}_{2}$, and the above decomposition is $\Ad(H)$-invariant.  Here, $H=\SO_{4}$ is the connected Lie group corresponding to $\fr{h}$.  In this way  we define a fibration 
  \[
  \bb{C}P^{1}=\SO_{4}/\U_{2}\to G/K=\G_{2}/\U_{2}  \to G/H=\G_{2}/\SO_{4},
  \]
   with  the base space being irreducible symmetric space. 
    Left-invariant metrics  on $\G_2$ are given now by $\langle\langle \ , \ \rangle\rangle=w_{1}\cdot B|_{\fr{h}_{1}}+w_{2}\cdot B|_{\fr{h}_{2}}+w_{3}\cdot B|_{\fr{n}}$, with $w_{1}, w_{2}, w_{3}\in\bb{R}_{+}$.  This metric is $\Ad(H)$-invariant, thus also $\Ad(K)$-invariant  and  for $w_{1}=u_{0}=x_{2}$, $w_{2}=u_{1}$ and $w_{3}=x_{1}=x_{3}$ it coincides with the left-invariant metric $\langle \ , \ \rangle$ described before.  For these values, the same holds for the corresponding  Ricci tensors $\Ric_{\langle \ , \ \rangle}$ and $\Ric_{\langle\langle \ , \ \rangle\rangle}$. In particular,  let   us denote by $\bold{\tilde{r}}_{1}, \bold{\tilde{r}}_{2}, \bold{\tilde{r}}_{3}$ the components of  $\Ric_{\langle\langle \ , \ \rangle\rangle}$ with respect to the new left-invariant metric  $\langle\langle \ , \ \rangle\rangle$. Then,  for $u_{0}=w_{1}$, $x_{2}=w_{1}$, $u_{1}=w_{2}$, $x_{1}=w_{3}$ and $x_{3}=w_{3}$, it holds that
  \[
 \bold{\tilde{r}}_{1}=r_{0}=r_{3}, \quad \bold{\tilde{r}}_{2}=r_{1}, \quad \bold{\tilde{r}}_{3}=r_{2}=r_{4}.
 \]
 By using the  relation $r_{2}=r_{4}$ we  see that $A_{022} +A_{223} = A_{044}$ and thus  we get $A_{022}=1/12$.  Using relations   (\ref{A044})  we easily conclude.

  \noindent{\bf Case of $\E_8(8)$.}      By Proposition \ref{G22} for $\E_{8}(8)$ it is  $A_{133}\neq 0$. Therefore,  by (\ref{dimen}) and  (\ref{A234}) we cannot immediately compute $A_{033}$ (as we did for $\G_{2}(2)$).  We need to construct more equations. Set
        \[
    \fr{e}_{8}=\fr{g}=\fr{u}\oplus\fr{r}, \quad \fr{u}:=\fr{k}_{0}\oplus\fr{k}_{1}\oplus\fr{p}_{3}=\fr{k}\oplus\fr{p}_{3}, \quad \fr{r}:=\fr{p}_{1}\oplus\fr{p}_{2}
    \]
    Then,  the following inclusions hold $[\fr{u}, \fr{u}]\subset\fr{u}$, $[\fr{u}, \fr{r}]\subset\fr{r}$, i.e. $\fr{g}=\fr{u}\oplus\fr{r}$ is a reductive decomposition of the homogeneous space $G/U$,  where $U$ is the connected Lie subgroup generated by the Lie  algebra $\fr{u}$. Since $\fr{u}\subset \fr{g}$ we get the fibration $U/K\to G/K\to G/U$ where $G/K=\E_8/(\U_1\times\SU_{8})$ and $\fr{u}\cong\fr{su}(9)$. In full details
    \[
    \bb{C}P^{8}=\SU_{9}/\U_8\to \E_8/(\U_{1}\times\SU_8)\to \E_8/\SU_{9}
    \]
   where the base space is (strongly) isotropy irreducible but not a symmetric space.  Consider now left-invariant metrics on $\E_8$ given by $\langle\langle \ , \ \rangle\rangle= z_{1}\cdot B|_{\fr{u}}+z_{2}\cdot B|_{\fr{r}}$ with $z_{1}, z_{2}\in\bb{R}_{+}$.  This is an $\Ad(U)$-invariant metric and for $z_{1}=y_{0}=y_{1}=x_{3}$, $z_{2}=x_{1}=x_{2}$  coincides with the left-invariant metric $\langle \ , \ \rangle$, defined   by (\ref{invI3}).  For these values the associated Ricci components of  $(\E_{8}(8), \langle \ , \ \rangle)$ are such that
      $r_{0}=r_{1}=r_{4}$ and $r_{2}=r_{3}$.
       Hence, we get for example $63 A_{022} + 63 A_{033} - A_{122} - A_{133}=-441 + 10 A_{122} + 10 A_{133}$, $63 A_{044} - A_{111} - A_{144}=-1512 + 315 A_{044} + 40 A_{111}+ 355 A_{144}$ and $28 - 5 A_{022} + 10 A_{033} - 5 A_{122} + 10 A_{133}=0$. 
    After solving this system of equations, together with the equations obtained by the  Killing metric, we get that 
  \begin{equation}\label{e888}
 \begin{tabular}{l|l|l}
    $A_{044} = 3/10$, & $ A_{144} = 21/10$, & $ A_{122} = 476/15 - A_{022}$, \\ 
    $A_{111} = 84/5$, & $A033 = 7/10 - A_{022}$, & $A_{133} = 371/30 + A_{022}.$ 
 \end{tabular} 
 \end{equation}
 Now it is sufficient  to compute $A_{022}$.  Set    
        \[
    \fr{g}=\fr{h}\oplus\fr{n}, \quad \fr{h}:=\fr{k}_{0}\oplus\fr{k}_{1}\oplus\fr{p}_{2},\quad \fr{n}:=\fr{p}_{1}\oplus\fr{p}_{3}.
    \]
    In this case  one can easily prove that  $[\fr{h}, \fr{h}]\subset\fr{h}$, $[\fr{h}, \fr{n}]\subset\fr{n}$ and $[\fr{n}, \fr{n}]\subset\fr{h}$. For dimensional reasons it  is $\fr{h}\cong\fr{so}(16)$, or in other words, we get a fibration  $H/K\to G/K\to G/H$ with $\fr{n}\cong T_{o'}G/H$ $(o'=eH\in G/H)$. Both  the fiber $H/K$ and the base space $G/K$  are 
   irreducible symmetric spaces, in particular the fibration 
            \[
    \SO_{16}/\U_{8}\to \E_8/(\U_1\times\SU_8)\to \E_8/\SO_{16}. 
    \]
    is the twistor fibration of the  flag manifold $\E_8/\U_{8}\cong\E_8/(\U_1\times\SU_{8})$ over the symmetric space $\E_8/\SO_{16}$.
Consider new left-invariant metrics  on $\E_8$ related to the decomposition   $\fr{g}=\fr{h}\oplus\fr{n}$, i.e. $\langle\langle \ , \ \rangle\rangle'=v_{1}\cdot B|_{\fr{h}}+v_{2}\cdot B|_{\fr{n}}$ for some $v_{1}, v_{2}\in\bb{R}_{+}$.   For $v_1=u_{0}=u_{1}=x_{2}$ and $v_{2}=x_{1}=x_{3}$ this metric  coincides with the left-invariant metric $\langle \ , \  \rangle$.  In a similar way, we get the relations $r_{0}=r_{1}=r_{3}$ and $r_{2}=r_{4}$,   which imply now $63 A_{022} + 63 A_{044} - A_{122} - A_{144}=0$ and $3 A_{033} - A_{111} - A_{133}=0$. By combining  with (\ref{e888})  we conclude. 
   \end{proof}   
 \begin{remark}\label{sakidea}
 \textnormal{An alternative way to compute $A_{111}$   is given as follows. For $G\in\{\G_{2}(2), \E_8(8)\}$ consider the decomposition (\ref{I3}).  Let $K_{1}$  be the  connected  Lie (sub)group generated by the simple ideal $\fr{k}_{1}$. Since $K_{1}\subset G$, there exists a positive number $c>0$ such that $B_{K_{1}}=c\cdot B_{G}$, where $B_{G}$ denotes the negative of the Killing form of $G$. Let 
 \[
 a_{111}:=\widetilde{\displaystyle{1 \brack {11}}}
 \] the triple associated to $K_1$ with respect to $B_{K_{1}}$ (as a compact simple Lie group). Then, by Lemma \ref{ric} we get $a_{111}=\dim{\fr{k}_{1}}=:d_{1}$. On the other hand, by the definition of $A_{111}$ it is easy to see that 
 $A_{111}=c\cdot a_{111}$ (for Lie groups with roots of the same length). For  $\G_2$  one has to notice that $\fr{k}_{1}\cong\fr{su}(2)$ is generated by the long root $\al_1$, in particular $B_{\G_{2}}(\al_1, \al_1)=3B_{\G_{2}}(\al_2, \al_2)$.   Because $c=B_{\SU_{2}}/B_{\G_{2}}=4/24$,   it follows that  $A_{111}=3(c\cdot \dim\fr{su}_{2})=3/2$. For $\E_8(8)$  it is $\fr{k}_{1}\cong\fr{su}(8)$ and all the roots have the same length. In particular,   $c=B_{\SU_{8}}/B_{\E_8}=16/60$, hence $A_{111}=c\cdot\dim\fr{su}_{8}=84/5$.}
 \end{remark}

\subsection{Naturally reductive metrics} 
For a Lie group $G\cong G(i_{o})$ of Type $I_{b}(3)$, 
 left-invariant metrics  on $\G\cong G(i_{o})$  which are $\Ad(K)$-inavariant are given by $(\ref{invI3})$.
 
 \begin{prop}\label{prop4.6}
If a left invariant metric $\langle \ , \ \rangle$ of the form $(\ref{invI3})$  on $G\cong G(i_{o})$ of Type $I_{b}(3)$  is naturally reductive  with respect to $G \times L$ for some closed subgroup $L$ of $G$, 
then one of the following holds: 

$(1)$ for $G =\G_2(2)$,  $u_0 = x_2 $,   $x_{1} = x_{3}$, and  for $G =\E_8(8)$, $u_0 = u_1 = x_2 $,   $x_{1} = x_{3}$ \quad 
$(2)$   $u_0  = u_1 =x_3$,   $x_{1} = x_{2}$  
\quad 
$(3)$ $ x_{1} = x_{2} = x_{3} $.

Conversely, 
 if  one of the conditions $(1)$, $(2)$, $(3)$  holds, then the metric 
 $\langle \ , \ \rangle$ of the form  $(\ref{invI3})$ is  naturally reductive  with respect to $G \times L$ for some closed subgroup $L$ of $G$.
  \end{prop}
  \begin{proof}
   Let ${\frak l}$ be the Lie algebra of  $L$. Then we have either ${\frak l} \subset {\frak k}$  or ${\frak l} \not\subset {\frak k}$. 
First we consider the case of  ${\frak l} \not\subset {\frak k}$. Let ${\frak h}$ be the subalgebra of ${\frak g}$ generated by ${\frak l}$ and ${\frak k}$. 
Since 
$ \fr{g}=\fr{k}_{0}\oplus\fr{k}_{1}\oplus\fr{p}_{1}\oplus\fr{p}_{2}\oplus\fr{p}_{3}$ is an irreducible decomposition as $\mbox{Ad}(K)$-modules, we see that the Lie algebra $\frak h$  contains  at least one of  ${\frak p}_{1}$,  ${\frak p}_{2}$, ${\frak p}_{3}$. 
We first consider the case that $\frak h$  contains ${\frak p}_{1}$. 
 Since 
$\left[{\frak p}_{1}, {\frak p}_{1}\right] \cap {\frak p}_2 \neq \{0\}$, the space $\frak h$  contains ${\frak p}_{2}$.  
Notice  also that  $\left[{\frak p}_{1}, {\frak p}_{2}\right] \cap {\frak p}_3 \neq \{0\}$, hence  $\frak h$ contains ${\frak p}_{3}$. Thus we see that  $\frak h = \frak g $ and  the $\Ad(L)$-invariant metric $\langle  \ ,  \ \rangle$ of the form $(\ref{invI3})$ is bi-invariant. 
Now, if $\frak h$  contains ${\frak p}_{2}$, then $\fr{ h} \supset {\frak k}\oplus {\frak p}_{2}$. If $\fr{ h} =\fr{k}\oplus {\frak p}_{2}$, then $(\fr{ h}, \fr{p}_{1}\oplus\fr{p}_{3})$ is a symmetric pair. Thus the metric $\langle  \ , \   \rangle$ of the form $(\ref{invI3})$ satisfies $u_0 = x_2 $,   $x_{1} = x_{3}$ for $G =\G_2(2)$ and   $u_0 = u_1 = x_2 $,   $x_{1} = x_{3}$  for $G =\E_8(8)$. If $\fr{ h} \neq\fr{k}\oplus {\frak p}_{2}$, we see that $\fr{h}\cap {\frak p}_1\neq\{0\}$ or $\fr{h}\cap {\frak p}_3\neq\{0\}$ and thus $\fr{ h} \supset {\frak p}_1$ or $\fr{ h} \supset {\frak p}_3$. Hence, we obtain $\fr{ h} = \fr{g}$ and the $\Ad(L)$-invariant metric $\langle \ ,\   \rangle$ of the form $(\ref{invI3})$ is bi-invariant. 
Next if $\frak h$  contains ${\frak p}_{3}$, then $\fr{ h} \supset {\frak k}\oplus {\frak p}_{3}$. If $\fr{ h} =\fr{k}\oplus {\frak p}_{3}$, then $\fr{ h}$ is a simple Lie algebra, in fact, for $G =\G_2(2)$ we see that $\fr{ h} = \fr{su}_{3}$ and for $G=\E_8(8)$ we see that $\fr{h} = \fr{su}_{9}$,   and ${\frak p}_{1}\oplus {\frak p}_{2}$ is an irreducible $\Ad(H)$-module. 
Thus the metric $\langle  \ ,\   \rangle$ of the form $(\ref{invI3})$ satisfies $u_0 = u_1 =  x_3 $,   $x_{1} = x_{2}$. 
If $\fr{ h} \neq\fr{k}\oplus {\frak p}_{3}$, we see that $\fr{ h}\cap {\frak p}_1\neq\{0\}$ or $\fr{h}\cap {\frak p}_2\neq\{0\}$ and thus $\fr{h} \supset {\frak p}_1$ or $\fr{ h} \supset {\frak p}_2$. Hence, we obtain $\fr{ h} = \fr{g}$ and the $\Ad(L)$-invariant metric $\langle  \ ,\   \rangle$ of the form $(\ref{invI3})$ is bi-invariant. 

 We proceed with  the case ${\frak l} \subset {\frak k}$.  Because the  orthogonal complement
 ${\frak l}^{\bot}$ of ${\frak l}$ with respect to $B$ contains the  orthogonal complement 
${\frak k}^{\bot}$ of ${\frak k}$, it follows  that ${\frak l}^{\bot} \supset {\frak p}_{1} \oplus  {\frak p}_{2}\oplus  {\frak p}_{3}$.   
Since the  invariant metric $\langle \ , \  \rangle$ is naturally reductive  with respect to $G\times L$,  
we conclude that  $ x_{1} = x_{2} = x_{3} $ by  Theorem \ref{ziller}.   

Conversely, 
 if the conditions $(1)$  holds, then   Theorem \ref{ziller} sates that the metric 
 $\langle \ ,   \rangle$ given by  $(\ref{invI3})$ is  naturally reductive  with respect to $G \times L$, where $\fr{l} = \fr{k}\oplus {\frak p}_{2}$.  If the condition  $(2)$  holds, then the metric $\langle \ , \ \rangle$ given by $(\ref{invI3})$ is  naturally reductive  with respect to $G \times L$ where $\fr{l} = \fr{k}\oplus {\frak p}_{3}$.  Finally, if the condition  $(3)$  holds, then the metric  defined by $(\ref{invI3})$ is  naturally reductive  with respect to $G \times K$. 
  \end{proof}

    \subsection{Einstein metrics}
 Due to  Lemma \ref{gg2}  and Remark \ref{dimg2},  Corollary \ref{ricg2} determines now explicitly   the Ricci tensor $\Ric_{\langle \ , \ \rangle}$ of the Lie groups $(\G_{2}(2), \langle \ , \ \rangle)$ and  $(\E_{8}(8), \langle \ , \ \rangle)$.  Recall that the homogeneous Einstein equation for the left-invariant metric $\langle \ , \ \rangle$  is given by
    \[
    \{ r_{0}-r_{1}=0,  \ r_{1}-r_{2}=0,\  r_{2}-r_{3}=0, \ r_{3}-r_{4}=0\}.
    \]
    
\noindent{\bf \underline{Case of $\G_2(2)$}}
    
    We normalize the  metric  by setting $x_3=1$. Then, we see that the homogeneous Einstein equation is equivalent to the following system of equations: 
\begin{equation}\label{eing2}
\left\{  \begin{array}{lll}
  f_1 &=&9 {u_0} {u_1} {x_1}^2 {x_2}^2+2
   {u_0} {u_1} {x_1}^2+{u_0} {u_1}
   {x_2}^2-3 {u_1}^2 {x_1}^2 {x_2}^2-3
   {u_1}^2 {x_2}^2-6 {x_1}^2
   {x_2}^2=0,\\ & & \\
  f_2&=&{u_0} {u_1} {x_2}+6 {u_1}^2
   {x_1}^2 {x_2}+15 {u_1}^2 {x_2}-6
   {u_1} {x_1}^3+6 {u_1} {x_1}
   {x_2}^2-48 {u_1} {x_1} {x_2}+6
   {u_1} {x_1}\\ 
   & & +8 {u_1} {x_2}^2+12
   {x_1}^2 {x_2}= 0, \\ \\
  f_3&=&4 {u_0} {x_1}^2-{u_0}
   {x_2}^2-9 {u_1} {x_2}^2+18 {x_1}^3
   {x_2}-32 {x_1}^2 {x_2}-18 {x_1}
   {x_2}^3+48 {x_1} {x_2}^2\\ 
   & &+6 {x_1}
   {x_2}-16 {x_2}^3 = 0, \\ \\
  f_4&=&9 {u_0} {x_1}^2
   {x_2}^2-4 {u_0} {x_1}^2+9 {u_1}
   {x_1}^2 {x_2}^2-6 {x_1}^3 {x_2}-48
   {x_1}^2 {x_2}^2+32 {x_1}^2 {x_2}\\ 
   & &+18
   {x_1} {x_2}^3-18 {x_1} {x_2}+8
   {x_2}^3= 0.
  \end{array} \right.
 \end{equation}
Consider the polynomial ring $R= {\mathbb Q}[z,  u_0,  u_1,  x_1, x_2] $ and the ideal $I$, generated by polynomials $\{ \,z \,  u_0 \,  u_1 \, x_1 \, x_2-1,$  $f_1, \, f_2, \, f_3, \, f_4\}$.   We take a lexicographic ordering $>$,  with $ z >  u_0 >  u_1 > x_2 > x_1$ for a monomial ordering on $R$. Then, by the aid of computer, we see that a  Gr\"obner basis for the ideal $I$ contains a  polynomial  of   $x_1$  given by 
$( x_1-1 ) ( 9 x_1- 11 ) h_1(x_1)$, where
$h_1(x_1)$ is a polynomial  of degree 39 of the form
{\small 
\begin{equation*} \begin{array}{l} 
h_1(x_1)=578531204393508729 {x_1}^{39}-2907419178698478477
   {x_1}^{38}+9728488450924774839
   {x_1}^{37}\\
   -25248552448377295323
   {x_1}^{36} 
   +56466222208555751172
   {x_1}^{35}-112561770533625695268
   {x_1}^{34}\\
   +198560662542278445420
   {x_1}^{33}-303749957092666314564
   {x_1}^{32}+403108239570614919684
   {x_1}^{31}\\-461651040693288248940
   {x_1}^{30}+457287888650310982692
   {x_1}^{29}-385901524001421918252
   {x_1}^{28}\\+263698285244084996724
   {x_1}^{27}-121027010977188296460
   {x_1}^{26}-3723246665533789740
   {x_1}^{25}\\+82231517231748069876
   {x_1}^{24}-102191761959692380074
   {x_1}^{23}+82264683344411071386
   {x_1}^{22} \\
   -47064365277272324622
   {x_1}^{21}+17165576727452434302
   {x_1}^{20}-1543691819299617324
   {x_1}^{19} \\
   -4724193236441019084
   {x_1}^{18}+4420128602614576596
   {x_1}^{17}-3155137453029513948
   {x_1}^{16}\\ +1592256867663356196 {x_1}^{15}-768600370620963068
   {x_1}^{14}+294461889742168084 {x_1}^{13} \\
   -112323449859022284
   {x_1}^{12}+34778832154783148 {x_1}^{11}-10968398600557556
   {x_1}^{10}\\ +2775906314750316 {x_1}^9-734820705350612
   {x_1}^8+149714619756553 {x_1}^7 \\
   -33449638312869
   {x_1}^6+5237027887391 {x_1}^5-997188227243
   {x_1}^4 \\ +107164942344 {x_1}^3-17880627984
   {x_1}^2+968584320 {x_1}-150264576. 
\end{array}
\end{equation*}
}
We also remark  that in the Gr\"obner basis, $u_0,  u_1,   x_2$ are given by polynomials of  degree 40  of $x_1$ with coefficients of rational numbers. 
Solving $h_1(x_1)=0$ numerically, we get only one solution, which is given approximately by $x_1 \approx 0.93245951$. Further, we see that a solution of the system of equations $\{f_1=0, f_2=0, f_3=0, f_4=0, h_1(x_1)=0 \}$ has the form by 
$$\{u_0 \approx 1.0851961, \ u_1 \approx 0.69929486,  \ x_1 \approx 0.93245951, \ x_2 \approx 1.0225069 \}. $$
Due to Proposition  \ref{prop4.6},  we conclude that this solution induces  a  non-naturally reductive Einstein metric.

For $ x_1= 11/9$, the system $\{f_1=0, f_2=0, f_3=0, f_4=0 \}$ has a solution,  given by 
$$\{ u_0 = 1, \ u_1 = 1,  \ x_1 = 11/9, \ x_2 =11/9 \}.$$ 
For $ x_1= 1$, we get  $u_0 = x_2$, $({x_2}-1) (875 {x_2}^3-1165 {x_2}^2+250 {x_2}-14) =0$ and $ {u_1}=(1750 {x_2}^3-4080 {x_2}^2+2585 {x_2}-192)/63 $. Thus,  we get solutions of the system of equations $\{f_1=0, f_2=0, f_3=0, f_4=0 \}$, namely \begin{equation*} 
 \begin{array}{l}
\{ u_0 \approx 0.095267235, \ u_1 \approx 0.29761039,  \ x_1 =1, \ x_2 \approx 0.095267235 \}, \\ 
\{ u_0 \approx 0.15539816, \ u_1 \approx 1.8689705,  \ x_1 =1, \ x_2 \approx 0.15539816 \}, \\ 
\{ u_0 \approx 1.0807632, \ u_1 \approx 0.71913340,  \ x_1 =1, \ x_2 \approx 1.0807632 \} 
\end{array}
\end{equation*} 
and $ u_0 = u_1 = x_1 = x_2 = 1$. 
By Proposition \ref{prop4.6}, one can  deduce that these values give  rise to naturally reductive Einstein metrics.

\medskip

\noindent{\bf \underline{Case of $\E_8(8)$}}

For a normalization of the metric  $x_3=1$,   the homogeneous Einstein equation is equivalent to the following system of equations: 
\begin{equation}\label{eing22}
 \left\{ \begin{tabular}{l}
 $g_1=14 u_0 u_1 {x_1}^2 - 6 {u_1}^2 {x_1}^2 + 7 u_0 u_1 {x_2}^2 - 15 {u_1}^2 {x_2}^2 - 
 8 {x_1}^2 {x_2}^2 + 9 u_0 u_1 {x_1}^2 {x_2}^2 - {u_1}^2 {x_1}^2 {x_2}^2 =0$,\\ \\  
  $g_2=48 {u_1}^2 {x_1}^2 + 24 u_1 x_1 x_2 - 24 u_1 {x_1}^3 x_2 + u_0 u_1 {x_2}^2 + 
 255 {u_1}^2 {x_2}^2 - 480 u_1 x_1 {x_2}^2 + 64 {x_1}^2 {x_2}^2 $\\ 
 $+$ $8 {u_1}^2 {x_1}^2 {x_2}^2 
  +  80 u_1 {x_2}^3   + 24 u_1 x_1 {x_2}^3  =0$, \\  \\  
   $g_3=14 u_0 {x_1}^2 + 108 u_1 {x_1}^2 + 24 x_1 x_2 - 320 {x_1}^2 x_2 + 72 {x_1}^3 x_2 - 
 u_0 {x_2}^2 - 135 u_1 {x_2}^2 
  + 480 x_1 {x_2}^2$\\ $-$   $ 160 {x_2}^3 - 72 x_1 {x_2}^3 =0$,\\   \\  
   $g_4=-4 u_0 {x_1}^2 - 108 u_1 {x_1}^2 - 216 x_1 x_2 + 320 {x_1}^2 x_2 + 120 {x_1}^3 x_2 - 
  480 {x_1}^2 {x_2}^2 + 9 u_0 {x_1}^2 {x_2}^2  $ \\ $ + 63 u_1 {x_1}^2 {x_2}^2  
 + 80 {x_2}^3  + 
 216 x_1 {x_2}^3 = 0$.
  \end{tabular} \right.
 \end{equation}
 We consider the polynomial ring $R= {\mathbb Q}[z,  u_0,  u_1, x_1,  x_2] $ and an ideal $I$, generated by polynomials $\{ \,z \,  u_0 \, u_1 \, x_1 \, x_2-1,$  $g_1, \, g_2, \, g_3, \, g_4\}$.   We choose the lexicographic ordering $>$,  with $ z >  u_0 >  u_1 > x_2 > x_1$ for a monomial ordering on $R$. Then,  a  Gr\"obner basis for the ideal $I$ contains a  polynomial  of   $x_1$,  given by 
$({x_1}-1) (9 {x_1}-41) k_1(x_1)$, where
$k_1(x_1)$ is a polynomial  of degree 49 explicitly defined as follows:

{ \small 
 \begin{equation*} 
\begin{array}{l} k_1(x_1)=78627620134518984869670299619225
   {x_1}^{49}-1879012156849489779707699697741525
   {x_1}^{48} \\ 
   +24958517859683233851773742581453130
   {x_1}^{47}-226345877119100379348478414653942870
   {x_1}^{46} \\
   +1513185548162200779194248616933419791
   {x_1}^{45}-7767087422550952023317555159123235339
   {x_1}^{44} \\
   +30802377992721718905775487881405527156
   {x_1}^{43}-91469740820715969757364654368958734704
   {x_1}^{42} \\
   +190873255583958103333447847157388070763
   {x_1}^{41}-251450263614444457947297352429222170207
   {x_1}^{40} \\
   +174697300086572957060190519165619552914
   {x_1}^{39}-348342277415980322916046484535937167846
   {x_1}^{38}\\ 
   +3002677791406351946122412362593386235357
   {x_1}^{37}-14163273159456167553892516550563377195873
   {x_1}^{36} \\+44027812533760332962711411379822436468296
   {x_1}^{35}-103359433218793762175332925928788224767564
   {x_1}^{34}\\ +195341684679658672333566251584926273888618
   {x_1}^{33}-305149823721731252899197653136452852072562
   {x_1}^{32}\\ +393503106465236359948116939782270035587444
   {x_1}^{31}-405394724912565051721521587945499542805996
   {x_1}^{30}\\ +272987735753676571252110160502472211820742
   {x_1}^{29}+67527266567097795889985999048107668548322
   {x_1}^{28}\\ -648798451413297437284550685137354441143464
   {x_1}^{27}+1406965042656066428973970557643766067268176
   {x_1}^{26}\\ -2175743101927487127818021138090846891067546
   {x_1}^{25}+2652679256280234163460510381340166807164594
   {x_1}^{24}\\ -2558905970930356615952062891716100348759020
   {x_1}^{23}+1819840234201169247432197581725018033034180
   {x_1}^{22}\\ -573110366219719484295636119428516440072790
   {x_1}^{21}-898192008904623391867803363166579084394930
   {x_1}^{20}\\ +2196777571676601155406190148800336652452480
   {x_1}^{19}-3047573028495615879932513597313365789288840
   {x_1}^{18}\\ +3388852959467640315510569932345608419421165
   {x_1}^{17}-3273630300427318361235238941319922925655545
   {x_1}^{16}\\ +2823340833679036161326179394502997732627650
   {x_1}^{15}-2191869205433840326556829886605796355495550
   {x_1}^{14}\\ +1538080739419688790208901918515707596638875
   {x_1}^{13}-976760862170802820093764370749377583728775
   {x_1}^{12}\\ +563675702099787379912980508814827349226900
   {x_1}^{11}-295654068049591813276291328141962741142400
   {x_1}^{10}\\ +141495683505513028215917983599798929865375
   {x_1}^9-61509598161488435381920159529683875449475
   {x_1}^8\\ +24312853879710780673048343979366706124250
   {x_1}^7-8664729466196887937652173231674074174750
   {x_1}^6\\ +2770234999000558646762022239898046895625
   {x_1}^5-786852533021232630512967801167059768125
   {x_1}^4\\ +194733271361911111513857132379687575000
   {x_1}^3-41207441950449158811069657833488687500
   {x_1}^2\\ +6741277041530521149993243000562500000
   {x_1}-859782169024171126981444722656250000. 
    \end{array}
    \end{equation*}
}

Notice  that  in the Gr\"obner basis, $u_0,  u_1,   x_2$ are given by polynomials of  degree 50  of $x_1$ with coefficients of rational numbers. Solving $k_1(x_1)=0$ numerically, we get three positive  and two negative solutions,  which are given approximately by $x_1 \approx 0.46131382, x_1 \approx 0.91172474, x_1 \approx 4.0130840$ and  $x_1 \approx  -1.2146356, x_1 \approx -1.1542138$.  Moreover,  solutions of the system   $\{g_1=0, g_2=0,  g_3=0, g_4=0 \}$ have the approximate form
\begin{equation*} 
 \begin{array}{l}
 \{ u_0 \approx 1.0767925, \ u_1 \approx 0.12842350,  \ x_1 \approx 0.46131382, \ x_2 \approx 0.73659849 \},\\
\{ u_0 \approx 0.77844700, \ u_1 \approx 0 .17409566,  \ x_1 \approx 0.91172474, \ x_2 \approx 0.52532563 \}, \\
\{ u_0 \approx 1.1022316, \ u_1 \approx 0.85179391,  \ x_1 \approx 4.0130840, \ x_2 \approx 4.0222155 \}   
\end{array}
\end{equation*} 
and 
\begin{equation*} 
 \begin{array}{l}
\{ u_0 \approx -1.3411877, \ u_1 \approx -0.75642675,  \ x_1 \approx -1.2146356, \ x_2 \approx -4.9166783 \},
\\ 
\{ u_0 \approx -1.4503818, \ u_1 \approx -0.54000582,  \ x_1 \approx -1.1542138, \ x_2 \approx -4.7370866 \}.
\end{array}
\end{equation*} 
Thus, we obtain three Einstein metrics which  are non-naturally reductive, by Proposition \ref{prop4.6}. 
After  computing the related scale invariants (cf.  \cite[Section 7]{Chry2}) we deduce  that these metrics are non-isometric each other. 

For $ x_1= 9/41$, the system of equation  $\{f_1=0, f_2=0, f_3=0, f_4=0 \}$  has a solution, given by 
$$\{ u_0 = 1, \ u_1 = 1,  \ x_1 = 9/41, \ x_2 = 9/41 \}.$$ 
For $ x_1= 1$, we see that $ u_0 =  u_1 = x_2 = 7/23$ and $ u_0 = u_1 = x_1 = x_2 = 1$. 
Due to Proposition \ref{prop4.6}, these solutions define naturally reductive Einstein metrics.

   \section{Left-invariant non-naturally reductive Einstein metrics on Lie groups of Type $II_{b}(3)$}\label{TypeII}

Let $G\cong G(i_{o})$ be a compact connected Lie groups of Type $II_{b}(3)$. Then  $G$ is isometric to $\F_4(2)$, $\E_7(3)$, $\E_8(2)$ or $\E_7(5)$.  Let $\fr{g}=T_{e}G$ be the corresponding Lie algebra.  

\subsection{The Ricci tensor}
For a Lie group $G\cong G(i_{o})$ of Type $II_{b}(3)$  consider the  orthogonal decomposition 
\begin{equation}\label{II3}
\fr{g}=\fr{k}_{0}\oplus\fr{k}_{1}\oplus\fr{k}_{2}\oplus\fr{p}_{1}\oplus\fr{p}_{2}\oplus\fr{p}_{3}=\fr{m}_{0}\oplus\fr{m}_{1}\oplus\fr{m}_{2}\oplus\fr{m}_{3}\oplus\fr{m}_{4}\oplus\fr{m}_{5}.
\end{equation}
This is a reductive decomposition of $\fr{g}$ of the form  (\ref{go})  and a left-invariant metric  on $\G\cong G(i_{o})$ is given by
 \begin{eqnarray}
  \langle \ , \  \rangle&=&u_{0}\cdot B|_{\fr{k}_{0}}+u_{1}\cdot B|_{\fr{k}_{1}}+u_{2}\cdot B|_{\fr{k}_{2}}+x_{1}\cdot B|_{\fr{p}_{1}}+x_{2}\cdot B|_{\fr{p}_{2}}+x_{3}\cdot B|_{\fr{p}_{3}}\nonumber\\&=&y_{0}\cdot B|_{\fr{m}_{0}}+y_{1}\cdot B|_{\fr{m}_{1}}+y_{2}\cdot B|_{\fr{m}_{2}}+y_{3}\cdot B|_{\fr{m}_{3}}+y_{4}\cdot B|_{\fr{m}_{4}}+y_{5}\cdot B|_{\fr{m}_{5}}\label{invII3}
 \end{eqnarray}
 for some positive numbers $u_{p}, x_{i}, y_{j}\in\bb{R}_{+}$.  Thus, a $G$-invariant metric on   $M=G/K$ is of the form
 \[
 ( \ , \ )=x_{1}\cdot B|_{\fr{p}_{1}}+x_{2}\cdot B|_{\fr{p}_{2}}+x_{3}\cdot B|_{\fr{p}_{3}}=y_{3}\cdot B|_{\fr{m}_{3}}+y_{4}\cdot B|_{\fr{m}_{4}}+y_{5}\cdot B|_{\fr{m}_{5}}.
 \]
   For a  Lie group $G\cong G(i_{o})$ of Type $II_{b}(3)$ in   Table 3  we state the subalegbras $\fr{k}_{i}$,  the dimensions $d_{i}:=\dim_{\bb{R}}\fr{m}_{i}$ for $i=1, \ldots, 5$  ($d_{0}=1$) and  the vanishing or not of the triple $A_{144}$, which plays an essential role.      {\small{ \begin{center}
{\bf{Table 3.}} The simple ideals $\fr{k}_{i}$, the dimensions $d_{i}$, and the vanishing of $A_{144}$ \end{center} }}
\begin{center}
\begin{tabular}{c|c|c|c|c|c|c|c|c}
 $G(i_{o})$ & $\fr{k}_{1}$, & $\fr{k}_{2}$, & $d_{1}$,  & $d_{2}$, & $d_{3}$, & $d_{4}$, & $d_{5}$ & $A_{144}=0$  \\
 \thickline
 $\F_{4}(2)$ &  $\fr{su}_{2}$ & $\fr{su}_{3}$ & $3$ & $8$ & $24$ & $12$ & $4$ & $\checked$ \\
 $\E_{7}(5)$ & $\fr{su}_{2}$ & $\fr{su}_{6}$ & $3$ & $35$ & $60$ & $30$ & $4$ & $\checked$ \\
 $\E_8(2)$ & $\fr{su}_{2}$ & $\fr{e}_6$ &  $3$ & $78$ & $108$ & $54$ & $4$& $\checked$ \\
 \hline
 $\E_7(3)$ & $\fr{su}_{5}$ & $\fr{su}_{3}$ & $24$ & $8$ & $60$ & $30$ & $10$ & $A_{144}\neq 0$\\
 \thickline
 \end{tabular}
\end{center}
 
 \begin{prop}\label{F42}
  For the reductive decomposition (\ref{II3}) associated to  the compact simple Lie groups $\F_{4}(2)$, $\E_{7}(5)$, $\E_8(2)$ and for the left-invariant metric given by (\ref{invII3}), the non-zero structure constants $A_{ijk}$ $(0\leq i, j, k\leq 5)$ are the following (and their symmetries): $A_{033}$,  $A_{044}$, $A_{055}$, $A_{111}$, $A_{133}$, $A_{155}$,  $A_{222}$, $A_{233}$, $A_{244}$,  $A_{334}$, and  $A_{345}$.  This also holds   for $\E_7(3)$, but in this case one has  in addition $A_{144}\neq 0$. 
 \end{prop}
 
  \begin{proof}
  Similarly with Proposition  \ref{G22}, the triples $A_{334}, A_{345}$ are non-zero  because $M=G/K$ is such that $b_{2}(M)=1$ with  $\fr{p}=\fr{p}_{1}\oplus\fr{p}_{2}\oplus\fr{p}_{3}$.  From the other cases we describe these which are less obvious.     
  
   Consider the decomposition (\ref{II3}) of the Lie algebra   $\fr{g}$ and  let   $\fr{t}\subset\fr{g}$ be  a maximal abelian subalgebra. Let $\fr{g}^{\bb{C}}=\fr{t}^{\bb{C}}\oplus\sum_{\al\in R}\fr{g}^{\al}$ the root space decomposition of  the complexification $\fr{g}^{\bb{C}}:=\fr{g}\otimes\bb{C}=\fr{k}^{\bb{C}}\oplus\fr{p}_{1}^{\bb{C}}\oplus\fr{p}_{2}^{\bb{C}}\oplus\fr{p}_{3}^{\bb{C}}$  with respect to the Cartan subalgebra (CSA) $\fr{t}^{\bb{C}}\subset\fr{g}^{\bb{C}}$.   Next we shall identify roots  $\al\in R$ with   vectors   $H_{\al}\in\sqrt{-1}\fr{t}$,  defined by $\al(H)=B(H_{\al}, H),  \ \forall H\in\fr{t}^{\bb{C}}$.   Choose a  Weyl  basis   of  root vectors $\{E_{\al}\in\fr{g}^{\al} : \al\in  R\}$ and let $\Pi=\{\al_1, \ldots, \al_\ell\}$ $(\ell=\dim\fr{t})$ be the  fixed  fundamental basis   and $R^{+}$ the associated positive roots. Then, there exists a subset $\Pi_{K}\subset\Pi$  such that $R_{K}=R\cap\langle\Pi_{K}\rangle$  be the root system of the (semi-simple part) of the reductive Lie algebra  $\fr{k}^{\bb{C}}=Z(\fr{k}^{\bb{C}})\oplus \fr{k}^{\bb{C}}_{ss}$, where $\langle\Pi_{K}\rangle$  is  the subspace of $\sqrt{-1}\fr{t}$ generated by $\Pi_{K}$ with integer coefficients. Similarly, we write $R_{K}^{+}:=R^{+}\cap\langle\Pi_{K}\rangle$ for the corresponding positive roots.   Due to reductive decomposition (\ref{II3}),  $R_{K}$ splits into two root  subsystems, say $R_{K_{1}}, R_{K_{2}}$, which can be identified with the root systems  of  $\fr{k}_{i}^{\bb{C}}=\fr{k}_{i}\otimes\bb{C}$ $(i=1, 2)$. Hence we get $R_{K}=R_{K_{1}}\sqcup R_{K_{2}}$, $\Pi_{K}=\Pi_{K_{1}}\sqcup\Pi_{K_{2}}$, e.t.c., and  we decompose $\fr{k}^{\bb{C}}_{ss}$ as follows:
  \[
    \fr{k}_{ss}^{\bb{C}}:=[\fr{k}^{\bb{C}}, \fr{k}^{\bb{C}}]=\sum_{\al\in\Pi_{K_{1}}}\bb{C}H_{\al}\oplus\sum_{\be\in\Pi_{K_{2}}}\bb{C}H_{\be}\oplus\sum_{\al\in R_{K_{1}}}\fr{g}^{\al}\oplus\sum_{\be\in R_{K_{2}}}\fr{g}^{\be}=\fr{t}'\oplus \sum_{\al\in R_{K_{1}}}\fr{g}^{\al}\oplus\sum_{\be\in R_{K_{2}}}\fr{g}^{\be}.  \]
Here $\fr{t}':=\{\sum_{\al\in\Pi_{K_{1}}}\bb{C}H_{\al}\oplus\sum_{\be\in\Pi_{K_{2}}}\bb{C}H_{\be}\big\}\subset\fr{t}^{\bb{C}}$ is a CSA of   $\fr{k}_{ss}^{\bb{C}}$. Set now $\fr{h}:=\{H\in\fr{t} : (H, \Pi_{K})=0\}$; this is a real form of the centre $Z(\fr{k}^{\bb{C}})\cong\fr{k}_{0}^{\bb{C}}$ of  $\fr{k}^{\bb{C}}$.  Since $\fr{h}\subset\fr{t}\subset\fr{k}$, we  write  $\fr{t}=\fr{h}\oplus\fr{h}^{\perp}$ with 
\[
\fr{h}^{\perp}\cong {\rm span}\{\sqrt{-1}H_{\al} : \al\in \Pi_{K_{1}}\}\oplus{\rm span}\{\sqrt{-1}H_{\be} : \be\in \Pi_{K_{2}}\}:=\fr{h}^{\perp}_{1}\oplus\fr{h}^{\perp}_{2}.
\]
Hence the  CSA $\fr{t}'\subset\fr{t}^{\bb{C}}\subset\fr{k}^{\bb{C}}$ is just the complexification of $\fr{h}^{\perp}$, i.e. $\fr{t}'=\fr{h}^{\perp}\oplus\sqrt{-1}\fr{h}^{\perp}$. Then, one has that
 \[
 \dim_{\bb{R}}\fr{h}^{\perp}=\dim_{\bb{C}}\fr{t}'=|\Pi_{K}|=|\Pi_{K_{1}}|+|\Pi_{K_{2}}|=\dim_{\bb{R}}\fr{h}^{\perp}_{1}+\dim_{\bb{R}}\fr{h}^{\perp}_{2}  \quad \Rightarrow\quad \dim_{\bb{R}}\fr{h}=\ell-|\Pi_{K}|=1.
 \]
Thus $\sqrt{-1}\fr{h}\cong \fr{k}_{0}=\fr{u}_{1}$.  For $\F_4(2)$, $\E_7(5)$ and $\E_8(2)$ it is  $\fr{k}_{1}\cong\fr{su}_{2}$ with   corresponding  root system    $R_{K_{1}}=\{\pm \al_1\}$, $R_{K_1}=\{\pm \al_1\}$ and $R_{K_1}=\{\pm\al_6\}$, respectively.    If $G=G(i_{o})$ is $\E_7(3)$, then $\fr{k}_{1}\cong\fr{su}_{5}$ and $R_{K_{1}}$ is more complicated (see below).  In the following table,  for any Lie group $G=G(i_{o})$ of Type $II_{b}(3)$ we summarize  the  data encoded by the decomposition    $\fr{t}=\fr{h}\oplus\fr{h}_{1}^{\perp}\oplus\fr{h}_{2}^{\perp}$.
      \[
   \begin{tabular}{c|c|c|c|c|c|c}
   & $\dim\fr{t}$ & $\Pi_{K_{1}}$ & $\dim\fr{h}_{1}^{\perp}$ & $\Pi_{K_{2}}$ &  $\dim\fr{h}_{2}^{\perp}$ & $\Pi_{M}=\{\al_{i_{o}}\}$ \\
   \thickline
   $\F_4(2)$ & $4$ &  $\{\al_1\}$ & $1$ & $\{\al_3, \al_4\}$ & $2$ & $\al_2$\\
   $\E_7(5)$ & $7$ & $\{\al_6\}$ & $1$ & $\{\al_1, \cdots, \al_4, \al_7\}$ & $5$ & $\al_5$ \\
   $\E_8(2)$ & $8$ & $\{\al_1\}$ & $1$ & $\{\al_3, \cdots, \al_8\}$ & $6$ & $\al_2$\\
     \hline
     $E_7(3)$ & $7$ & $\{\al_4, \cdots, \al_7\}$ & $4$ & $\{\al_1, \al_2\}$ & $2$ & $\al_3$ \\
     \thickline
   \end{tabular}
   \]
     Now, $\fr{t}$ is a common maximal abelian subalgebra of  $\fr{k}\subset\fr{g}$. Hence $\fr{k}=\fr{u}_{1}\oplus\fr{k}_{1}\oplus\fr{k}_{2}=\fr{t}\oplus\sum_{\al\in R_{K}^{+}}\{\bb{R}A_{\al}\oplus\bb{R}B_{\al}\}$, where $A_{\al}:=(E_{\al}+E_{-\al})$ and $B_{\al}:=\sqrt{-1}(E_{\al}-E_{-\al})$.  The simple ideals   $\fr{k}_{1}, \fr{k}_{2}$ can be viewed as 
          \begin{eqnarray*}
 \fr{k}_{1}&=&\sum_{\al\in\Pi_{K_{1}}}\bb{R} H_{\al}\oplus \sum_{\al\in R_{K_{1}}^{+}}\{\bb{R}A_{\al}\oplus\bb{R}B_{\al}\}= \fr{h}_{1}^{\perp}\oplus \sum_{\al\in R_{K_{1}}^{+}}\{\bb{R}A_{\al}\oplus\bb{R}B_{\al}\},\quad (\ast)\\
 \fr{k}_{2}&=&\sum_{\be\in\Pi_{K_{2}}}\bb{R}H_{\be}\oplus \sum_{\be\in R_{K_{2}}^{+}}\{\bb{R}A_{\be}\oplus\bb{R}B_{\be}\}=\fr{h}_{2}^{\perp}\oplus \sum_{\be\in R_{K_{2}}^{+}}\{\bb{R}A_{\be}\oplus\bb{R}B_{\be}\}. \quad \ (\ast\ast)
  \end{eqnarray*}
        Let $R_{M}^{+}:=R^{+}\backslash (R_{K_{1}}^{+}\sqcup R_{K_{2}}^{+})$ be the complementary roots of $M=G/K$.   Because $\Hgt(\al_{i_{o}})=3$, 
 this set splits into 3 subsets $R_{M}^{+}=R_{1}^{+}\sqcup R^{+}_{2}\sqcup R^{+}_{3}$, given by 
 $
 R^{+}_{t}:=\{ \al=\sum_{j=1}^{4}c_{j}\al_{j} \in R_{M}^{+}  :  \ c_{i_{o}} =t, \ 1\leq t\leq 3\}.
$
 Then
 \begin{equation}\label{3p}
  \fr{m}_{3}\cong\fr{p}_{1}=\sum_{\al\in R^{+}_{1}}\{\bb{R}A_{\al}\oplus\bb{R}B_{\al}\}, \quad    \fr{m}_{4}\cong\fr{p}_{2}=\sum_{\al\in R^{+}_{2}}\{\bb{R}A_{\al}\oplus\bb{R}B_{\al}\}, \quad 
   \fr{m}_{5}\cong\fr{p}_{3}=\sum_{\al\in R^{+}_{3}}\{\bb{R}A_{\al}\oplus\bb{R}B_{\al}\}.
 \end{equation}
  \noindent{\bf $1^{\rm st}$ way.}  Fix   one of the Lie groups  $\F_4(2)$, $\E_7(5)$ and $\E_8(2)$.  A method to prove that $A_{144}=0$ but  $A_{133}\neq 0$, $A_{155}\neq 0$, is via the inclusions $[\fr{k}_{1}, \fr{m}_{3}]\subset\fr{m}_{3}$, $[\fr{k}_{1}, \fr{m}_{5}]\subset\fr{m}_{5}$ and $[\fr{k}_{1}, \fr{m}_{4}]=0$.  These can be obtained directly by considering root vectors associated to roots of $R_{K_{1}}^{+}$ and $R_{M}^{+}$ and combining  $(\ast)$ and (\ref{3p}).   In the same way and by using now $(\ast\ast)$ and (\ref{3p}) we get $[\fr{k}_{2}, \fr{m}_{5}]=0$, which implies that $A_{255}=0$.   Of course, this method applies  also  for $\E_7(3)$; for this Lie group  the inclusion $[\fr{k}_{1}, \fr{m}_{4}]\subset\fr{m}_{4}$ holds, so $A_{144}\neq 0$. 
 
  \noindent{\bf $2^{\rm nd}$ way.}  An alternative way to examine  the behaviour of $A_{144}$ is based on the orthogonality of roots.  Let us denote the  unique simple root  belonging in $\Pi_{K_{1}}$ by $\phi$ (recall that  $|\Pi_{K_{1}}|=1$ for $\F_4(2)$, $\E_7(5)$ and $\E_8(2)$).   Consider some complementary  root  $\al\in R_{2}^{+}$ associated to $\fr{p}_{2}\cong\fr{m}_{4}$.  In terms of simple roots, $\al$ may be expressed  by $\al=\phi+2\al_{i_{o}}+\sum_{k}c_{k}\al_{k}$ with $(i_{o}<k\leq\ell)$, or $\al=\sum_{k}c_{k}\al_{k}+2\al_{i_{o}}+\phi$ with $(1\leq k<i_{o})$.   In full details:
   \[
   \F_4(2) : \al=\phi+2\al_2+\sum_{k=3}^{4}c_{4}\al_{k}, \quad \E_7(5) : \al=\sum_{k=1}^{4}c_{k}\al_k+c_7\al_7+2\al_5+\phi, \quad \E_8(2) : \al=\phi+2\al_2+\sum_{k=3}^{8}c_{k}\al_{k}.
   \]
  The fact that the simple root $\phi\in\Pi_{K}\equiv R_{K}^{+}$  appears in the expression of $\al\in R^{+}_{+}$   always with coefficient 1, can be straightforward checked by the expressions of positive roots in terms of simple roots.
 Since $\phi$  is  connected to $-\tilde{\al}$ it is $(\phi, \tilde{\al})\neq0$  (in general $(\al_i, \tilde{\al})\geq 0$).  By assuming now that $\al=\phi+2\al_{i_{o}}+\sum_{k}c_{k}\al_{k}$ (the other case is treated similarly) we get that
      \begin{eqnarray*}
   \frac{2(\phi+2\al_{i_{o}}+\sum_{k}c_{k}\al_{k}, \phi)}{(\phi, \phi)}&=&\frac{2(\phi+2\al_{i_{o}}, \phi)}{(\phi, \phi)}= \frac{2(\phi, \phi)}{(\phi, \phi)}+\frac{2(2\al_{i_{o}}, \phi)}{(\phi, \phi)}=1-1=0,      \end{eqnarray*}
since   $2(\al_{i_{o}}, \phi)=-(\phi, \phi)$.   This shows that the orthogonality of the roots $\al$ and $\phi$, and since $\phi$ spans $R_{K_{1}}^{+}$, by the definition of   $A_{ijk}$ we conclude that $A_{144}=0$. 
   Let us illustrate the computations shortly for  $\F_4(2)$. Let $\Pi=\{\al_=e_2-e_3, \al_2=e_3-e_4, \al_3=e_4, \al_4=\frac{1}{2}(e_1-e_2-e_3-e_4)\}$ be the fixed fundamental basis, with $\tilde{\al}=2\al_1+3\al_2+4\al_3+2\al_4$.    Then, the Cartan matrix is given by 
   \[
   A =  \left( \begin{tabular}{cccc}
 2  & -1 &  0 &  0  \\
-1  &  2 & -2 &  0  \\
 0  & -1 &  2 & -1 \\
 0  &  0 & -1 &  2 
 \end {tabular}
\right)  \ \Rightarrow \    \frac{2(\al_1+2\al_{2}+c_{3}\al_{3}+c_{4}\al_{4}, \al_1)}{(\al_1 \al_1)}= \frac{2(\al_1, \al_1)}{(\al_1, \al_1)}+\frac{2(2\al_{2}, \al_1)}{(\al_1, \al_1)}=1-1=0.        
\]
   We explain now why $A_{144}\neq 0$ for $\E_7(3)$. Recall that    $\Pi_{K_{1}}=\{\al_4, \cdots, \al_7\}$, with result   $R_{K_{1}}\cong R_{\SU_{5}}$ and
  $  R_{K_{1}}^{+}=\Pi_{K}\sqcup \{\al_4+\al_5, \ \al_4+\al_5+\al_6, \ \al_4+\al_7, \ \al_5+\al_6, \ \al_4+\al_5+\al_7, \al_4+\al_5+\al_6+\al_7\}$.  For convenience, we present   the set $R^{+}_{2}$ in terms of simple roots (we state only the coefficients) 
  \[
  R_{2}^{+}=\left\{\begin{tabular}{c|c|c|c|c}
  $(0, 1, 2, 2, 1, 0, 1)$, & $(0, 1, 2, 2, 1, 1, 1)$, & $(0, 1, 2, 2, 2, 1, 1)$, &   $(0, 1, 2, 3, 2, 1, 1)$, & $(1, 1, 2, 2, 1, 0, 1)$, \\
$(1, 1, 2, 2, 1, 1, 1)$, & $(1, 1, 2, 2, 2, 1, 1)$, & $(1, 1, 2, 3, 2, 1, 1)$, & $(1, 2, 2, 2, 1, 0, 1)$,  & $(1, 2, 2, 2, 1, 1, 1)$, \\
$(1, 2, 2, 2, 2, 1, 1)$, & $(1, 2, 2, 3, 2, 1, 1)$,  & $(0, 1, 2, 3, 2, 1, 2)$, & $(1, 1, 2, 3, 2, 1, 2)$, &  $(1, 2, 2, 3, 2, 1, 2)$
\end{tabular}
\right\}.
 \]
 Choose for example $\al:=\al_1+\al_2+2\al_3+2\al_4+2\al_5+\al_6+\al_7\in R_{2}^{+}$ and $\phi:=\al_4+\al_7\in R_{K_{1}}^{+}$. Then,  $[E_{\al}, E_{\phi}]=N_{\al\phi}E_{\al+\phi}\neq 0$, since $\al+\phi=\al_1+\al_2+2\al_3+3\al_4+2\al_5+\al_6+2\al_7\in R^{+}_{2}$, i.e. $[E_{\al}, E_{\phi}]\in\fr{m}_{4}$ and hence $A_{144}\neq 0$.  This can   be verified  by the orthogonality of roots as well;  since the fixed basis of  simple roots of  $\E_7$   is such that $(\al_{i}, \al_{i})=2$ and $(\al_{1}, \al_{2})=(\al_2, \al_3)=(\al_3,  \al_4)=(\al_4, \al_5)=(
\al_4, \al_7)=(\al_5, \al_6)=-1$, it follows that   
 \[
 \displaystyle \frac{2(\al, \phi)}{(\phi, \phi)}=-3/2\neq 0.
 \]
   We finish the proof with a short remark about   $A_{225}$. This case can be treated by similar methods as above, however  it  occurs in an faster way via  the painted Dynkin diagram associated to a Lie group $G=G(i_{o})$ of Type $II_{b}(3)$.  For such a group and the reductive decomposition of its Lie algebra $\fr{g}$ given by (\ref{II3}), observe  that  the Dynkin diagram of  $\fr{k}_{2}$ is not connected with   $-\tilde{\al}$.  Hence $[\fr{k}_{2}, E_{\tilde{\al}}]=0$, and  since $\tilde{\al}\in R_{3}^{+}$, i.e. $E_{\tilde{\al}}\in\fr{p}_{3}\cong\fr{m}_{5}$,  it follows that $A_{225}=0$.
    \end{proof}
 Now, an easy application of Lemma \ref{ric} gives that
   \begin{corol}\label{ricg3}
On $(\E_7(3), \langle \ , \ \rangle)$, the components ${r}_{i}$ of  the Ricci tensor $\Ric_{\langle \ , \ \rangle}$ associated to the left-invariant metric $\langle \ , \ \rangle$  given by (\ref{invII3}),  are  described as follows
\begin{equation*}
\left\{\begin{array}{ll} 
r_0 &=  \displaystyle{\frac{u_{0}}{4d_0}\biggl(\frac{A_{033}}{{x_{1}}^{2}}+\frac{A_{044}}{{x_{2}}^{2}}+\frac{A_{055}}{{x_{3}}^{2}}
\biggr), }\quad 
r_1  =  \displaystyle{\frac{A_{111}}{4d_{1}}\cdot\frac{1}{u_{1}} + 
\frac{u_{1}}{4d_1} \biggl( \frac{A_{133}}{{x_{1}}^{2}} +\frac{A_{144}}{{x_{2}}^{2}} +\frac{A_{155}}{{x_{3}}^{2}}\biggr),}
  \\ & \\
  r_2 &=  \displaystyle{\frac{A_{222}}{4d_{2}}\cdot\frac{1}{u_{2}} + 
\frac{u_{2}}{4d_2} \biggl( \frac{A_{233}}{{x_{1}}^{2}} +\frac{A_{244}}{{x_{2}}^{2}}\biggr),}
 \\ & \\
r_3 &=  \displaystyle{\frac{1}{2 x_1} -
\frac{1}{2d_3}\cdot\frac{1}{{x_{1}}^{2}} 
\biggl(u_{0}\cdot A_{033}+u_{1}\cdot A_{133}+u_{2}\cdot A_{233}+x_{2}\cdot A_{334} \biggr) 
 +  \frac{A_{345}}{2d_3}  
\biggl(\frac{x_1}{x_{2} x_{3}} - \frac{x_{2}}{x_1 x_{3}}- \frac{x_{3}}{x_1 x_{2}}
\biggr),  }
\\ & \\
r_4 & = \displaystyle{\frac{1}{2 x_{2}}   -\frac{1}{2d_{4}{x_{2}}^{2}}\bigg(u_{0}\cdot A_{044}+u_{1}\cdot A_{144} +u_{2}\cdot A_{244}\biggr)
+\frac{A_{334}}{4d_{4}}\biggl( \frac{x_{2}}{{x_{1}}^{2}} - \frac{2}{x_{2}}\biggr)
  +\frac{A_{345}}{2d_{4}}
\biggl(\frac{x_{2}}{x_1 x_{3}} - \frac{x_1}{x_{2} x_{3}}- \frac{x_{3}}{x_1 x_{2}}
\biggr), }
\\ & \\
r_5 & = \displaystyle{\frac{1}{2 x_{3}} -  
\frac{1}{2d_5} \cdot\frac{1}{{x_{3}}^{2}}\biggl(u_{0}\cdot A_{055}+u_{1}\cdot A_{155}\biggr)+\frac{A_{345}}{2d_{5}}
\biggl(\frac{x_{3}}{x_1 x_{2}} - \frac{x_1}{x_{3} x_{2}}- \frac{x_{2}}{x_{3} x_1}
\biggr). } 
\end{array}
\right.
\end{equation*}
The corresponding Ricci components  of $\F_4(2)$, $\E_{7}(5)$ and $\E_8(2)$  occur by the same  expressions, by setting however $A_{144}=0$.
\end{corol}

 \subsection{The structure constants} We proceed now with the non-zero structure constants. We   prove that
    \begin{lemma}\label{A345}
  For the reductive decomposition (\ref{II3}) and for the left-invariant metric $\langle \ , \ \rangle$ on a Lie group $G=G(i_{o})$ of Type $II_{b}(3)$   the non-zero triples $A_{ijk}$ attain the following values:
\[
\begin{tabular}{c|c|c|c|c|c|c|c|c|c|c|c|c}
& $A_{033}$ & $A_{044}$ & $A_{055}$ & $A_{111}$ & $A_{133}$ & $A_{144}$ & $A_{155}$ & $A_{222}$ & $A_{233}$ & $A_{244}$ & $A_{334}$ & $A_{345}$ \\
\thickline
$\F_4(2)$ & $2/9$ & $4/9$ & $1/3$ & $2/3$ & $2$ & $0$ & $1/3$ & $4/3$ & $40/9$ & $20/9$ & $4$ & $4/3$ \\
$\E_7(5)$ & $5/18$ & $5/9$ & $1/6$ & $1/3$ & $5/2$ & $0$ & $1/6$ & $35/3$ & $140/9$ & $70/9$ & $10$ & $5/3$ \\
$\E_8(2)$ & $3/10$ & $3/5$ & $1/10$ & $1/5$ & $27/10$ & $0$ & $1/10$ & $156/5$ & $156/5$ & $78/5$ & $18$ & $9/5$ \\
\hline
$\E_7(3)$ & $2/9$ & $4/9$ & $1/3$ & $20/3$ & $12$ & $4$ & $4/3$ & $4/3$ & $40/9$ & $20/9$ & $10$ & $10/3$  \\
\thickline
\end{tabular}
\]
  \end{lemma}
  \begin{proof}
We use the Killing metric to obtain a system  of 5 equations $\{r_0-r_1=0, r_1-r_2=0, r_2-r_3=0, r_3-r_4=0, r_4-r_5=0 \}$ depending on 11 (or 12) unknowns, i.e. the  triples appearing in Proposition  \ref{F42}.  Two of them,  namely  the triples  $A_{334}$ and $A_{345}$  can be computed  similarly with   a Lie group $G$ of Type $I_{b}(3)$,    using   the (unique) K\"ahler-Einstein metric $x_1=1, x_{2}=2, x_{3}=3$ that $M=G/K$ admits. This gives   (see \cite{stauros})
\begin{equation}\label{A3456}
 A_{334}=\displaystyle\frac{d_3 d_4 + 2 d_3 d_5 - d_4 d_5}{d_3 + 4 d_4 + 9 d_5}, \qquad A_{345}=\displaystyle\frac{(d_3 + d_4) d_5}{d_3 + 4 d_4 + 9 d_5}.
 \end{equation}
 \noindent{\bf Case of $\F_4(2), \E_7(5), \E_8(2)$.}  Assume that  $G\in\{\F_4(2), \E_7(5), E_8(2)\}$. We will show how one can compute the other triples, i.e., $A_{033}$, $A_{044}$, $A_{055}$, $A_{111}$, $A_{133}$, $A_{144}$, $A_{155}$, $A_{222}$, $A_{233}$, and $A_{244}$, in a global way.  For the construction of more equations, we use first the twistor fibration of our  flag manifold over a   symmetric space.  Set
 \begin{equation}\label{fib1}
 \fr{g}=\fr{h}\oplus\fr{n}, \quad\fr{h}:=\fr{h}_{1}\oplus\fr{h}_{2}, \quad \fr{h}_{1}:=\fr{k}_{0}\oplus\fr{k}_{2}\oplus\fr{p}_{2},  \quad \fr{h}_{2}:=\fr{k}_{1}, \quad \fr{n}:=\fr{p}_{1}\oplus\fr{p}_{3}.
  \end{equation}
 This is a reductive decomposition of $\fr{g}$ with $[\fr{h}, \fr{h}]\subset\fr{h}$, $[\fr{h}, \fr{n}]\subset\fr{n}$ and $[\fr{n}, \fr{n}]\subset\fr{h}$. Since $\fr{k}\subset\fr{h}$ we get a fibration $G/K\to G/H$, where $H\subset G$ is the connected Lie subgroup generated by $\fr{h}$. The base space $B=G/H$ is an  
  irreducible symmetric space and the fiber is a Hermitian symmetric space.
 
   A second  reductive decomposition of $\fr{g}$ is given by
  \begin{equation}\label{fib2}
 \fr{g}=\fr{q}\oplus\fr{r}, \quad \fr{q}:=\fr{q}_{1}\oplus\fr{q}_{2}, \quad \fr{q}_{1}:=\fr{k}_{0}\oplus\fr{k}_{1}\oplus\fr{p}_{3}, \quad \fr{q}_{2}:=\fr{k}_{2}, \quad  \fr{r}:=\fr{p}_{1}\oplus\fr{p}_{2}.
  \end{equation}
  In this case, the pair $(\fr{g}, \fr{q})$ is not symmetric, since  $[\fr{q}, \fr{q}]\subset\fr{q}$, $[\fr{q}, \fr{r}]\subset\fr{r}$ but  $[\fr{r}, \fr{r}]\subset(\fr{q}\oplus\fr{r})=\fr{g}$. It is $\fr{r}=T_{o'}G/Q$ where $Q\subset G$ is the connected Lie subgroup with Lie algebra $\fr{q}$.  The fiber of the induced fibration  is $\bb{C}P^{2}\cong\SU_3/(\U_1\times\SU_2)\cong\SU_3/\U_2$ and the base space $B'=G/Q$ is isotropy irreducible (\cite{Bes}).  Let us summarize the necessary details as follows:
    \[
 \begin{tabular}{l | l | l | l     l |l|l}
 \multicolumn{4}{c}{The twistor fibration $G/K\to G/H$}    & \multicolumn{3}{c}{The fibration $G/K\to G/Q$} \\ 
   \hline\hline
 $G=G(i_{o})$ & $\fr{h}_{1}$ & $\fr{h}_{2}$ & $B=G/H$ & $\fr{q}_{1}$ & $\fr{q}_{2}$ &  $B'=G/Q$ \\
 \thickline
 $\F_4(2)$ & $\fr{sp}_{3}$ & $\fr{su}_{2}$ & $\F_4/(\Sp_3\times\SU_2)$ & $\fr{su}_3$ & $\fr{su}_{3}$ &  $\F_4/(\SU_{3}\times\SU_{3})$ \\
 $\E_7(5)$ & $\fr{so}_{12}$ & $\fr{su}_{2}$ & $\E_7/(\SO_{12}\times\SU_2)$ &  $\fr{su}_3$ & $\fr{su}_{6}$ &   $\E_7/(\SU_{3}\times\SU_{6})$\\
 $\E_8(3)$ & $\fr{e}_{7}$ &  $\fr{su}_{2}$ & $\E_8/(\E_7\times\SU_2)$   & $\fr{su}_{3}$ & $\fr{e}_{6}$ &   $\E_8/(\SU_{3}\times\E_{6})$  \\
 \thickline
   \end{tabular}
  \]
  The reductive  decompositions (\ref{fib1}) and (\ref{fib2}) induce  left-invariant metrics on $G$, given by
 \[  
  \langle\langle \ , \ \rangle\rangle=w_1\cdot B|_{\fr{h}_{1}}+w_{2}\cdot B|_{\fr{h}_{2}}+w_{3}\cdot B|_{\fr{n}}, \quad w_1, w_2, w_3\in\bb{R}_{+},\quad 
  \langle\langle \ , \ \rangle\rangle'=z_1\cdot B|_{\fr{l}_{1}}+z_{2}\cdot B|_{\fr{l}_{2}}+z_{3}\cdot B|_{\fr{r}}, \quad z_1, z_2, z_3\in\bb{R}_{+},
    \]
   respectively.   For $w_{1}=u_{0}=u_{2}=x_{2}$, $w_2=u_1$ and $w_3=x_1=x_3$,  the metrics $\langle\langle \ , \ \rangle\rangle$ and $\langle \ , \ \rangle$ coincide, and the same holds between  $\langle\langle \ , \ \rangle\rangle'$ and $\langle \ , \ \rangle$ for  $z_{1}=u_{0}=u_{1}=x_{3}$, $z_2=x_1$ and $z_3=x_1=x_2$.  For these values, by comparing the Ricci components we get the relations $r_0=r_2=r_4$ and $r_{0}=r_{1}=r_{5}$. 
The first one $r_0=r_2=r_4$ and after  introducing the values of $A_{334}, A_{345}$ given by (\ref{A3456}),  implies that
 \begin{eqnarray*}
    A_{334} + 2 A_{345} -d_4(A_{033} +A_{055})&=&d_{2}(A_{334}+ 2A_{345}) - 
    d_4A_{233},\\
     -A_{222}- A_{244}+d_{2}A_{044} &=& 2(A_{044}  + 
    A_{244}  +   A_{334}   + 2 A_{345})  + d_{4}(A_{044}-2),\\
    2(A_{044} + A_{244}  +  A_{334}   + 2 A_{345}) +d_{4}(A_{044} - 2) &=&2d_{2}(A_{044} +A_{244} + A_{334} +2
    A_{345})+ d_4(A_{222}+A_{244}- 2 d_{2}).
       \end{eqnarray*}
 Similarly, by $r_{0}=r_{1}=r_{5}$ we use the relations
   \begin{eqnarray*}
   A_{133} - d_{1}(A_{033}+ 
   A_{044})&=& -d5 (d_{3}(A_{133} - 2 d_{1}) + d_{4}(4 A_{133}  - 2 d1) + 
    9 d_{5}A_{133}),\\
     A_{111} + A_{155} - A_{055}&=&2d_{3}(A_{055} + 
    A_{155})+ 8d_{4}(A_{055}+A_{155}) + 18d_{5}(A_{055} + 
   A_{155}) \\
   && + d_{5}\big(d_{3}A_{055}+2d3+ 4d_{4}(A_{055}-1)
   + 9d_{5}(A_{055} -2)\big).
   \end{eqnarray*}
   After combining now  this data with the system defined by the Killing metric, we obtain   all $A_{ijk}$  in terms of the dimensions $d_{i}, i=1, \ldots, 5$ (recall that $d_{0}=1$); then one can  complete the proof based on Table 3.  
    
   \noindent{\bf Case of $\E_7(3)$.}  In this case, the first reductive decomposition   is again the twistor fibration associated to the flag manifold $\E_7/(\U_1\times\SU_5\times\SU_3)$.   This is almost similar with (\ref{fib1}), i.e. we set
   \[
  \fr{g}=\fr{h}\oplus\fr{n}, \quad\fr{h}:=\fr{k}\oplus\fr{p}_{2}\cong \fr{k}\oplus\fr{m}_{4}, \quad \fr{n}:=\fr{p}_{1}\oplus\fr{p}_{3}\cong\fr{m}_{3}\oplus\fr{m}_{5}.
   \]
   It follows  that $\fr{h}\cong\fr{su}_{8}\subset\fr{g}$ and since $[\fr{n}, \fr{n}]\subset\fr{h}$   the base space of the fibration $G/K\to G/H$ is  an irreducible symmetric space, namely $G/H\cong \E_7/\SU_8$. For a  second reductive decomposition   we use  (\ref{fib2});   it is $\fr{q}=\fr{q}_{1}\oplus\fr{q}_{2}$ with $\fr{q}_{1}\cong\fr{su}_{6}$ and $\fr{q}_{2}=\fr{su}_{3}$. Thus we obtain the fibration 
   \[
   \bb{C}P^{5}=\SU_6/\U_5\to \E_7/(\U_1\times\SU_5\times\SU_3)\to\E_7/(\SU_6\times\SU_3),\]
    where  the base space is  isotropy irreducible (\cite{Bes}).  Considering     new left-invariant metrics associated to these decompositions and   following a similar procedure like before, we obtain the desired results.
       \end{proof}
\begin{remark}
\textnormal{Let us   verify the values of $A_{111}, A_{222}$ via Remark \ref{sakidea}. For $\F_4(2)$  recall that    $B_{\F_{4}}(\al_1, \al_1)=B_{\F_{4}}(\al_2, \al_2)=2B_{\F_{4}}(\al_3, \al_3)=2B_{\F_{4}}(\al_4, \al_4)$.  Because  $\fr{k}_{1}=\fr{su}_{2}$ is generated by $\al_1$,   we see that $A_{111}=c\cdot \dim\fr{su}_{2}$ where $c=B_{\SU_{2}}/B_{\F_{4}}=4/18$. Thus $A_{111}=2/3$. For the triple $A_{222}$, the simple ideal $\fr{k}_{2}=\fr{su}_{3}$ is generated by the short   roots. Hence $A_{222}=\frac{1}{2}\big(c'\cdot \dim\fr{su}_{3}\big)$ where $c'=B_{\SU_{3}}/B_{\F_{4}}=6/18$. This shows that $A_{222}=4/3$. Similar are treated the other Lie groups. For example, for $\E_7(3)$ we get $A_{111}=(B_{\SU_{5}}/B_{\E_{7}})\cdot\dim\fr{su}_{5}=(10/36)\cdot 24=2/3$ and $A_{222}=(B_{\SU_{3}}/B_{\E_{7}})\cdot\dim\fr{su}_{3}=(6/36)\cdot 8=4/3$.  Although  one is possible to  use these values in  the proof of Lemma \ref{A345},  we remark that  the    reductive decompositions (\ref{fib1}) and (\ref{fib2})  described above, are both  necessary.}
\end{remark}

\subsection{Naturally reductive metrics} 
For a Lie group $G\cong G(i_{o})$ of Type $II_{b}(3)$, 
 left-invariant metrics  on $\G\cong G(i_{o})$  which are $\Ad(K)$-inavariant are given by
 \begin{equation}\label{invI35}
 \langle \ , \  \rangle=u_{0}\cdot B|_{\fr{k}_{0}}+u_{1}\cdot B|_{\fr{k}_{1}}+u_{2}\cdot B|_{\fr{k}_{2}}+x_{1}\cdot B|_{\fr{p}_{1}}+x_{2}\cdot B|_{\fr{p}_{2}}+x_{3}\cdot B|_{\fr{p}_{3}}. 
 \end{equation}

\begin{prop}\label{prop5.5}
If a left invariant metric $\langle \ , \ \rangle$ of the form $(\ref{invI35})$  on $G\cong G(i_{o})$ of Type $II_{b}(3)$ is naturally reductive  with respect to $G \times L$ for some closed subgroup $L$ of $G$, 
then one of the following holds: 

$(1)$ for $G =\F_4(2)$, $\E_7(5)$ and $\E_8(2)$, $u_0 =  u_2 = x_2 $, $x_{1} = x_{3}$,   and for  $G =\E_7(3)$,  $u_0 = u_1 = u_2 = x_2 $,   $x_{1} = x_{3}$ \quad 
$(2)$   $u_0  = u_1  = x_3$,   $x_{1} = x_{2}$  
\quad 
$(3)$ $ x_{1} = x_{2} = x_{3} $.

Conversely, 
 if  one of the conditions $(1)$, $(2)$, $(3)$  holds, then the metric 
 $\langle \ , \ \rangle$ of the form  $(\ref{invI35})$ is  naturally reductive  with respect to $G \times L$, for some closed subgroup $L$ of $G$.
  \end{prop} 
    \begin{proof}
   Let ${\frak l}$ be the Lie algebra of  $L$. Then there are two cases: ${\frak l} \subset {\frak k}$  or ${\frak l} \not\subset {\frak k}$. 
We begin with the second one, i.e.  ${\frak l} \not\subset {\frak k}$. Let ${\frak h}$ be the subalgebra of ${\frak g}$ generated by ${\frak l}$ and ${\frak k}$. 
Since 
$ \fr{g}=\fr{k}_{0}\oplus\fr{k}_{1}\oplus\fr{k}_{2}\oplus\fr{p}_{1}\oplus\fr{p}_{2}\oplus\fr{p}_{3}$ is an irreducible decomposition as $\mbox{Ad}(K)$-modules,  the Lie algebra $\frak h$  needs to contain  at least one of  ${\frak p}_{1}$,  ${\frak p}_{2}$, ${\frak p}_{3}$. 
Assume that $\fr{p}_{1}\subset \frak h$.
 Then,  
$\left[{\frak p}_{1}, {\frak p}_{1}\right] \cap {\frak p}_2 \neq\{0\}$  and hence $\frak h$   contains   ${\frak p}_{2}$.  
It is also  $\left[{\frak p}_{1}, {\frak p}_{2}\right] \cap {\frak p}_3 \neq \{0\}$, thus $\frak h$ contains ${\frak p}_{3}$, as well.  It follows that $\frak h=\frak g $ and  the $\Ad(L)$-invariant metric $\langle \ , \ \rangle$ of the form $(\ref{invI35})$ is bi-invariant. 
Now, if $\frak h$  contains ${\frak p}_{2}$, then $\fr{ h} \supset {\frak k}\oplus {\frak p}_{2}$. If $\fr{ h} =\fr{k}\oplus {\frak p}_{2}$, then $(\fr{ h}, \fr{p}_{1}\oplus\fr{p}_{3})$ is a symmetric pair. Thus, the metric $\langle \ , \ \rangle$ of the form $(\ref{invI35})$ satisfies $u_0 = u_2 = x_2 $,   $x_{1} = x_{3}$ for $G =\F_4(2)$, $\E_7(5)$ and $\E_8(2)$ and   $u_0 = u_1 = u_2 = x_2 $,   $x_{1} = x_{3}$  for $G =\E_7(3)$. If $\fr{ h} \neq\fr{k}\oplus {\frak p}_{2}$, then it must be $\fr{h}\cap {\frak p}_1\neq\{0\}$ or $\fr{ h}\cap {\frak p}_3\neq\{0\}$, so $\fr{ h} \supset {\frak p}_1$ or $\fr{ h} \supset {\frak p}_3$. Thus we also get $\fr{h} = \fr{g}$ and the $\Ad(L)$-invariant metric $\langle \ , \ \rangle$ of the form $(\ref{invI35})$ is again bi-invariant. 
Consider now the case $\fr{p}_{3}\subset\frak h$.   Then, $\fr{h}\supset {\frak k}\oplus {\frak p}_{3}$. If $\fr{h} =\fr{k}\oplus {\frak p}_{3}$, then $\fr{h}$ is a semi-simple Lie algebra and ${\frak p}_{1}\oplus {\frak p}_{2}$ is an irreducible $\Ad(H)$-module. 
Thus, the metric $\langle \ , \ \rangle$ of the form $(\ref{invI35})$ satisfies $u_0 = u_1 = x_3 $,   $x_{1} = x_{2}$. 
If $\fr{ h} \neq\fr{k}\oplus {\frak p}_{3}$, we conclude that  $\fr{h}\cap {\frak p}_1\neq\{0\}$ or $\fr{ h}\cap {\frak p}_2\neq\{0\}$ and thus $\fr{ h} \supset {\frak p}_1$, or $\fr{ h} \supset {\frak p}_2$. Then, we obtain $\fr{h} = \fr{g}$ and the $\Ad(L)$-invariant metric $\langle \ , \  \rangle$ of the form $(\ref{invI35})$ must be bi-invariant. 

Now we consider the case ${\frak l} \subset {\frak k}$.  Since the  orthogonal complement
 ${\frak l}^{\bot}$ of ${\frak l}$ with respect to $B$ contains the  orthogonal complement 
${\frak k}^{\bot}$ of ${\frak k}$, it follows that ${\frak l}^{\bot} \supset {\frak p}_{1} \oplus  {\frak p}_{2}\oplus  {\frak p}_{3}$.   
Since the  invariant metric $\langle \ , \ \rangle$ is naturally reductive  with respect to $G\times L$,  using Theorem \ref{ziller}.   
 we conclude that  $x_{1} = x_{2} = x_{3} $.

Conversely, 
 if the condition $(1)$  holds, then due to  Theorem \ref{ziller}, the metric 
 $\langle \ , \ \rangle$ given by  $(\ref{invI35})$ is  naturally reductive  with respect to $G \times L$, where $\fr{l} = \fr{k}\oplus {\frak p}_{2}$. Similarly, if the condition  $(2)$  holds, then the metric    given by  $(\ref{invI35})$ is  naturally reductive  with respect to $G \times L$, where $\fr{l} = \fr{k}\oplus {\frak p}_{3}$. Finally, if the condition $(3)$  holds, then the metric given by $(\ref{invI35})$ is  naturally reductive  with respect to $G \times K$. 
  \end{proof}

  \subsection{The homogeneous Einstein equation}
Corollary \ref{ricg3} in combination with Lemma \ref{A345} determines now explicitly   the Ricci tensor $\Ric_{\langle \ , \ \rangle}$ of a Lie group $G=G(i_{o})$ of Type $II_{b}(3)$ with respect to  the left-invariant metric $\langle \ , \ \rangle$.  Hence we can write down explicitly  the homogeneous Einstein equation; this is given by
    \[
    \{r_{0}-r_{1}=0, \quad r_{1}-r_{2}=0, \quad r_{2}-r_{3}=0,  \quad r_{3}-r_{4}=0, \quad r_{4}-r_{5}=0\}
    \]
and it turns out to be equivalent to the following system of equations (we normalise the metric by setting $x_1=1$).

\noindent{\bf \underline{Case of $\F_4(2)$}}
\begin{equation}\label{einf42}
\left\{  \begin{tabular}{ll}
  $g_0=$ &$2 {u_0} {u_1} {x_2}^2 {x_3}^2+3 {u_0} {u_1} {x_2}^2+4{u_0} {u_1} {x_3}^2-6 {u_1}^2 {x_2}^2 {x_3}^2-{u_1}^2 {x_2}^2-2 {x_2}^2 {x_3}^2=0$,\\\\
$g_1=$ &$12 {u_1}^2 {u_2}{x_2}^2 {x_3}^2+2 {u_1}^2 {u_2}{x_2}^2-10 {u_1} {u_2}^2 {x_2}^2{x_3}^2-5 {u_1} {u_2}^2 {x_3}^2-3 {u_1} {x_2}^2 {x_3}^2+4 {u_2} {x_2}^2 {x_3}^2= 0$,\\\\
$g_2=$ &${u_0} {u_2} {x_2}^2 {x_3}+9 {u_1} {u_2}{x_2}^2 {x_3}+50 {u_2}^2 {x_2}^2{x_3}+15 {u_2}^2 {x_3}+18 {u_2}{x_2}^3 {x_3}+6 {u_2} {x_2}^3-108{u_2} {x_2}^2 {x_3}$ \\
& $+6 {u_2}{x_2} {x_3}^2-6 {u_2} {x_2}+9 {x_2}^2 {x_3}= 0$, \\\\
$g_3=$ &$-{u_0} {x_2}^2{x_3}+4 {u_0} {x_3}-9 {u_1}{x_2}^2 {x_3}-20 {u_2} {x_2}^2 {x_3}+20 {u_2} {x_3}-36 {x_2}^3 {x_3}-18 {x_2}^3+108 {x_2}^2 {x_3}$ \\
& $+6 {x_2} {x_3}^2-72 {x_2}{x_3}+18 {x_2}= 0$, \\\\
$g_4=$ &$9 {u_0} {x_2}^2-4{u_0} {x_3}^2+9 {u_1} {x_2}^2-20{u_2} {x_3}^2+18 {x_2}^3{x_3}^2+48 {x_2}^3 {x_3}-108{x_2}^2 {x_3}-48 {x_2} {x_3}^3$ \\
& $+72{x_2} {x_3}^2+24 {x_2}{x_3}= 0$.
  \end{tabular} \right.
 \end{equation}
  We consider the polynomial ring $R= {\mathbb Q}[z,  u_0,  u_1,  u_2, x_2,  x_3] $ and an ideal $I$, generated by polynomials $\{g_0, \, g_1, \, g_2, $ $ g_3, \, g_4, z \,  u_0 \, u_1 \,u_2 \, x_2 \, x_3-1 \}$.   We take a lexicographic ordering $>$  with $ z >  u_0 >  u_1>  u_2 > x_2 > x_3$ for a monomial ordering on $R$. Then,   a  Gr\"obner basis for the ideal $I$, contains a  polynomial  of   $x_3$  given by 
$({x_3}-1) (884 {x_3}^3-1816  {x_3}^2+873 {x_3}-117) h_1(x_3)$, where
$h_1(x_3)$ is a polynomial  of degree 101 defined by
{ \small 
 \begin{equation*} 
\begin{array}{l} h_1(x_3)=63459125312728809061842327355883965092426940416 {x_3}^{101} \\ -920983389392901147130377725860474388294846644224 {x_3}^{100}\\+9284030996847000229461748615161458859581745659904 {x_3} ^{99} \\ -72231264824781521074460759465312440396973648904192 {x_3}^{98}\\+474090092783696430917122803302705511377312635944960 {x_3}^{97} \\ -2716426925303009081280894916587505354693138471452672 {x_3}^{96} \\+13928981666688388147256261223934160302674603191304192 {x_3}^{95}  \\ -64892996424133928137698638603811273070013563879292928 {x_3}^{94}\\ +277731488039641471269947873533553174033919791213838336{x_3}^{93}  \\ -1100681274916513923824472334739425330176822351407087616{x_3}^{92}\\ +4063212316482458197610601184936117632052406776092426240{x_3}^{91}  \\ -14035457932119745594854722553431991322029237116190326784{x_3}^{90}\\ +45521154694910054415489734540334666446613275447418421248{x_3}^{89}  \\ -138989344903939502760451623904236457989017703097989857280{x_3}^{88}\\ +400294158959755643083550785378988938431734533883781709824{x_3}^{87}  \\ -1089022114669507770392870584714647593690007896934880641024{x_3}^{86}\\ +2801397856363975282367542926720029558144403125321202401280{x_3}^{85}  \\ -6817885379922156536111652515165617977589266615874129231872{x_3}^{84}\\ +15701952412543110505856916943820752724521223983962586152960{x_3}^{83}  \\ -34216393233291025770029652161256199220193387714745792659456{x_3}^{82}\\ +70516001329117409764548923930872952867194021910385077846016{x_3}^{81}  \\ -137327646289169111738793490244260987895509460573594545225728{x_3}^{80}\\ +252415013953372334056500075226916999570602665659035875999744{x_3}^{79}  \\ -437145843020198931924876541636921539338330946745999868755968{x_3}^{78}\\ +711698478115149919191326763960685722556799467931964286992384{x_3}^{77}  \\ -1085854704561151226820526323366031976152058255866699795660800{x_3}^{76}\\ +1545872987530402035886865292388450940956043862142555058380800{x_3}^{75}  \\ -2040770782009497991746315379021685834308361348137207206174720{x_3}^{74}\\ +2474617817759971607277241788391185316015028171144823683063808{x_3}^{73}  \\ -2713416472372073086525101694784351591143139473401445118681088{x_3}^{72}\\ +2613262983470191771589830160198704851628875439784801973596160{x_3}^{71}  \\ -2069290218418158264029847032239247119549407157373813654286336{x_3}^{70}\\ 
  +1073308692563330553384597653734241243314674928772408635447296{x_3}^{69}\\ +243383268162901260317977997414016371083655623812595679834112{x_3}^{68}  \\ -1609375404546416218236600482252912174118464079287975942948352{x_3}^{67}\\ +2672694963662822735196784395673190550399134349252116960776704{x_3}^{66}  \\ -3114412610690042189775494926465007145115374286782216606602752{x_3}^{65}\\ +2780945528894563748863142343650355721268277028914024232564224{x_3}^{64}  \\ -1777803153232519896953172291143481639817296527070092107145472{x_3}^{63}\\ +471376850935213838891432653962426924940360769359294745810432{x_3}^{62}\\ +623750383609419878697829019463867520660227080289808518313920{x_3}^{61}  \\ -1040603038326989861022950362038120196827146648905065209736576{x_3}^{60}\\  
 +564965083609062720849342591878835952132818055988824262111376{x_3}^{59}\\ +649619372601281147814040345731771930965779727246024293143504{x_3}^{58}  \\ -2118012753182384601180559819198400780972530222436309420548080{x_3}^{57}\\ +3205388926087602866149736984715502422941036840916282804140976{x_3}^{56}  \\ -3399079409763580385139579223928130084953706063347065837023000{x_3}^{55}\\ +2526218612878420387036879439950686812518557153001711837979064{x_3}^{54}  \\ -834433401752454857645647262023568632488618740772182979146772{x_3}^{53}  \\ -1101385395838477492629746207491960938518647539595102389316872{x_3}^{52}\\ +2631037139614283301382941342160092654160835574746224378985565{x_3}^{51}  \\ -3282383226236719186391399204082397795775877578301500715196503{x_3}^{50}\\ +2937659835203687649893255383322025931399696823727244558286097{x_3}^{49}  \\ -1851035734616115365629019524331368234044251593409402913444411{x_3}^{48}\\ +498022109054373150395850625480317180124853810659501884348651{x_3}^{47}\\ 
+627355526093354022331902293368985581146006710909004658401207{x_3}^{46}  \\ -1207045734980737886266791122600422887565869937071216149025129{x_3}^{45}\\ +1182073307241943169778530232935546805223288647965360710571511{x_3}^{44}  \\ -718203246621926363352199146463822746343025839992293567678988{x_3}^{43}\\ +96644002741648960970123374601314631036176300051445193949656{x_3}^{42}\\ +422104069724220785760650059880175853561682549627138288223924{x_3}^{41}  \\ -683397637477216946107242310612442756271548207816068470401088{x_3}^{40}\\ 
\end{array}
    \end{equation*} 
 \begin{equation*} 
\begin{array}{l}
+666605501872783619736009469651934045513775273586944371856186{x_3}^{39}  \\ -452558710835995517519348198954994152927678719275198879680886{x_3}^{38}\\ +162700999764271444384776912405378431410116935139824677075998{x_3}^{37}\\ +93939119293706672636919053860282946725625620077832878556914{x_3}^{36}  \\ -255147629760439396399382896402394067054928161056011639874616{x_3}^{35}\\ +308396407398907558760838372721158434444550638676652132806780{x_3}^{34}   \\ -277614375770280461907766589637738690110973121068605430180588{x_3}^{33}\\ +201477143418161317806510553487993326097416754880917183892864{x_3}^{32}   \\ -115969901716816418101426110321015198051414624540044967024168{x_3}^{31}\\ +44725833423519690527789412859895224361666978693724364777728{x_3}^{30}\\ +2520894187976416901089600003713269486821645786597713677964{x_3}^{29}   \\ -26486403386554915622730996860937772397953223304069070744920{x_3}^{28}\\ +33294891175701376903563757255573808063712434139603578954691{x_3}^{27}   \\ -30202076884961887362828779340403368100144991071060222187421{x_3}^{26}\\ +23099423829744618419506899206593906790694209676039929614503{x_3}^{25}   \\ -15680468836461732445911867905683122953673941068247037316197{x_3}^{24}\\ +9671984604335042749855869480653845526481625108821701237745{x_3}^{23}  \\ -5490069835615675038150115079186954852016192157885614205023{x_3}^{22}\\ +2889374986224761685442552566105705408525169491706393088757{x_3}^{21} \\ -1416494130573647866724791573304034482931984659037915662955{x_3}^{20}\\ +648751477566144492286603717146642412127675425476232549029{x_3}^{19} \\ -278070848833909095692181583970081277475757069716551459039{x_3}^{18}\\ +111644780007072437002497603341623817366814844891773288945{x_3}^{17} \\ -41997637999280362408255038869439084693844361764781722551{x_3}^{16}\\ +14796365048687965062327228402196888338897351487134760371{x_3}^{15} \\ -4877819943619247140179803803393490676071239723633511361{x_3}^{14}\\ +1502397252705966074311686966159213579712203646293106491{x_3}^{13} \\ -431425769917898886316566505773997593134379146416278993{x_3}^{12}\\ +115173309907631982466727249550402897704325044395195191{x_3}^{11} \\ -28477674526413418871108871989347987750566901612750701{x_3}^{10}\\ +6490427377397192429051997305835939685976063245759899{x_3}^{9} \\ -1355038458126617854787910001663192486402952805121845{x_3}^{8}\\ +257038525794588371891326953346269064271241145105863{x_3}^{7} \\ -43821959983623391384006935940560441466837308645513{x_3}^{6}\\ +6615718858884967733806801498204423491247659359667{x_3}^{5} \\ -865800811288635528014727851709253657063913408253{x_3}^{4}\\ +95113645509127999488018808217880353172800724072{x_3}^{3} \\ -8317620756384461283345771094408286533995673140{x_3}^{2}\\ +522260976842192193467930459675557760528309268{x_3}  \\ -17941225061011393318805831702298275346720564. 
    \end{array}
    \end{equation*}
}

 Solving $h_1(x_3)=0$ numerically, we see that there exist five positive solutions, which are given approximately by $x_3 \approx 0.25594917, x_3 \approx 0.49280351, x_3 \approx 1.1060677, x_3 \approx 1.3849054, x_3 \approx 2.4753269$. Moreover,  real solutions of the system   $\{ g_0 =0, g_1=0, g_2=0, g_3=0, g_4=0, h_1(x_3)=0 \}$ with $ u_0 \, u_1 \,u_2 \, x_2 \, x_3 \neq 0$ are of the form
 \begin{equation*} 
 \begin{array}{l}
  \{  u_0 \approx 0.26967359, u_1 \approx 0.21126932, u_2 \approx 0.11447898, x_2 \approx 1.0039269, x_3 \approx 0.25594917 \},
 \\
\{ u_0 \approx 0.54675598, u_1 \approx 0.28246829, u_2 \approx 1.1715986, x_2 \approx 1.022763119985955, x_3 \approx 0.49280351 \}, 
\\
\{  u_0 \approx 0.72660638, u_1 \approx 0.143286728, u_2 \approx 0.11340863, x_2 \approx 0.49841277, x_3 \approx 1.1060677 \}, 
\\
\{ u_0 \approx 1.3613346, u_1 \approx 1.4496103, u_2 \approx 0.15582596, x_2 \approx 0.95415861, x_3 \approx 1.3849054 \},  
\\
\{u_0 \approx 2.4948349, u_1 \approx 0.25221774, u_2 \approx 0.17749081, x_2 \approx 1.7461504, x_3 \approx 2.4753269 \}. 
    \end{array}
    \end{equation*}
Hence, we obtain five Einstein metrics which are non-naturally reductive  by Proposition \ref{prop5.5}. 
By  computing the associated scale invariants  (cf.  \cite{Chry2}) we conclude that  
these five metrics are non-isometric each other.

For $884 {x_3}^3-1816  {x_3}^2+873 {x_3}-117 = 0 $, we see that $ x_2= 1$, $u_0 = u_1 = x_3$ 
and $442 {x_3}^2-739 x_3+65 u_ 2+144=0$ are solutions of the system  $\{ g_0 =0, g_1=0, g_2=0, g_3=0, g_4=0 \}$. 
Hence, in this case solutions are given by 
\begin{equation*} 
 \begin{array}{l}
\{ u_0 =  u_1  = x_3  \approx 0.23910517,\ u_2 \approx 0.11429253,  \ x_2 =1 \}, \\ 
\{ u_0 =  u_1  = x_3  \approx 0.38779156, \ u_2 \approx 1.1709075,  \  x_2 =1 \}, \\ 
\{ u_0 =  u_1  = x_3  \approx 1.4274019, \ u_2 \approx 0.15823885,  \  x_2 =1 \}.
\end{array}
\end{equation*}

For $ x_3= 1$, we see that  $({x_2}-1) (2375 {x_2}^3-4195 {x_2}^2+1960 x_2-272) =0$, $u_0 = u_2= x_2$ and $102 u1-2375 {x_2}^3+6570 {x_2}^2-5305 x_2+1008 =0$. Then we obtain again solutions,  approximately given by 
 \begin{equation*} 
 \begin{array}{l}
\{ u_0  = u_2 = x_2 \approx 0.27971768, \ u_1 \approx 0.13559746,  \ x_3 =1 \}, \\ 
\{  u_0  = u_2 = x_2 \approx 0.36506883, \ u_1 \approx 1.6532011,  \ x_3 =1 \}, \\ 
\{  u_0  = u_2 = x_2 \approx 1.1215293, \ u_1 \approx 0.27621689,  \ x_3 =1 \} 
\end{array}
\end{equation*} 
and $ u_0 = u_1 = x_2 = x_3 = 1$. 
Notice that these solutions define   naturally reductive Einstein metrics, by Proposition \ref{prop5.5}.

\medskip
   
\noindent{\bf \underline{Case of $\E_7(5)$}}
\begin{equation}\label{eine75}
 \left\{ \begin{tabular}{ll}
 $g_0 =$& $5 {u_0} {u_1} {x_2}^2
   {x_3}^2+3 {u_0} {u_1} {x_2}^2+10
   {u_0} {u_1} {x_3}^2-15 {u_1}^2
   {x_2}^2 {x_3}^2-{u_1}^2 {x_2}^2-2
   {x_2}^2 {x_3}^2 =0$,\\\\  
 $g_1 =$& $15 {u_1}^2 {u_2}
   {x_2}^2 {x_3}^2+{u_1}^2 {u_2}
   {x_2}^2-8 {u_1} {u_2}^2 {x_2}^2
   {x_3}^2-4 {u_1} {u_2}^2 {x_3}^2-6
   {u_1} {x_2}^2 {x_3}^2+2 {u_2}
   {x_2}^2 {x_3}^2 =0$, \\\\  
 $g_2 =$&  ${u_0} {u_2}
   {x_2}^2 {x_3}+9 {u_1} {u_2}
   {x_2}^2 {x_3}+104 {u_2}^2 {x_2}^2
   {x_3}+24 {u_2}^2 {x_3}+36 {u_2}
   {x_2}^3 {x_3}+6 {u_2} {x_2}^3$ \\
 & $ -216
   {u_2} {x_2}^2 {x_3}+6 {u_2}
   {x_2} {x_3}^2-6 {u_2} {x_2}+36
   {x_2}^2 {x_3}=0$,\\\\  
 $g_3 =$& $-{u_0} {x_2}^2
   {x_3}+4 {u_0} {x_3}-9 {u_1}
   {x_2}^2 {x_3}-56 {u_2} {x_2}^2
   {x_3}+56 {u_2} {x_3}-72 {x_2}^3
   {x_3}-18 {x_2}^3$ \\
 & $+216 {x_2}^2
   {x_3}+6 {x_2} {x_3}^2-144 {x_2} 
  {x_3}+18 {x_2}= 0$,\\\\
$g_4 =$ & $9 {u_0} {x_2}^2-4 {u_0} {x_3}^2+9 {u_1} {x_2}^2-56
   {u_2} {x_3}^2+36 {x_2}^3
   {x_3}^2+102 {x_2}^3 {x_3}$  \\
 & $ -216 {x_2}^2 {x_3}-102 {x_2}
   {x_3}^3+144 {x_2} {x_3}^2+78
   {x_2} {x_3}=0$. 
   \end{tabular} \right.
 \end{equation}
   We consider the polynomial ring $R= {\mathbb Q}[z,  u_0,  u_1,  u_2, x_2,  x_3] $ and an ideal $I$ generated by polynomials $\{g_0, \, g_1, \, g_2, $ $ g_3, \, g_4, z \,  u_0 \, u_1 \,u_2 \, x_2 \, x_3-1 \}$.   Fix a lexicographic ordering $>$,  with $ z >  u_0 >  u_1>  u_2 > x_2 > x_3$ for a monomial ordering on $R$. Then, a  Gr\"obner basis for the ideal $I$ contains a  polynomial  of   $x_3$  given by 
$({x_3}-1) ( 2332 {x_3}^3-4013 {x_3}^2+1053 x_3-72) h_1(x_3)$, where
$h_1(x_3)$ is a polynomial  of degree 101 given by 
{ \small 
 \begin{equation*} 
\begin{array}{l} 
 h_1(x_3)= 1407625033454098518064986243441658839583396911621093750000
   0000000000
   {x_ 3}^{101}  \\ -171879184796389482519846089317782509140
   670299530029296875000000000000
   {x_ 3}^{100} \\ +145313195810420524500555185389570281512
   7745270729064941406250000000000
   {x_ 3}^{99}  \\ -9473331387374978930557135592382675403647
   411614656448364257812500000000
   {x_ 3}^{98} \\ +5222842903272600890435696339447708175573
   9115923643112182617187500000000
   {x_ 3}^{97}  \\ -2515439510033906451541094833423024544136
   19741813838481903076171875000000
   {x_ 3}^{96} \\ +1085613000952484788032182156567205573264
   792393599331378936767578125000000
   {x_ 3}^{95}  \\ -4258882459359618640576916999859198229608
   694138852670788764953613281250000
   {x_ 3}^{94} \\ +1535338179220109316926144274572859997446
   9277102204543948173522949218750000
   {x_ 3}^{93}  \\ -5124767143797420002927992624219035075094
   1601745823506236076354980468750000
   {x_ 3}^{92} \\ +1592779681935318277337709289386967923945
   19161440710672831535339355468750000
   {x_ 3}^{91}  \\ -4629265083461444987427288702488935923748
   79475272611933228969573974609375000
   {x_ 3}^{90} \\
    +1262177999422453138852307438542980370000
   112170664719322771072387695312500000
   {x_ 3}^{89}  \\ 
   -3236435366700326861977880348547854838802
   610550418218812595033645629882812500
   {x_ 3}^{88} \\ +7818839674582226915745074206106944841667
   485955194252036326932907104492187500
   {x_ 3}^{87}  \\ -1782302647568357498992075672784991499282
   8758644347139452949285984039306640625
   {x_ 3}^{86} \\ +3837456246457382049135240596363338184869
   2969526613948014654363346099853515625
   {x_ 3}^{85} \\
     -7811151515668095220856832281528364097115
   0078626459790696200255393981933593750
   {x_ 3}^{84} \\ +1504182128662493927464301361364771903121
   82131346708641783848342323303222656250
   {x_ 3}^{83}  \\ -2742267977992201360688296538498164470181
   58152571233876100911887187957763671875
   {x_ 3}^{82} \\ +4736664363713521739859761763382573823011
   03274684463812827026091655731201171875
   {x_ 3}^{81}  \\ -7759378481573667569060630844484222791340
   97874371188256677871034072875976562500
   {x_ 3}^{80} \\
   +1207150400502827505609484743337103792488
   617773663813184312931088278198242187500
   {x_ 3}^{79} \\ 
-1786917539761664393734512028432217575971
   928747001709362696207536544036865234375
   {x_ 3}^{78} \\ +2523283328042500705555771937323177208536
   480253814622741933239530885467529296875
   {x_ 3}^{77}  \\ -3410040133101191376307418153187605141034
   005485155297795948203643118286132812500
   {x_ 3}^{76} \\ +4427227829007830039323985782362477540309
   186880924655617991484124562622070312500
   {x_ 3}^{75}    \\
    -5543315902885880050459899218485938422135
   193995167025320528375071687469482421875
   {x_ 3}^{74} \\ +6716145578402019256272307762678559235656
   688619049543258249088075475286865234375
   {x_ 3}^{73}  \\ 
    \end{array}
    \end{equation*} 
 \begin{equation*} 
\begin{array}{l}
-7889325415521622316319117014762959306907
   667871100870016410896831429394531250000
   {x_ 3}^{72} \\ +8987871554017934014648949314880532895050
   225239219525645069037648161762695312500
   {x_ 3}^{71}  \\ -9917208157477374337767877875545519619987
   011928976120543819610153102255859375000
   {x_ 3}^{70} \\
  +1057779920369831962038745559233774102700
   1238708759415658121587619127446289062500
   {x_ 3}^{69}  \\ -1088731978221505090194964448983111262502
   5397307880553428016972363473247070312500
   {x_ 3}^{68} \\ 
   +1080827271867176488040930732463294698725
   0904801571361749421569584932820507812500
   {x_ 3}^{67}  \\ 
  -1035192575061247524162549318187841585103
   6122298454696642632323147793814306640625
   {x_ 3}^{66} \\ 
   +9570049235629322030129441573282410555362
   644512739858159923675452999522548828125
   {x_ 3}^{65}  \\ -8527728410757032303006877739950121077903
   047743081938287293482897667003828125000
   {x_ 3}^{64} \\ +7293671555087344683446136873087134513658
   482869052893430821710819055271882812500
   {x_ 3}^{63}  \\ -5938364201902229269149097008903216497897
   458637908470944612521462398610144531250
   {x_ 3}^{62} \\ +4542803163100520326399144842143300663460
   846346791886417852196110264118186718750
   {x_ 3}^{61}  \\ -3203367797363298249711802033233747850433
   277119218375865387587239202504776562500
   {x_ 3}^{60} \\ +2014757550707134984844606694758564217815
   440179528736200907390078758647693437500
   {x_ 3}^{59}  \\ -1059435189706540868544396172374688428356
   034633803064566575279632528066663046875
   {x_ 3}^{58} \\ +3734537469533921134934775217232581231759
   87786047445920972626980195467362484375
   {x_ 3}^{57} \\ +4453036221051807335321639649274824031301
   3722867507760033999699035992575625000
   {x_ 3}^{56}  \\ -2455513116256111000162106198030114921359
   65095443877459657793044352671972512500
   {x_ 3}^{55} \\ +2897080587113708245118481944852599927613
   49222486398117737364263666816009706250
   {x_ 3}^{54}  \\ -2510215104220945760347042051621298232030
   80952220679044063099876553064746298750
   {x_ 3}^{53} \\ +1769741809807208837022028692194111574894
   56320399570126363647370079444394210000
   {x_ 3}^{52}  \\ -1072934257491105524254817776078959188007
   03150070799670761477099264571385035000
   {x_ 3}^{51} \\ +5266110512997056259402821010863960672870
   1622938738221071821689744328675071625
   {x_ 3}^{50}  \\ -2020997691604380775174503965435885582353
   2177997588147118317671488142414778125
   {x_ 3}^{49} \\ +2385419029486295607532750803216893332123
   501222180408286275540947211555452000
   {x_ 3}^{48} \\ +3587821317546678618380217860464824345029
   642368821605269921381443889756146300
   {x_ 3}^{47}  \\ -5067035401021094823238464500969069646022
   963960363015867883320197337727649300
   {x_ 3}^{46} \\ +3717083826554658487879634379063790156238
   599455962131271814245768935578843000
   {x_ 3}^{45}  \\ -2427908795670998175497265078142249597119
   437614024384136832425990803168150860
   {x_ 3}^{44} \\ +1195740741369749034292959741455193545682
   210186321696447733419655241996493300
   {x_ 3}^{43}  \\ -5479030023364986664649764059184980180008
   43427855146404566963006900577627145
   {x_ 3}^{42} \\ +1685743230277189355797472565385565053992
   57352625388808341195401828470961037
   {x_ 3}^{41}  \\ -3165649270926734393334670727262004850799
   2779716620141360266989747001709088
   {x_ 3}^{40}  \\ -1914146377926607606046469696392226561756
   7912033126227666771741940037097492
   {x_ 3}^{39} \\ +2205780519622561622203822076978400342166
   4886692244997020731817495990602040
   {x_ 3}^{38}  \\ -1611110335681748334222433039874564630170
   5527199879085417642376254012031316
   {x_ 3}^{37} \\ +8787371168033807690276772563251134119884
   622902813540153994393586109768596
   {x_ 3}^{36}  \\ -4203136839615520864011534300079602987537
   250849136320058974663974205829364
   {x_ 3}^{35} \\ +1703272099354575963895402631669026840927
   858600614506299961357771268340029
   {x_ 3}^{34}  \\ -5889708743927881827096363244046561700144
   08669935999543313372542768494777
   {x_ 3}^{33} \\ +1538066099238323268236047694320150349210
   34937840431892763207363982129000
   {x_ 3}^{32}  \\ -1537847071578040584956036616184077903674
   7529936249773865278808437342372
   {x_ 3}^{31}  \\ -1521191948519739107963572266481714158282
   6291606688169501432074860152694
   {x_ 3}^{30} \\ +1460288472642585438759553298310024396907
   2684905051532652335375162018618
   {x_ 3}^{29}  \\ -8742436435221198517785178525746417053733
   167289492701877870544777319580
   {x_ 3}^{28} \\ +4322817710270412971587127477714461997494
   242365934113140183874504735076
   {x_ 3}^{27}  \\ -1894810327329607561294394972277312203013
   169914051239516934551367169545
   {x_ 3}^{26} \\
 +7592072008123351201808041249315530873960
   41536632050200202305980399381
   {x_ 3}^{25}  \\ -2824255116392047878725460531831747703978
   88747384260698976130175319380
   {x_ 3}^{24} \\ +9837787002559908171063760849337228152821
   1731859183399640998528484952
   {x_ 3}^{23}  \\ -3227308456867535524040908781969025865618
   3993127476521438269949061881
   {x_ 3}^{22} \\ +1000189548315273268401949436959539522575
   0217998389286530858888305093
   {x_ 3}^{21}  \\ -2934440317716499911008283722824705734798
   918191533548524562464809626
   {x_ 3}^{20} \\ +8158864970788104006944133430830081360019
   65170879611485375752520798
   {x_ 3}^{19}  \\ -2150136292403177345923129013208916526016
   83163107174766915835738002
   {x_ 3}^{18} \\ +5369415525375707036477183937541549399616
   3225978928643803111746886
   {x_ 3}^{17} \\
 -1269375479380148438574080671261868057864
   1191192421808450522707324
   {x_ 3}^{16} \\ +2836871266322051379178111812447045352149
   407143579909976950226088
   {x_ 3}^{15}  \\ -5982209091325806162329732673310346362991
   34258811161962399579373
   {x_ 3}^{14} \\ +1187290569901301091407944506457598541586
   05184235508731658294357
   {x_ 3}^{13}  \\ -2211226023346013592735720118341345664726
   2866287675952659637328
   {x_ 3}^{12} \\ +3850068285596588964841282201703973992205
   738072480521413341576
   {x_ 3}^{11}  \\ -6238372874697553568628267306190258219477
   27433644178976756360
   {x_ 3}^{10} \\ 
    \end{array}
    \end{equation*} 
 \begin{equation*} 
\begin{array}{l}
+9354843836484139280769398239228207317040
   3921008788886330336
   {x_ 3}^9  \\ -1289111526960146091059897128746823887288341
   2219299739419408
   {x_ 3}^8 \\ +1617872450840116983877791561306647646149713
   087425933514704
   {x_ 3}^7  \\ -1827750853526793542057087904079820607669879
   38597053694592
   {x_ 3}^6 \\ 
  +1829141810808924707180215084524494095223981
   4288051333120
   {x_ 3}^5  \\ -1584767080810168250170572570688441271916400
   990819270656
   {x_ 3}^4 \\ +1148809700130537873075491563920129377010596
   89924067328
   {x_ 3}^3  \\ -6569798492576356809368814625199214133359843
   746512896
   {x_ 3}^2 \\ +2637296039599514329860357599354551483247082
   03626496
   {x_ 3}  \\ -554614008146390030150597978844287609670452065
   0752 
     \end{array}
    \end{equation*}
    }
     Solving $h_1(x_3)=0$ numerically, we get five positive and four negative solutions, given approximately by $x_3 \approx 1.1800573, x_3 \approx 0.12169301, x_3 \approx 0.20754861, x_3 \approx 1.5303652, x_3 \approx 2.1692738$ and $x3 \approx -0.49217418,  x3 \approx -0.48405841, x3 \approx -0.57150512, x3 \approx -0.6231710$.  In addition,  we see that  real solutions of the system $\{ g_0 =0, g_1=0, g_2=0, g_3=0, g_4=0, h_1(x_3)=0 \}$ with $ u_0 \, u_1 \,u_2 \, x_2 \, x_3 \neq 0$, are   given by 
 \begin{equation*} 
 \begin{array}{l}
  \{  u0 \approx 0.96224102, u1 \approx 0.073621731, u2 \approx 0.24880488, x3 \approx 1.1800573, x2 \approx 0.66994762 \},
 \\
\{ u0 \approx 0.12590407, u1\approx 0.10843941, u2 \approx 0.24541305, x3 \approx 0.12169301, x2 \approx 1.0005214 \}, 
\\
\{  u0 \approx 0.22329332, u1 \approx 0.15361011, u2 \approx 1.1356290, x3 \approx 0.20754861, x2 \approx 1.0023024 \}, 
\\
\{ u0 \approx 1.5711912, u1 \approx  1.3666215, u2 \approx  0.31577380, x3 \approx  1.5303652, x2 \approx  1.1132523 \},  
\\
\{u0 \approx  2.1869813, u1 \approx  0.10309808, u2 \approx  0.33851600, x3 \approx  2.1692738, x2 \approx  1.5911459 \} 
    \end{array}
    \end{equation*}
    and 
     \begin{equation*} 
 \begin{array}{l}
  \{ u0 \approx 0.77867456, u1 \approx 0.18058466, u2 \approx 1.0076703, x3 \approx -0.49217418, x2 \approx 2.8011127 \},
 \\
\{  u0 \approx 0.75556174, u1 \approx 0.56064249, u2 \approx 0.94565380, x3 \approx -0.48405841, x2 \approx 2.8687188 \}, 
\\
\{ u0 \approx 1.0814568, u1 \approx 0.76564465, u2 \approx 0.48544174, x3 \approx -0.57150512, x2 \approx 3.1347988 \}, 
\\
\{ u0 \approx 1.2615164, u1 \approx 0.13352863, u2 \approx 0.44039770, x3 \approx -0.62317102, x2 \approx 3.1494041 \}.  
  \end{array}
    \end{equation*}

Thus, we obtain five Einstein metrics which are non-naturally reductive  by Proposition \ref{prop5.5}. 
We can also see that these five metrics are non-isometric each other, by computing the induced scale invariants (cf. \cite{Chry2}).

For $2332 {x_3}^3-4013 {x_3}^2+1053 x_3-72 = 0 $, we see that $x_2= 1$, $u_0 = u_1 = x_3$ 
and $2332 {x_3}^3-6345 {x_3}^2+4642 x_3-224 u_2-405=0$ are solutions of the system of equations $\{ g_0 =0, g_1=0, g_2=0, g_3=0, g_4=0 \}$. 
Thus,  approximately  we obtain the following  solutions of the system corresponding to the homogeneous Einstein equation:
\begin{equation*} 
 \begin{array}{l}
\{ u_0 =  u_1  = x_3  \approx 0.11699044,\ u_2 \approx 0.24536236,  \ x_2 =1 \}, \\ 
\{ u_0 =  u_1  = x_3  \approx 0.18615305, \ u_2 \approx 1.1352348,  \  x_2 =1 \}, \\ 
\{ u_0 =  u_1  = x_3  \approx 1.4176970, \ u_2 \approx 0.30406198,  \  x_2 =1 \}. 
\end{array}
\end{equation*}

For $ x_3= 1$, we see that  $({x_2}-1) (4949 {x_2}^3-9379 {x_2}^2+5155 x_2-875) =0$, $u_0 = u_2= x_2$ and $525 u_1-19796 {x_2}^3+57312 {x_2}^2-49561 x_2+11520 =0$. In this case,  we obtain the solutions
 \begin{equation*} 
 \begin{array}{l}
\{ u_0  = u_2 = x_2 \approx 0.37412457, \ u_1 \approx 0.069921978,  \ x_3 =1 \}, \\ 
\{  u_0  = u_2 = x_2 \approx 0.43525576, \ u_1 \approx 1.5741577,  \ x_3 =1 \}, \\ 
\{  u_0  = u_2 = x_2 \approx 1.0857500, \ u_1 \approx 0.12593450,  \ x_3 =1 \} 
\end{array}
\end{equation*} 
and $ u_0 = u_1 = x_2 = x_3 = 1$. 
By Proposition \ref{prop5.5},  we see that these solutions  induce  left-invariant Einstein metrics which are  naturally reductive.

\medskip

\noindent{\bf \underline{Case of $\E_8(2)$}}
\begin{equation}\label{eine82}
 \left\{ \begin{tabular}{ll}
 $g_0 =$ & $9 {u_0} {u_1} {x_2}^2
   {x_3}^2+3 {u_0} {u_1} {x_2}^2+18
   {u_0} {u_1} {x_3}^2-27 {u_1}^2
   {x_2}^2 {x_3}^2-{u_1}^2 {x_2}^2-2
   {x_2}^2 {x_3}^2=0$,\\\\  
 $g_1 =$ & $27 {u_1}^2 {u_2}
   {x_2}^2 {x_3}^2+{u_1}^2 {u_2}
   {x_2}^2-12 {u_1} {u_2}^2 {x_2}^2
   {x_3}^2-6 {u_1} {u_2}^2
   {x_3}^2-12 {u_1} {x_2}^2
   {x_3}^2+2 {u_2} {x_2}^2
   {x_3}^2=0$, \\\\  
  $g_2 =$ & ${u_0} {u_2} {x_2}^2
   {x_3}+9 {u_1} {u_2} {x_2}^2
   {x_3}+176 {u_2}^2 {x_2}^2
   {x_3}+36 {u_2}^2 {x_3}+60 {u_2}
   {x_2}^3 {x_3}+6 {u_2} {x_2}^3
   $ \\ 
   & $-360{u_2} {x_2}^2 {x_3}+6 {u_2}
   {x_2} {x_3}^2-6 {u_2} {x_2}+72
   {x_2}^2 {x_3}=0$,\\\\  
  $g_3 =$ & $-{u_0} {x_2}^2
   {x_3}+4 {u_0} {x_3}  -9 {u_1}
   {x_2}^2 {x_3}-104 {u_2} {x_2}^2
   {x_3}+104 {u_2} {x_3}-120 {x_2}^3
   {x_3}-18 {x_2}^3$ \\
   & $+360 {x_2}^2
   {x_3}+6 {x_2} {x_3}^2-240 {x_2}
   {x_3}+18 {x_2}=0$,\\\\ 
  $g_4 =$ & $9 {u_0} {x_2}^2-4
   {u_0} {x_3}^2+9 {u_1} {x_2}^2-104
   {u_2} {x_3}^2+60 {x_2}^3
   {x_3}^2+174 {x_2}^3 {x_3}-360
   {x_2}^2 {x_3}$ \\
   & $-174 {x_2}
   {x_3}^3+240 {x_2} {x_3}^2+150
   {x_2} {x_3}= 0$.
  \end{tabular} \right.
 \end{equation}
 
    Consider the polynomial ring $R= {\mathbb Q}[z,  u_0,  u_1,  u_2, x_2,  x_3] $ and an ideal $I$, generated by polynomials $\{g_0, \, g_1, \, g_2, $ $ g_3, \, g_4, z \,  u_0 \, u_1 \,u_2 \, x_2 \, x_3-1 \}$.   Fix the lexicographic ordering $>$,  with $ z >  u_0 >  u_1>  u_2 > x_2 > x_3$ for a monomial ordering on $R$. Then, by the aid of computer, we compute a  Gr\"obner basis for the ideal $I$; this contains a  polynomial  of   $x_3$  given by 
$({x_3}-1) ( 14863 {x_3}^3-23537 {x_3}^2+3841 x_3-159) h_1(x_3)$, where
$h_1(x_3)$ is a polynomial  of degree 101 given by 
{ \small 
 \begin{equation*} 
\begin{array}{l} 
 h_1(x_3)=
4300817068102634554892991062842129042773576732311618329635448320000000000 {x_3}^{101} \\  -48434234613549258526882893300617941462391685037263593180733084672000000000 {x_3}^{100} \\ +384412414000835153497962732335882305935605918425027227535381768849920000000 {x_3}^{99} \\-2343356739326681041821514580356528768679045141260294555070279790135552000000 {x_3}^{98} \\ +12118750392476326222236533006962244871827435009196838279324500368662097280000 {x_3}^{97} \\  -54694072490084700582663695979794196000771759437095459232706976031069894016000 {x_3}^{96}  \\ +221353224406881636197442133058784492973200290115121516088278948796937464371200 {x_3}^{95}  \\ -813886740762282999039919900837623187565903920555774696522128146040136921808640 {x_3}^{94}  \\ +2749311143142358208694757864933801593004701230356043145755440407202107969205248 {x_3}^{93}  \\ -8592602645047256285489822902989371514808504859611203756617492188201050618429696 {x_3}^{92}  \\ +24987213573707362099498581598734764124216513406921443771718443643063519416173424 {x_3}^{91}  \\ -67886371151784002113561327332159978566678061704085234326096771582310691912274832 {x_3}^{90}  \\ +172855472915888793901915966837849077301067956523296460430225355574530069733610000 {x_3}^{89}  \\ -413520597399357357100869381565342441018263038178115494696023037473035005867485760 {x_3}^{88}  \\ +931201134042478123783110021283671383867077991181025700457605569636368192762700536 {x_3}^{87}  \\ -1977180680461234875709924061155165455728674465172548132743172377066740066067927072 {x_3}^{86}  \\ +3963718711108383180414158119936397189175742439421867408680408165110894988557985004 {x_3}^{85}  \\ -7513418122634704515990854605187661012865820116011720424270638810049480025748301056 {x_3}^{84}  \\ +13485034115259604979202660060065426293023734909520325939447188838010611615994227899 {x_3}^{83}  \\ -22953711751263858356155615317993870971118178202495574682769731600848840969284513211 {x_3}^{82}  \\ +37118739850644510331943371356546784531091762451361904566881906562825838499125532570 {x_3}^{81}  \\ -57143381922307096358894776517481035852112895921488424717999575316225143862664779938 {x_3}^{80}  \\ +83933828799625827613068937568371537392889738211567188929892877463435790674461594289 {x_3}^{79}  \\ -117916849189812448657352362010079618375373892675617348197030706818467434459065258729 {x_3}^{78}  \\ +158840601295945663967771914515932889808801662497376498082318501776572544579687584532 {x_3}^{77}  \\ -205640403374951186570968875180984604118953468621671855171500403944532124984014864228 {x_3}^{76}  \\ +256373362974138627914650186250703319027584745315995489017313493311167581918499221209 {x_3}^{75}  \\ 
-308180457016417463546322721674440420609583883456596081505287450408132053995421688417 {x_3}^{74}  \\ +357437554714305319709984089082160479773294679519620215685747593181661184213216972570 {x_3}^{73}  \\ -399909639063653725636419295835615541787351965342966672935569109836991011645113298530 {x_3}^{72}  \\ +431447528558011049848674754343083702264157712308167972064641240229606934035503112751 {x_3}^{71}  \\ -448435605178939066244729947066284792252456509588375968250745873651107135576456816847 {x_3}^{70}  \\ +448909425142372532200712964875932843329902983017003297057000758414145225424316299356 {x_3}^{69}  \\ -432505986718923727726733642983121597245375455185536827443718365457823110316105606920 {x_3}^{68}  \\ +400926851409634138076302850835949788171629035117892773365279927750703254819903773328 {x_3}^{67}  \\ -356896733431404331148195671666424421693848148003049729639070229457919596548217497480 {x_3}^{66}  \\ +304117822268847097534981499298924198481044422812969119854167753513347346281757734984 {x_3}^{65}  \\ 
-246800083651937571501758366885734742867433138789154590922218902161270727737822408680 {x_3}^{64}  \\ +189342447332993406197869816861229081733498325675444469649008978221731371770480267524 {x_3}^{63}  \\ 
 -136571233453352364375287398041234776929197601714558372362384908184149357492228993124 {x_3}^{62}  \\ +91614174252541036183261171953169851489145245846377901063718187242760215516142261424 {x_3}^{61}  \\ -57099174743574573249126209346809266661445782155473938421471474074519986140555333552 {x_3}^{60}  \\ 
+32358254824131527171676058875969221067276651159847698655117241266482730946477408340 {x_3}^{59}  \\ -16824644814321478644964952093517021134010663065144141852879740223864734400865575220 {x_3}^{58} \\ +7514287187418446688509325329918072893001195320661832719856558857169013514235203720 {x_3}^{57}  \\ -2984410089768692604131117349590773663180032424321033953036482267179033266809725640 {x_3}^{56} \\ +732445676639314254349675396500675892268684514864442230285668149583939903578943676 {x_3}^{55}  \\ -67917208841309119819997191749712546424834785848702592652838825458694211721954412 {x_3}^{54}  \\ -201226807780484731518790106748244287589322324607269243036806074712412930566274184 {x_3}^{53} \\ +129154056812060487385840015218861595472761603600998499969817431810471728718352640 {x_3}^{52}  \\ -107834101343192331330480731011834920513318003452547815351643422543284825020233542 {x_3}^{51} \\ 
+39139302564982238988531563687231140402316126213769637000710966449760382422072550 {x_3}^{50}  \\
 -25014007394124148943263880664666605315368407336198528031642273872797438993869524 {x_3}^{49} \\ +5149637394602359328977892265587734288759270529004199895668361076640377995064036 {x_3}^{48}  \\ -3349820696141993438064008629889341109854136879528713117453652408803499479096386 {x_3}^{47}  \\ -85257700269450629092725979032910336876774470641715409238975393462391067569870 {x_3}^{46}  \\ 
  \end{array}
    \end{equation*} 
 \begin{equation*} 
\begin{array}{l}
-172220176540080050421320403630494784694130728150815944148121248793493036348296 {x_3}^{45}  \\ 
 -181152831869469320750602320397679974645876788987125838360682382259389590194584 {x_3}^{44} \\ +29933044025820518391774124363260016171392605814344766308091877183115582225214 {x_3}^{43}  \\ 
-40721449677473162458230185588830044223400555655697639872956766127283187956142 {x_3}^{42} \\ 
+7974104969338520813563262017073690821078239472420921364391270897806684138700 {x_3}^{41}  \\ -5189717122288769398584568108177516806684072008723652373050127514128137067260 {x_3}^{40} \\ 
+817800735737374024438923649637174077182654957307312155831702991153313172306 {x_3}^{39}  \\ -369906466718587873259473075984672842769304229881787413153149704347082720882 {x_3}^{38} \\ +12420361539059598078858560583923436698430183398779246053297377510620249848 {x_3}^{37}  \\ -185006606576738833950603767371099558647439473457636904506936645730880144 {x_3}^{36}  \\ -8953116075124778902766759252670482762033989180259875789742764174595473548 {x_3}^{35} \\ +3507574455024951317268532533764032607025789087442517730121154962348230396 {x_3}^{34}  \\ -1506654748146747600991905293586765340051270768797114324562291232201486232 {x_3}^{33} \\ +462728840082228859334694943477743198307650012420374125365389249712651768 {x_3}^{32}  \\ -131776588232594972255777649206567320338061471477811725806702054151290252 {x_3}^{31} \\ +31472043180361439617980833132544850110943883369080890722278786629107628 {x_3}^{30}  \\ -5961664481968065504823382003538804229336739136240908717770483137477456 {x_3}^{29} \\ +722351769246397427865270518182475924172222357378144113248455430696656 {x_3}^{28} \\ +93878343461212773918586759022334368944010651758103940471975728130612 {x_3}^{27}  \\ -95592681800341241277296233300832558302214353663993838838577883834868 {x_3}^{26} \\ +41634790191131063210888629765583357020845625933576181588238778182552 {x_3}^{25}  \\ -13936295068064242654729392915446150762330722439463395349506372368488 {x_3}^{24} \\ +4045182362993096546784545052702282647971953511355346978628439483220 {x_3}^{23}  \\ -1062701894734956175602327175315458674030959500282038965423982676924 {x_3}^{22} \\ +256459649784050971207284569212441996943675040503103180406710637212 {x_3}^{21}  \\ -57709734724845648912090174522151135714278853267038668156157868992 {x_3}^{20} \\ +12153005083455119529122284311415304647462102060888120578436682307 {x_3}^{19}  \\ -2406401626414324343746563575869808824782931044900643158950599971 {x_3}^{18} \\ +449105840481193161884750567257028703367516412585554400171866890 {x_3}^{17}  \\ -78996265368429079118013549058436508446509452503993511834925650 {x_3}^{16} \\ +13109093052218561111137671541747428560864068581509706530641689 {x_3}^{15}  \\ -2049218622402956644350490453379175367227864675085518185017233 {x_3}^{14} \\ +301303415905142916449843588495664606433759693923619398291732 {x_3}^{13}  \\ -41580761868772328901842590178904706369026472084206008451876 {x_3}^{12} \\ +5365113040262943479433997705054180067036308260891222155505 {x_3}^{11}  \\ -644693829313954760172498533221854992519656915601541723705 {x_3}^{10} \\ +71731828298864812802268603899355624836434020184867743210 {x_3}^{9}  \\ -7334795513210803706517436006266060058370643492181530930 {x_3}^{8} \\ +682989728878875810684537346120890928588745523322856055 {x_3}^{7}  \\ -57172697972619449741387104525586264686698343949422455 {x_3}^{6} \\ +4226158290506747020131443419536269019866825889091500 {x_3}^{5}  \\ -269121764902623042560625370550056308979867340869800 {x_3}^{4} \\ +14198857298392070342554145297476835831442968012100 {x_3}^{3}  \\ -578321988619101115121213409365880613402609688700 {x_3}^{2} \\ +15961340415131975059297267606639156809924912000 {x_3}  \\ -220720851022013939349538517195371169708198400 
   \end{array}
    \end{equation*} 
} 

Solving $h_1(x_3)=0$ numerically, we get five positive and four negative solutions, which are given approximately by $x_3 \approx 0.2205104, x_3 \approx 0.072293790, x_3 \approx 0.11492789, x_3 \approx 1.5437649, x_3 \approx 1.972915$ and $ x3 \approx -0.30929231, x3 \approx -0.30556009, x3 \approx -0.44683455,  x3 \approx -0.39992522$. Furthermore, we take the following real solutions of the system $\{ g_0 =0, g_1=0, g_2=0, g_3=0, g_4=0, h_1(x_3)=0 \}$ with $ u_0 \, u_1 \,u_2 \, x_2 \, x_3 \neq 0$: 
    \begin{equation*} 
 \begin{array}{l}
    \{  u0 \approx 1.0725485, u1 \approx 0.045012693, u2 \approx 0.31014510, x3 \approx 0.2205104, x2 \approx 0.75504389 \},
 \\
\{ u0 \approx 0.073985052, u1\approx 0.067061174, u2 \approx 0.30836161, x3 \approx 0.072293790, x2 \approx 1.0001177 \}, 
\\
\{  u0 \approx 0.12074935, u1 \approx 0.096083894, u2 \approx 1.0975908, x3 \approx 0.11492789, x2 \approx 1.0004421 \}, 
\\
 \{  u0 \approx 1.5886788, u1 \approx 1.3434091, u2 \approx 0.37367555, x3 \approx 1.5437649, x2 \approx 1.1394129 \},  
\\
 \{  u0 \approx 2.0094538, u1 \approx 0.057685208, u2 \approx 0.39443471, x3 \approx 1.9729151, x2 \approx 1.4757000 \} 
    \end{array}
    \end{equation*}
    and 
     \begin{equation*} 
 \begin{array}{l}
 \{  u0 \approx 0.57479844, u1 \approx 0.091481076, u2 \approx 1.0446277, x3 \approx -0.30929231, x2 \approx 2.2251780 \},
 \\
\{  u0 \approx 0.56438862, u1 \approx 0.57295621, u2 \approx  0.98344796, x3 \approx -0.30556009, x2 \approx 2.3226977 \}, 
\\
\{  u0 \approx 1.1366881, u1 \approx 0.71052398, u2 \approx  0.50096966, x3 \approx -0.44683455, x2 \approx 2.5636731 \}, 
\\
\{  u0 \approx 0.95640165, u1 \approx 0.80160322, u2 \approx  0.54149306, x3 \approx -0.39992522, x2 \approx 2.5715323 \}.
  \end{array}
    \end{equation*}
Thus we obtain five Einstein metrics which according to Proposition \ref{prop5.5},
 are non-naturally reductive. In particular, a computation of  the related scale invariants shows that these  five metrics are non-isometric each other.

\smallskip

For $ 14863 {x_3}^3-23537 {x_3}^2+3841 x_3-159 = 0 $, we see that $ x_2= 1$, $u_0 = u_1 = x_3$ 
and $44589 {x_3}^2-65894 {x_3}+2756 u_2+3573=0$ are  solutions of the system of equations $\{ g_0 =0, g_1=0, g_2=0, g_3=0, g_4=0 \}$. 
Thus,    approximately we obtain the following solutions:
\begin{equation*} 
 \begin{array}{l}
\{ u_0 =  u_1  = x_3  \approx 0.070481512,\ u_2 \approx 0.30834779,  \ x_2 =1 \}, \\ 
\{ u_0 =  u_1  = x_3  \approx 1.4050939, \ u_2 \approx  0.35656500,  \  x_2 =1 \}, \\ 
\{ u_0 =  u_1  = x_3  \approx 0.10802147, \ u_2 \approx 1.0974869,  \  x_2 =1 \}.  
\end{array}
\end{equation*}

For $ x_3= 1$, we see that  $(x_2-1) (864 {x_2}^3-1676 {x_2}^2+973 x_2-177) =0$, $u_0 = u_2= x_2$ and $59 u_1-3456 {x_2}^3+10160 {x_2}^2-9003 x_2+2240=0$. Thus, in this case the solutions approximately have the form 
 \begin{equation*} 
 \begin{array}{l}
\{ u_0  = u_2 = x_2 \approx 0.41888766, \ u_1 \approx 0.042734087,  \ x_3 =1 \}, \\ 
\{  u_0  = u_2 = x_2 \approx 0.46172446, \ u_1 \approx 1.5439133,  \ x_3 =1 \}, \\ 
\{  u_0  = u_2 = x_2 \approx 1.0592027, \ u_1 \approx 0.072250884,  \ x_3 =1 \} 
\end{array}
\end{equation*} 
and $ u_0 = u_1 = x_2 = x_3 = 1$. 
According to Proposition \ref{prop5.5}, these values define   naturally reductive Einstein metrics.

\medskip

\noindent{\bf \underline{Case of $\E_7(3)$}}
\begin{equation}\label{eine73}
 \left\{ \begin{tabular}{ll}
 $g_0=$ & $4 {u_0} {u_1} {x_2}^2
   {x_3}^2+6 {u_0} {u_1} {x_2}^2+8
   {u_0} {u_1} {x_3}^2-9 {u_1}^2
   {x_2}^2 {x_3}^2-{u_1}^2 {x_2}^2-3
   {u_1}^2 {x_3}^2-5 {x_2}^2
   {x_3}^2=0$,\\\\  
 $g_1=$ & $9 {u_1}^2 {u_2} {x_2}^2
   {x_3}^2+{u_1}^2 {u_2} {x_2}^2+3
   {u_1}^2 {u_2} {x_3}^2-10 {u_1}
   {u_2}^2 {x_2}^2 {x_3}^2-5 {u_1}
   {u_2}^2 {x_3}^2-3 {u_1} {x_2}^2
   {x_3}^2$ \\
   & $+5 {u_2} {x_2}^2 {x_3}^2=0$, \\\\  
 $g_2=$ & $2 {u_0} {u_2} {x_2}^2 {x_3}+108
   {u_1} {u_2} {x_2}^2 {x_3}+190
   {u_2}^2 {x_2}^2 {x_3}+75 {u_2}^2
   {x_3}+90 {u_2} {x_2}^3 {x_3}+30
   {u_2} {x_2}^3$ \\
   & $-540 {u_2} {x_2}^2
   {x_3}+30 {u_2} {x_2} {x_3}^2-30
   {u_2} {x_2}+45 {x_2}^2 {x_3}=0$,\\\\  
 $g_3=$ & $-{u_0} {x_2}^2 {x_3}+4
   {u_0} {x_3}-54 {u_1} {x_2}^2
   {x_3}+36 {u_1} {x_3}-20 {u_2}
   {x_2}^2 {x_3}+20 {u_2} {x_3}-90
   {x_2}^3 {x_3}-45 {x_2}^3 $ \\
   & $+270 {x_2}^2 {x_3}+15 {x_2}
   {x_3}^2-180 {x_2} {x_3}+45 {x_2}=0$,\\\\  
 $g_4=$ & $9 {u_0} {x_2}^2-4 {u_0}
   {x_3}^2+36 {u_1} {x_2}^2-36 {u_1}
   {x_3}^2-20 {u_2} {x_3}^2+45
   {x_2}^3 {x_3}^2+120 {x_2}^3
   {x_3}-270 {x_2}^2 {x_3}$ \\
   & $-120 {x_2}
   {x_3}^3+180 {x_2} {x_3}^2+60 {x_2} {x_3}= 0$.
  \end{tabular} \right.
 \end{equation}

     We consider the polynomial ring $R= {\mathbb Q}[z,  u_0,  u_1,  u_2, x_2,  x_3] $ and an ideal $I$ generated by polynomials $\{g_0, \, g_1, \, g_2, $ $ g_3, \, g_4, z \,  u_0 \, u_1 \,u_2 \, x_2 \, x_3-1 \}$.   We take a lexicographic order $>$  with $ z >  u_0 >  u_1>  u_2 > x_2 > x_3$ for a monomial ordering on $R$. Then, by the aid of computer, we see that a  Gr\"obner basis for the ideal $I$ contains a  polynomial  of   $x_3$  given by 
$({x_3}-1) ( 5632 {x_3}^{3}-9488 {x_3}^{2}+3933 x_3-477 ) h_1(x_3)$, where
$h_1(x_3)$ is a polynomial  of degree 119 given by 
{ \small 
 \begin{equation*} 
\begin{array}{l} 
 h_1(x_3)= 5317991353240733131727570427468867047230653060663383949312000000 {x_3}^{119}  \\ -98019036075297848183056077552432780274599684871346934691921920000 {x_3}^{118} \\ +1098759293323246546512110892440119473621419358215417540389109760000 {x_3}^{117}  \\ -9305718974821085677811148533918240248548281811944475222438051840000 {x_3}^{116} \\ +65113268802402200478525426791012973944464656086732606877047193600000 {x_3}^{115}  \\ -394046712276679250169937971915477424725853689736022406260651458560000 {x_3}^{114} \\ +2120477184719184687909529467574362503309040056101340761096036810752000 {x_3}^{113}  \\ -10334275693315162730742501425160945657985705715529777725552751280128000 {x_3}^{112} \\ +46202883626048272893214883998659314525038880033265154243257408946176000 {x_3}^{111}  \\ -191297437298409715833973441970431226167739110188661910161889358249984000 {x_3}^{110} \\ 
+738780827716860155253065617261688708737616255470467254037109394531942400 {x_3}^{109}  \\ -2676181837866157465124121869252517828887974363206276139033897785517670400 {x_3}^{108} \\ +9133426483052267815004033114279867587137539131902245605998876572093644800 {x_3}^{107}  \\ -29472676795039732970734547702915016051055469216731276093351250213864755200 {x_3}^{106} \\ +90184928909873771799432761661136216708184513107610597609746891377263136000 {x_3}^{105}  \\ 
 \end{array}
    \end{equation*} 
 \begin{equation*} 
\begin{array}{l}
-262307357426217645783615015266291229953598598109678467351599564732054385920 {x_3}^{104} \\ +726607074049122059949903497786269863005806535278198855839594575897076134400 {x_3}^{103}  \\ 
-1920000384806180026768867421031537805803949214018610412439630758504547658240 {x_3}^{102} \\ 
+4846060182856939686502282874159690356258654662377278202700184434834229251840 {x_3}^{101} \\
 -11695698788657658011476017683648425807762896347480114232710059819716386479872 {x_3}^{100}  \\
   +27013496275542415222023898874133963150112893916727310080078230054731729365760 {x_3}^{99} \\ -59749103051104069900082654913774362932116508404875695431707764026079324402176 {x_3}^{98} \\ +126611218508152514171021884078842994748237775495659584068454223721873984155136 {x_3}^{97} \\ -257106642098986647794761688311621745956940627647648459595343771802548540196608 {x_3}^{96} \\ +500353563387757953315580846341686490346251719009013776065028881870121108834304 {x_3}^{95} \\ 
-933016705955937135119323562999416823231028982454515912424262648992194752998400 {x_3}^{94} \\  +1666350818986961443923020156264509852302084853480440742528762549779805507678176 {x_3}^{93} \\  -2848340507587261881920978915930434665052734871687834518943803961285485435643680 {x_3}^{92} \\+4654627782602904770609612602896237363606246929547882870907718816433560227299552 {x_3}^{91} \\ -7260233976254603124125288557777967078459922242578361722097381607395252017834656 {x_3}^{90} \\+10784574134608484614276985456547825052341992329805230146179268661318496034215328 {x_3}^{89} \\ -15206890525973623513320114356674689405486194792925228453628670976495944887480416 {x_3}^{88}  \\ +20259564142565286382067849026583726403797529586011137954566489403735710614317808 {x_3}^{87} \\ -25322646140778716665367209582629589229402613718695019077543030437619026934511872 {x_3}^{86}  \\ +29361049740649058190095422773938113133727049612048263846703870604726429273818848 {x_3}^{85} \\ 
-30958898262828372238095893876477609469256161228209064040807927251062574932231808 {x_3}^{84}  \\ +28503379824617763124097418982546938583969974303980617091962899726958833174324912 {x_3}^{83} \\ 
-20543542214535775343928825028111253860312177959979652483009700253869835271874960 {x_3}^{82}  \\ +6294369642269549673426137151386593173604138265575910629636206811805707511761437 {x_3}^{81}  \\ +13818211675329783251259102108459155176182048091914822307750576223359914637453455 {x_3}^{80} \\ -37746690309720320985994632077313901348434569764528336593925529654080198675646244 {x_3}^{79}  \\ +61753673833125499831493137549382712686215192881835369866733354263253711754859516 {x_3}^{78} \\ -80854547668572429281210675119560892341208614027928965716233593125981815435023888 {x_3}^{77}  \\ +89812864504052969401079860636375036763636864288507524282286470644623557615081560 {x_3}^{76} \\ -84535429159155233862550289641436464123003887379615811860052338807579720883023478 {x_3}^{75}  \\ +63510492448268670372687644118857216784764228069599171336461180087553652873279182 {x_3}^{74} \\ -28806449551758167167381070887475541786044250380199710546870177374097227663790102 {x_3}^{73} \\ -13819200006296826887332553242838278598966190774297960615356684080259858578141754 {x_3}^{72}  \\ +55928883492126546344671899768119435072107772811080555472486301876729279728981672 {x_3}^{71} \\ -88395953903308372639893176053856383836415322538976509138114310058609164273031496 {x_3}^{70}  \\      +103856916225563001905855479279765329540729798592811495080146581136794192997376655 {x_3}^{69} \\ -98858852486177981613232255698962149696211353827260075743006756056647230069954747 {x_3}^{68}  \\ +74958489545439966999218456704579034920402987004206042834929843114050283410125302 {x_3}^{67} \\ -38345022907819988966927460815044040501482214859561550856915205551605926851820930 {x_3}^{66} \\ -1933984189289290909758871181978501503101517835043559405349470423047513359646551 {x_3}^{65}  \\ +36563363991331046122778540325816464357200982292566309581429772964642137060912551 {x_3}^{64} \\ -58503023292998810962908820458611734395532683946103963774344297544010645359207942 {x_3}^{63}  \\ +64667016802341655818291153908789200846667212133580648087527559877011677639859826 {x_3}^{62} \\ -56276689024449602807134567146798786876912510628461995409738619384328995479952536 {x_3}^{61}  \\ 
+37923246804712273090895791414049387753363649314958327356686885503364596103010544 {x_3}^{60} \\ 
-15804099822546624431889611254391986315665298657063854127105952207976547865842410 {x_3}^{59} \\ -4200208277200018663918651182754824696675095403130261450762898272570211599919886 {x_3}^{58}  \\ +18003614435967420980166907884494599849844980028338457204276354810013230173559043 {x_3}^{57} \\ -23996289375134293702397518464953017416246840639996724092934265805096201432712959 {x_3}^{56}  \\ +22896421430919467671417906469523160708148809783127242961452861967501009283931308 {x_3}^{55} \\ -16995162401109429557563518044417613234180033412772287866423551699626340777728236 {x_3}^{54}  \\ +9149147822823806906962230588662742746287438437686110908864974694553824509909820 {x_3}^{53} \\ -1869742082865858170225132674874943447999722991925364389597345372662327003166296 {x_3}^{52} \\ -3251030792816307243173986937505264238031030686149502911284928922438610811256614 {x_3}^{51}  \\ +5695431308223163606750558298931569422162344146029502538889329771578111276950458 {x_3}^{50} \\ -5821933601002230572963625899395935531203363265386106574852162725576728360729418 {x_3}^{49}  \\ +4474341067570798992753095703878403797977587900230422629718901911972829744054802 {x_3}^{48} \\ -2576535585553216941235505881748794931817753435988641400362962157269000059779618 {x_3}^{47}  \\ +843529225509532660779460709613447818199390135853982049123805773210824449139958 {x_3}^{46}  \\ +340595998345860120135567539574048861325777859683069626803906775976914933436704 {x_3}^{45} \\ -900249258061313444009645965975640288401133468845786027012643996236702493508180 {x_3}^{44}  \\ 
\end{array}
    \end{equation*} 
 \begin{equation*} 
\begin{array}{l}
+963559193738019288439599749329178881976923800367915767314219049684509727035118 {x_3}^{43} \\ -740277516467837260839729452407794530488339855141041598796887742435105433205290 {x_3}^{42}  \\ +426598693884268067338770371748908806365114744547975422291326227423958316393441 {x_3}^{41} \\ -154344454413030230195558495324921592627026175652284039931753566289870331238957 {x_3}^{40} \\ -18885641734696226473355081221670109230833882865922496011983680467999251729610 {x_3}^{39}  \\ +92892730203535900594667159886630284670524914382364037000390898844638421946062 {x_3}^{38} \\ 
-97845830445709246884644720769627109615510112764829332200297118664141758644699 {x_3}^{37}  \\ +69959142495111919393279706463185659777899040344087434117780022465957676438599 {x_3}^{36} \\ -36656988921005907726196449218481987931866142444159602442626409986332511730722 {x_3}^{35}  \\ 
+11970618337580002956766486464244718621348103877552988866424400460263954584070 {x_3}^{34}  \\ +1040976609337647949468550831667972659771119324734367564216165811934766725081 {x_3}^{33} \\ -5200598435234017820505208745149033546910570946107973493933702513891175352873 {x_3}^{32}  \\ 
+4753327116203770112084797783068716505097074618408061967453779229506733606706 {x_3}^{31} \\ -2927038165335938086853904991790396160894433995898356814034484966458012844294 {x_3}^{30}  \\ 
+1343092453302416775833184680845511498973033819341236028253064406432275375821 {x_3}^{29} \\ -429090583194986620166032640192537465928491354854617432007488220581026924349 {x_3}^{28}  \\ +50203978196716392445791374937845914764500892727351051576948994601806057256 {x_3}^{27}  \\ +46158080222791719528724103649073653365386485473469418723802372444031145104 {x_3}^{26} \\ -41003170223247363091599963775495621392415719665053564221553740094953432948 {x_3}^{25}  \\ +19285883832633296542887968256983510105494652603418248603315795245263457344 {x_3}^{24} \\ -5634188094551462720145386313127480084110640040599162247381032217495719512 {x_3}^{23}  \\ +475504463177526413957760735598869270023707338832121518609174521427627560 {x_3}^{22}  \\ +558986006147223809653953006323757972748961682009357527454919823835136791 {x_3}^{21} \\ 
-399619501134062850592562711231620683906906119687234950479946264083966699 {x_3}^{20}  \\ +155786600706317593080732257843590930097562320551848816986555178225045864 {x_3}^{19} \\ -37306475891729024014762407510734957928120463380785652234775685168349684 {x_3}^{18}  \\ +2187089992872666480014105137597260558103254357667821151374732139513218 {x_3}^{17}  \\ +2818112964763938831003560438163327584479611740100456207329307087403146 {x_3}^{16} \\ -1614594187941944814891286660729588683748976829917613597170905151176782 {x_3}^{15}  \\ +516151068910408155820399739685397453980831540964931713471262497972746 {x_3}^{14} \\ -100829000245539286655843015478331689172231142856179692701741937186305 {x_3}^{13}  \\ +3368364183716110092469752017751431110250756244416851974618979321701 {x_3}^{12}  \\ +6572883557813607115063681865239951263598799030394109795800045060474 {x_3}^{11} \\ -3283450397388770774006401566576559204621165267840621179138287123670 {x_3}^{10}  \\ +1031319099761786646251537681910532174270190841045644984130431961954 {x_3}^{9} \\ -249680291106143280241018835564758843320175661225821177424525292822 {x_3}^{8}  \\ +49212547709312655403731992925601954539383684164719382796569460742 {x_3}^{7} \\ -8026521807779078099496238258304727304087936835093634106534385898 {x_3}^{6}  \\ +1081478431569654626009258053539930201185578706540437961929852917 {x_3}^{5} \\ -118566734244541430886739082169060500770838134263466432240586949 {x_3}^{4}  \\ +10256500407063673515218118277406679603598463165413692852750954 {x_3}^{3} \\ -661996321847218796068287208154089929196932463764640942695594 {x_3}^{2}  \\ +28572142588965250828016623923248449622787682754546590053837 {x_3} \\ -623744015652591601670205734762122870831994162006221990413. 
   \end{array}
    \end{equation*} 
} 

Solving $h_1(x_3)=0$ numerically, we obtain seven real  solutions  which are given   by $x_3 \approx 1.1800573, x_3 \approx 0.49280351, x_3 \approx 1.1060677, x_3 \approx 1.3849054, x_3 \approx 2.4753269$ (we state only the positive). As a consequence, real solutions of the system of equations $\{ g_0 =0, g_1=0, g_2=0, g_3=0, g_4=0, h_1(x_3)=0 \}$ with $ u_0 \, u_1 \,u_2 \, x_2 \, x_3 \neq 0$ are of the form
    \begin{equation*} 
 \begin{array}{l}
  \{ u0 \approx 0.30587680, u1 \approx 0.23162043, u2 \approx 0.11719295, x2 \approx 1.0035307, x3 \approx 0.27827971 \},
 \\ 
  \{  u0 \approx 0.43465453, u1 \approx 0.27733727, u2 \approx 1.4182653, x2 \approx 1.0086185, x3 \approx 0.37945991 \},
 \\
\{ u0 \approx 0.33445150, u1 \approx 0.23695076, u2 \approx 0.36978513, x2 \approx 0.31241976, x3 \approx 1.0008636 \}, 
\\
\{  u0 \approx 0.28679936, u1 \approx 0.36605764, u2 \approx 0.14958786, x2 \approx 0.28763468, x3 \approx 1.0026185\}, 
\\
\{ u0 \approx 0.77541704, u1 \approx 0.19715742, u2 \approx 0.11437270, x2 \approx 0.52666358, x3 \approx 1.0826430 \},  
\\
\{ u0 \approx 1.5820396, u1 \approx 0.30622692, u2 \approx 1.3221125, x2 \approx 1.2303151, x3 \approx 1.3552648 \}, 
\\
\{ u0 \approx 2.3846395, u1 \approx 0.30253103, u2 \approx 0.17015362, x2 \approx 1.6249173, x3 \approx 2.2246116 \}.
 
    \end{array}
    \end{equation*}
Thus, we obtain seven Einstein metrics which are non-naturally reductive  by Proposition \ref{prop5.5}. 
In particular, these  seven metrics are non-isometric each other and this follows after a computation of the corresponding   scale invariants.

\smallskip

  For $({x_3}-1) ( 5632 {x_3}^{3}-9488 {x_3}^{2}+3933 x_3-477 ) = 0$
 For $  5632 {x_3}^{3}-9488 {x_3}^{2}+3933 x_3-477  = 0 $, we see that $ x_2= 1$, $u_0 = u_1 = x_3$ 
and $1408 {x_3}^2-1948 x_3+53 u_2+387=0$ for solutions of the system of equations $\{ g_0 =0, g_1=0, g_2=0, g_3=0, g_4=0 \}$. 
Hence,  in this case we conclude that the following parameters define solutions of the homogeneous Einstein equation: 
\begin{equation*} 
 \begin{array}{l}
\{ u_0 =  u_1  = x_3  \approx 0.24536236,\ u_2 \approx 0.11699044,  \ x_2 =1 \}, \\ 
\{ u_0 =  u_1  = x_3  \approx 0.30406198, \ u_2 \approx  1.4176970,  \  x_2 =1 \}, \\ 
\{ u_0 =  u_1  = x_3  \approx 1.1352348, \ u_2 \approx 0.18615305,  \  x_2 =1 \}.
\end{array}
\end{equation*}

For $ x_3= 1$, we see that  $(x_2-1) (7 x_2-2) =0$, $u_0 =u_1 =  u_2= x_2$. Thus we obtain  a solution,   given by 
 \begin{equation*} 
 \begin{array}{l}
\{ u_0  =  u_1 = u_2 = x_2  = 2/7,  \ x_3 =1 \}, \\ 
\end{array}
\end{equation*} 
and $ u_0 = u_1 = x_2 = x_3 = 1$. 
Using Proposition \ref{prop5.5} we deduce that the induced left-invariant    Einstein metric are naturally reductive.

    \section{Left-invariant non-naturally reductive Einstein metrics on Lie groups of Type $III_{b}(3)$}\label{TypeIII}
\subsection{The Lie groups $\E_6(3)$.}  In this final section we examine  the Lie group $\E_6(3)$, which is the unique compact simple Lie group $G=G(i_{o})$ of Type $III_{b}(3)$, see Theorem \ref{class}.  Let us denote its Lie algebra by $\fr{g}$. 
Consider the  orthogonal decomposition 
\begin{equation}\label{III3}
\fr{g}=\fr{k}_{0}\oplus\fr{k}_{1}\oplus\fr{k}_{2}\oplus\fr{k}_{3}\oplus\fr{p}_{1}\oplus\fr{p}_{2}\oplus\fr{p}_{3}=\fr{m}_{0}\oplus\fr{m}_{1}\oplus\fr{m}_{2}\oplus\fr{m}_{3}\oplus\fr{m}_{4}\oplus\fr{m}_{5}\oplus\fr{m}_{6}.
\end{equation}
with $\fr{k}_{0}\cong\fr{u}_{1}$, $\fr{k}_{1}\cong \fr{su}_{2}$, $\fr{k}_{2}\cong\fr{su}_{3}\cong\fr{k}_{3}$. Hence, and according to Table 2, it is $d_{0}=1$, $d_{1}=3$, $d_{2}=d_{3}=8$, $d_{4}=36$, $d_{5}=18$ and $d_{6}=4$.   A left-invariant metric  on $\E_6$ is given by
 \begin{eqnarray}
 \langle \ , \  \rangle&=&u_{0}\cdot B|_{\fr{k}_{0}}+u_{1}\cdot B|_{\fr{k}_{1}}+u_{2}\cdot B|_{\fr{k}_{2}}+u_{3}\cdot B|_{\fr{k}_{3}}+x_{1}\cdot B|_{\fr{p}_{1}}+x_{2}\cdot B|_{\fr{p}_{2}}+x_{3}\cdot B|_{\fr{p}_{3}}\nonumber\\
 &=&y_{0}\cdot B|_{\fr{m}_{0}}+y_{1}\cdot B|_{\fr{m}_{1}}+y_{2}\cdot B|_{\fr{m}_{2}}+y_{3}\cdot B|_{\fr{m}_{3}}+y_{4}\cdot B|_{\fr{m}_{4}}+y_{5}\cdot B|_{\fr{m}_{5}}+y_{6}\cdot B|_{\fr{m}_{6}}, \label{invIII3}
 \end{eqnarray}
 for some positive numbers $u_{i}, x_{j}, y_{m}\in\bb{R}_{+}$. This metric is also $\Ad(K)$-invariant and since $\fr{m}_{i}\ncong\fr{m}_{j}$ for any $3\leq i\neq j\leq 6$, all  $G$-invariant metrics on the base space $M=G/K$ are a multiple of
 \[
 ( \ , \ )=x_{1}\cdot B|_{\fr{p}_{1}}+x_{2}\cdot B|_{\fr{p}_{2}}+x_{3}\cdot B|_{\fr{p}_{3}}=y_{4}\cdot B|_{\fr{m}_{4}}+y_{5}\cdot B|_{\fr{m}_{5}}+y_{6}\cdot B|_{\fr{m}_{6}}.
 \]
In a similar way with Proposition \ref{F42},  we   conclude   that   the non-zero triples $A_{ijk}$ $(0\leq i, j, k\leq 6)$ associated to  the reductive decomposition (\ref{III3})   and   the left-invariant metric on  $\E_{6}$ given by (\ref{invIII3}),   are the following (up to permutation of indices): 
  \[
  A_{044}, \ A_{055}, \ A_{066}, \ A_{111}, \ A_{144},   \ A_{166} \  A_{222}, \ A_{244}, \  A_{255}, \  A_{333}, \ A_{344}, \ A_{355}, \ A_{445}, \ A_{456}.
  \]
 In particular, it is easy to see that $A_{155}=A_{266}=A_{366}=0$.
 \begin{remark}
 \textnormal{In the reductive decomposition (\ref{III3}) it is $\fr{k}_{2}\cong\fr{su}_{3}\cong\fr{k}_{3}$. This isomorphism  does not effect on  the behaviour of the Ricci tensor $\Ric_{\langle \ , \ \rangle}$  corresponding left-invariant metric $\langle \ , \ \rangle$. In particular,  by using root vectors corresponding to $\fr{k}_{2}$ and $\fr{k}_{3}$, it follows that $\Ric_{\langle \ , \ \rangle}(\fr{k}_{2}, \fr{k}_{3})=0$, hence $\Ric_{\langle \ , \ \rangle}$ is still diagonal.}
 \end{remark}
 
\subsection{The Ricci tensor and the structure constants}
Let us apply Lemma \ref{ric}    to get a first version of  the Ricci tensor   in terms of the parameters of $\langle \ , \ \rangle$, the dimensions $d_{i}$  and the non-zero triples $A_{ijk}$.
 \begin{prop}\label{ricg4}
The components ${r}_{i}$ of  the Ricci tensor $\Ric_{\langle \ , \ \rangle}$ associated to the left-invariant metric $\langle \ , \ \rangle$ on  $\E_6$  described  by (\ref{invIII3}), are given by
 \begin{equation*} 
\left\{\begin{array}{ll} 
r_0 &=  \displaystyle{\frac{u_{0}}{4d_0}\biggl(\frac{A_{044}}{{x_{1}}^{2}}+\frac{A_{055}}{{x_{2}}^{2}}+\frac{A_{066}}{{x_{3}}^{2}}
\biggr), } \quad 
r_1   = \displaystyle{\frac{A_{111}}{4d_{1}}\cdot\frac{1}{u_{1}} + 
\frac{u_{1}}{4d_1} \biggl(\frac{A_{144}}{{x_{1}}^{2}} +\frac{A_{166}}{{x_{3}}^{2}}\biggr),}
  \\ & \\
  r_2 &=  \displaystyle{\frac{A_{222}}{4d_{2}}\cdot\frac{1}{u_{2}} + 
\frac{u_{2}}{4d_2} \biggl( \frac{A_{244}}{{x_{1}}^{2}} +\frac{A_{255}}{{x_{2}}^{2}}\biggr),}\quad 
 r_3=  \displaystyle{\frac{A_{333}}{4d_{3}}\cdot\frac{1}{u_{3}} + 
\frac{u_{3}}{4d_3} \biggl( \frac{A_{344}}{{x_{1}}^{2}} +\frac{A_{355}}{{x_{2}}^{2}}\biggr),}
 \\ & \\
r_4 &=  \displaystyle{\frac{1}{2 x_{1}} -
\frac{1}{2d_4{x_{1}}^{2}}
\biggl(u_{0}\cdot A_{044}+u_{1}\cdot A_{144}+u_{2}\cdot A_{244}+u_{3}\cdot A_{344}+x_{2}\cdot A_{445} \biggr) 
 +  \frac{A_{456}}{2d_4}  
\biggl(\frac{x_{1}}{x_{2} x_{3}} - \frac{x_{2}}{x_{1} x_{3}}- \frac{x_{3}}{x_{1} x_{2}}
\biggr),  }
\\ & \\
r_5 & = \displaystyle{\frac{1}{2 x_{2}}   -\frac{1}{2d_{5}{x_{2}}^{2}}\bigg(u_{0}\cdot A_{055}+u_{2}\cdot A_{255} +u_{3}\cdot A_{355}\biggr)
+\frac{A_{445}}{4d_{5}}\biggl( \frac{x_{2}}{{x_{1}}^{2}} - \frac{2}{x_{2}}\biggr)
  +\frac{A_{456}}{2d_{5}}
\biggl(\frac{x_{2}}{x_{1} x_{3}} - \frac{x_{1}}{x_{2} x_{3}}- \frac{x_{3}}{x_{1} x_{2}}
\biggr), }
\\ & \\
r_6 & = \displaystyle{\frac{1}{2 x_{3}} -  
\frac{1}{2d_6} \cdot\frac{1}{{x_{3}}^{2}}\biggl(u_{0}\cdot A_{066}+u_{1}\cdot A_{166}\biggr)+\frac{A_{456}}{2d_{5}}
\biggl(\frac{x_{3}}{x_{1} x_{2}} - \frac{x_{1}}{x_{2} x_{3}}- \frac{x_{2}}{x_{1} x_{3}}
\biggr). } 
\end{array}
\right.
\end{equation*}
 \end{prop}
We need now the values of the non-zero $A_{ijk}$. These are described by the following lemma.
\begin{lemma}\label{A456}
 For the reductive decomposition (\ref{III3}) and for the left-invariant metric $\langle \ , \ \rangle$ on the Lie group  $E_{6}=\E_6(3)$,    the non-zero $A_{ijk}$ are given explicitly  as follows:
\begin{eqnarray*}
&& A_{044}=1/4, \ A_{055}=1/2, \ A_{066}=1/4, \ A_{111}=1/2, \ A_{144}=9/4, \ A_{166}=1/4, \\
&& A_{222}=A_{255}=A_{333}=A_{355}=2, \  A_{244}=A_{344}=4, \ A_{445}=6, \ A_{456}=3/2.
\end{eqnarray*}
\end{lemma}
 \begin{proof}
 The computation of $A_{445}$ and $A_{456}$  is based on the unique  K\"ahler-Einstein metric $y_4=1, y_5=2, y_6=3$ that $M=G/K=\E_6/(\U_1\times\SU_2\times\SU_3\times\SU_3)$ admits.   Thus, the system which defines the Killing metric $y_{i}=1$ $(i=0, \ldots 6)$, consists now of six equations and 12 unknowns.  For the construction of more equations, let us consider first the twistor fibration of $M=G/K$ over an irreducible symmetric space. Set 
 \[
 \fr{g}=\fr{h}\oplus\fr{n}, \quad \fr{h}:=\fr{h}_{1}\oplus\fr{h}_{2}, \quad \fr{h}_{1}:=\fr{k}_{0}\oplus\fr{k}_{2}\oplus\fr{k}_{3}\oplus\fr{p}_{2}, \quad \fr{h}_{2}:=\fr{k}_{1}\cong\fr{su}_{2}, \quad \fr{n}:=\fr{p}_{1}\oplus\fr{p}_{3}.
 \]
 This is a symmetric reductive decomposition of $\fr{g}$; for dimensional reasons we take $\fr{h}_{1}\cong\fr{su}_{6}$ and  the corresponding irreducible symmetric space $G/H$ is the coset $\E_6/(\SU_6\times\SU_2)$. Consider a left-invariant metric  on $\E_6$, given by $\langle\langle \ , \ \rangle\rangle=v_1\cdot B|_{\fr{h}_{1}}+v_2\cdot B|_{\fr{h}_{2}}+v_3\cdot B|_{\fr{n}}$, for some $v_1, v_2, v_3\in\bb{R}_{+}$.   For $v_1=u_0=u_2=u_3=x_2$, $v_2=u_1$, $v_3=x_1=x_3$ this metric coincides with  $\langle  \ , \ \rangle$ and the same holds for the corresponding Ricci tensors.  Hence we get the following equations:
  \begin{eqnarray*}
   d_{0}A_{244} -d_{2}(A_{044}+A_{066}) &=& d_{2}A_{344} -d_{3} A_{244}, \\
  -d_{2}A_{344}+d_{3} A_{244}, &=& d_{5}A_{344}(d_{4}+4d_{5}+9d_{6})-4d_{3} d_{4} d_{6} -d_{3}d_{4}d_{5}- d_{3}d_{5}d_{6}, \\ 
    d_{0}(A_{222}+A_{255}) - 
    d_{2}A_{055} &=&d_{2}(A_{333} + A_{355}) - d_{3}(A_{222}+ 
     A_{255}),\\
      -d_{2}(A_{333} + A_{355})+d_{3}(A_{222}+ 
     A_{255}) &=&(d_{4} +4d_{5}+9d_{6})(2d_{3}(A_{055}+ A_{255}+ A_{355})+d_{5}(A_{333}+A_{355}))\\
     &&   -8d_{5}^{2}d_{3}  + 8 d_{3} d_{4} d_{6}  - 16 d_{3} d_{5} d_{6},\\
    0&=& (d_{4} + 4 d_{5} + 9 d_{6}) (d_{4}A_{166} - d_{6}A_{144}), \\
   0&=& d_{4}A_{066}(d_{4}+ 4d_{5}+9d_{6}) -d_{6}(A_{044}+A_{244}+A_{344})(d_{4}+4d_{5}) \\
   &&- 9d_{6}^{2}(A_{044}+A_{244}+A_{344}) -3 d_{4} d_{6}^2+ d_{4}^2 d_{6}.
    \end{eqnarray*}
    Set now
        \[
    \fr{g}=\fr{q}\oplus\fr{r}, \quad \fr{q}:=\fr{q}_{1}\oplus\fr{q}_{2}\oplus\fr{q}_{3}, \quad \fr{q}_{1}:=\fr{k}_{0}\oplus\fr{k}_{1}\oplus\fr{p}_{3}, \quad \fr{q}_{2}:=\fr{k}_{2}, \quad \fr{q}_{3}:=\fr{k}_{3}, \quad \fr{r}:=\fr{p}_{1}\oplus\fr{p}_{2}.
    \]
    It follows that $\fr{q}_{i}\cong\fr{su}_{3}$ for any $i=1, 2, 3$ and $[\fr{q}, \fr{q}]\subset\fr{q}$, $[\fr{q}, \fr{r}]\subset\fr{r}$. Since $\fr{k}\subset\fr{q}$, this defines the   fibration
    \[
    \bb{C}P^{2}=\SU_{3}/\U_{2}\to \E_6/(\U_1\times\SU_2\times\SU_3\times\SU_3)\to \E_6/(\SU_3\times\SU_3\times\SU_3),
    \]
   where the base space $G/Q\cong\E_6/(\SU_3\times\SU_3\times\SU_3)$ is isotropy irreducible, see \cite{Bes}.   By  a similar procedure as before and after considering   a new left-invariant metric on $\E_6$, we get that
        \begin{eqnarray*}
    -d_{0}A_{144}+d_{1}(A_{044}+ A_{055}) &=&  d_{6}A_{144}(d_{4}+4d_{5}+9d_{6}) - 2 d_{1}d_{6}(d_{4}+d_{5}),\\   
    -d_{0}(A_{111}+ A_{166}) +d_{1}A_{066} &=& 2d_{1}(A_{066}+ A_{166})(d_{4}+4d_{5}+9d_{6}) +d_{6}(A_{111}+A_{166})(d_{4} +4d_{5}+9d_{6})\\
    && + 2 d_{1} d_{4} d_{6}- 4 d_{1} d_{5} d_{6} - 18 d_{1}d_{6}^{2}.
    \end{eqnarray*}
   Now,  a combination  of these equations  together with the system defined by the Killing metric, shows that $A_{044}=A_{066}=A_{166}=1/4$,  $A_{055}=A_{111}=1/2$,  $A_{144}=9/4$, $A_{244}=A_{344}=4$  and 
 \[
   A_{222} =4 - A_{255}, \quad   A_{333} = A_{255},  \quad  A_{355} = 4 - A_{255}.
   \] 
 However, it is $A_{222}=c\cdot \dim\fr{su}_{3}$, where $c=B_{\SU_{3}}/B_{\E_6}=4/24$. Thus, $A_{222}=2$ and   we also get $A_{255}=2=A_{333}=A_{355}$. In fact, we can verify the values of $A_{111}, A_{333}$ as follows:
 $A_{111}=(B_{\SU_{2}}/B_{\E_6})\cdot\dim\fr{su}_{2}=1/2$ and $A_{333}=(B_{\SU_{3}}/B_{\E_6})\cdot\dim\fr{su}_{3}=2(=A_{222})$.
    \end{proof}
 
 \subsection{Naturally reductive metrics} 
 For a Lie group $G\cong G(i_{o})$ of Type $III_{b}(3)$, 
 left-invariant metrics  on $\G =\E_6(3)$  which are $\Ad(K)$-inavariant are given by
 \begin{equation}\label{invI36}
 \langle \ , \  \rangle=u_{0}\cdot B|_{\fr{k}_{0}}+u_{1}\cdot B|_{\fr{k}_{1}}+u_{2}\cdot B|_{\fr{k}_{2}}+u_{3}\cdot B|_{\fr{k}_{3}}+x_{1}\cdot B|_{\fr{p}_{1}}+x_{2}\cdot B|_{\fr{p}_{2}}+x_{3}\cdot B|_{\fr{p}_{3}}. 
 \end{equation}

\begin{prop}\label{prop6.4}
If a left invariant metric $\langle \ , \  \rangle$ of the form $(\ref{invI36})$  on $G =\E_6(3)$ of Type $III_{b}(3)$   is naturally reductive  with respect to $G \times L$ for some closed subgroup $L$ of $G$, 
then one of the following holds: 

$(1)$ $u_0 = u_2= u_3 =x_2 $,   $x_{1} = x_{3}$ \quad 
$(2)$   $u_0 = u_1 =x_3$,   $x_{1} = x_{2}$  
\quad 
$(3)$ $ x_{1} = x_{2} = x_{3} $.

Conversely, 
 if  one of the conditions $(1)$, $(2)$, $(3)$  holds, then the metric 
 $\langle \ , \  \rangle$ of the form  $(\ref{invI36})$ is  naturally reductive  with respect to $G \times L$ for some closed subgroup $L$ of $G$.
  \end{prop}
    \begin{proof}
   Let ${\frak l}$ be the Lie algebra of  $L$. Then we have either ${\frak l} \subset {\frak k}$  or ${\frak l} \not\subset {\frak k}$. 
First we consider the case of  ${\frak l} \not\subset {\frak k}$. Let ${\frak h}$ be the subalgebra of ${\frak g}$ generated by ${\frak l}$ and ${\frak k}$. 
Since 
$ \fr{g}=\fr{k}_{0}\oplus\fr{k}_{1}\oplus\fr{k}_{2}\oplus\fr{p}_{1}\oplus\fr{p}_{2}\oplus\fr{p}_{3}$ is an irreducible decomposition as $\mbox{Ad}(K)$-modules, we see that the Lie algebra $\frak h$  contains  at least one of  ${\frak p}_{1}$,  ${\frak p}_{2}$, ${\frak p}_{3}$. 
Let us start with the case $\fr{p}_{1}\subset\frak h$.
In this case, we get 
$\left[{\frak p}_{1}, {\frak p}_{1}\right] \cap {\frak p}_2 \neq\{0\}$ and hence $\fr{p}_{2}\subset\frak h$. Notice also that  $\left[{\frak p}_{1}, {\frak p}_{2}\right] \cap {\frak p}_3 \neq\{0\}$, i.e. $\frak h$ contains ${\frak p}_{3}$ as well. Therefore,  $\frak h = \frak g $  coincide, and  the $\Ad(L)$-invariant metric $\langle \ , \ \rangle$ of the form $(\ref{invI36})$ is bi-invariant. 
Now if $\fr{p}_{2}\subset\frak h$, then $\fr{ h} \supset {\frak k}\oplus {\frak p}_{2}$. If $\fr{ h} =\fr{k}\oplus {\frak p}_{2}$, then $(\fr{ h}, \fr{p}_{1}\oplus\fr{p}_{3})$ is a symmetric pair. Thus, the metric $\langle \ , \ \rangle$ of the form $(\ref{invI36})$ satisfies $u_0 = u_2 = u_3 =  x_2 $,   $x_{1} = x_{3}$. If $\fr{ h} \neq\fr{k}\oplus {\frak p}_{2}$, then it must be $\fr{ h}\cap {\frak p}_1\neq\{0\}$ or $\fr{ h}\cap {\frak p}_3\neq\{0\}$ and thus $\fr{ h} \supset {\frak p}_1$, or $\fr{ h} \supset {\frak p}_3$. Hence, we conclude again $\fr{h} = \fr{g}$ and the $\Ad(L)$-invariant metric $\langle \ , \  \rangle$ of the form $(\ref{invI36})$ must be bi-invariant. 
Finally, if $\frak h$  contains ${\frak p}_{3}$, then $\fr{ h} \supset {\frak k}\oplus {\frak p}_{3}$. If $\fr{ h} =\fr{k}\oplus {\frak p}_{3}$, then $\fr{h}$ is a semi-simple Lie algebra and ${\frak p}_{1}\oplus {\frak p}_{2}$ is an irreducible $\Ad(H)$-module. 
Hence, the metric $\langle \ , \ \rangle$ defined by $(\ref{invI36})$ satisfies the conditions $u_0 = u_1 = x_3 $,   $x_{1} = x_{2}$. 
If $\fr{h}\neq\fr{k}\oplus {\frak p}_{3}$, we similarly conclude that $\fr{h}\cap {\frak p}_1\neq\{0\}$ or $\fr{ h}\cap {\frak p}_2\neq\{0\}$. Thus it must be   $\fr{h} \supset {\frak p}_1$, or $\fr{ h} \supset {\frak p}_2$.  In the same way,  we obtain the identification $\fr{h} = \fr{g}$ and the $\Ad(L)$-invariant metric $\langle \ , \  \rangle$ given by $(\ref{invI36})$ is bi-invariant. 

Now,   consider the case ${\frak l} \subset {\frak k}$.  Since the  orthogonal complement
 ${\frak l}^{\bot}$ of ${\frak l}$ with respect to $B$ contains the  orthogonal complement 
${\frak k}^{\bot}$ of ${\frak k}$, we see that ${\frak l}^{\bot} \supset {\frak p}_{1} \oplus  {\frak p}_{2}\oplus  {\frak p}_{3}$.   
Since the  invariant metric $\langle \,\, , \,\, \rangle$ is naturally reductive  with respect to $G\times L$,  
 it follows that  $ x_{1} = x_{2} = x_{3} $, by  Theorem \ref{ziller}.   

Conversely, 
 if the condition $(1)$  holds, then  according to Theorem \ref{ziller} the metric 
 $\langle \ , \  \rangle$  given by  $(\ref{invI36})$ is  naturally reductive  with respect to $G \times L$, where $\fr{l} = \fr{k}\oplus {\frak p}_{2}$. If the condition $(2)$  holds, then the metric given by  $(\ref{invI36})$ is  naturally reductive  with respect to $G \times L$ where $\fr{l} = \fr{k}\oplus {\frak p}_{3}$. Finally, considering the  condition $(3)$, the  metric $\langle \ , \ \rangle$ defined by  $(\ref{invI36})$ is  naturally reductive  with respect to $G \times K$. This completes the proof.
  \end{proof}

  \subsection{The homogeneous Einstein equation}
Due to Proposition \ref{ricg4} and Lemma \ref{A456}, we are able now to describe explicitly the Einstein equation on $\E_6$ with respect to the left-invariant metric $\langle \ , \  \rangle$ given by (\ref{invI36}). 
 This has the form  (we normalise $\langle \ , \ \rangle$ by setting $x_1=1$):
\begin{equation}\label{eine63}
 \left\{ \begin{tabular}{ll}
 $g_0=$ &$3 {u_0} {u_1} {x_2}^2 {x_3}^2+3 {u_0} {u_1} {x_2}^2+6 {u_0} {u_1} {x_3}^2-9 {u_1}^2
   {x_2}^2 {x_3}^2-{u_1}^2 {x_2}^2-2 {x_2}^2 {x_3}^2=0$,\\\\  
 $g_1=$ &$9 {u_1}^2 {u_2} {x_2}^2 {x_3}^2+{u_1}^2 {u_2} {x_2}^2-6 {u_1} {u_2}^2 {x_2}^2 {x_3}^2-3 {u_1} {u_2}^2 {x_3}^2-3 {u_1} {x_2}^2 {x_3}^2+2 {u_2} {x_2}^2 {x_3}^2=0$, \\\\  
  $g_2=$ &$({u_2}-{u_3}) \left(2 {u_2} {u_3}
   {x_2}^2+{u_2} {u_3}-{x_2}^2\right)=0$,\\\\  
  $g_3=$ &${u_0} {u_3} {x_2}^2 {x_3}+9 {u_1} {u_3} {x_2}^2
   {x_3}+16 {u_2} {u_3} {x_2}^2 {x_3}+52 {u_3}^2 {x_2}^2 {x_3}+18 {u_3}^2 {x_3}+24 {u_3}
   {x_2}^3 {x_3}+6 {u_3} {x_2}^3$ \\&$-144 {u_3} {x_2}^2 {x_3}+6 {u_3} {x_2} {x_3}^2-6 {u_3}
   {x_2}+18 {x_2}^2 {x_3}= 0$,\\\\
   $g_4=$ &$-{u_0} {x_2}^2 {x_3}+4 {u_0} {x_3}-9 {u_1} {x_2}^2 {x_3}-16
   {u_2} {x_2}^2 {x_3}+16 {u_2} {x_3}-16 {u_3} {x_2}^2 {x_3}+16 {u_3} {x_3}-48 {x_2}^3
   {x_3}$ \\&$-18 {x_2}^3+144 {x_2}^2 {x_3}+6 {x_2} {x_3}^2-96 {x_2} {x_3}+18 {x_2}= 0$,\\\\
    $g_5=$ &$9 {u_0} {x_2}^2-4 {u_0} {x_3}^2+9 {u_1} {x_2}^2-16 {u_2} {x_3}^2-16 {u_3} {x_3}^2+24 {x_2}^3
   {x_3}^2+66 {x_2}^3 {x_3}-144 {x_2}^2 {x_3}$ \\&$-66 {x_2} {x_3}^3+96 {x_2} {x_3}^2+42 {x_2}
   {x_3}= 0$.
  \end{tabular} \right.
 \end{equation}

 Hence, here we need to separate our study in two cases: $u_2=u_3$ and $u_2\neq u_3$.
 
 \smallskip
 
 \noindent{\bf \underline{Case of $u_2=u_3$.}}
 
  Consider the polynomial ring $R= {\mathbb Q}[z,  u_0,  u_1,  u_2, u_3, x_2,  x_3] $ and the ideal $I$, generated by polynomials
$\{ g_0, \, g_1, \, g_2, \,u_2-u_3, \ g_4, \, g_5,  \,z \,  u_0 \, u_1 \,u_2 \, u_3 \, x_2 \, x_3-1\}$.   We take a lexicographic ordering $>$,  with $ z >  u_0 >  u_1>  u_2>  u_3  > x_2 > x_3$ for a monomial ordering on $R$. Then,   a  Gr\"obner basis for the ideal $I$, contains a  polynomial  of   $x_3$  given by 
$(x_3-1) (17 x_3 - 3) (105 {x_3}^2-180 x_3+43) h_1(x_3)$, where
$h_1(x_3)$ is a polynomial  of degree 101 explicitly defined as follows:
{ \small 
 \begin{equation*} 
\begin{array}{l}
 h_1(x_3)=
446201620912851710794055315043558727793812439040 {x_3}^{101} \\ -5970091094929099291107016178576455666569981198336 {x_3}^{100} \\+54956862112783810928181826738240314636476199469056 {x_3}^{99}  \\ -390811972141409570262352446028389195713538504523776 {x_3}^{98} \\+2345416005757619944748010845785735225833990004408320 {x_3}^{97}  \\ - 12296269949783674561034099631570563346793487654977536 {x_3}^{96} \\+57724805446986778564540591366311097869260812896485376 {x_3}^{95}  \\ - 246317806151607009947614038893190006132943659073929216 {x_3}^{94} \\
+965833450260597608670637228985038216735033541178353664 {x_3}^{93}  \\ - 3507391144656252588241576000834153338016240516439009280 {x_3}^{92} \\+11864289574806461749157687866946545721640388480080753216 {x_3}^{91} 
\end{array}
    \end{equation*} 
 \begin{equation*} 
\begin{array}{l}
 \\ - 37549398669469064618491404250349251298103313074223249536 {x_3}^{90} \\
+111556403718909975585042193985610935450985330327798025296 {x_3}^{89}  \\ - 311910444204478491525686781524742216649919870695292936208 {x_3}^{88} \\+822280356201276757927422346722674238555203030542292664576 {x_3}^{87}  \\
 - 2046807634124746479321033866947971592607665709250896466720 {x_3}^{86} \\
+4815217751927711346791066364373414448388380533686308362728 {x_3}^{85}  \\
 - 10713154223781970681173785420072062491556536536976223630728 {x_3}^{84} \\+22549161222430386183608267502748067727666572080394217257988 {x_3}^{83}  \\ - 44907345465673484752677835549915381616308978558104109929360 {x_3}^{82} \\+84618521132058565642549103143785317869429133164796407117101 {x_3}^{81}  \\ - 150844185455615567255483714719341206736546450798252724485465 {x_3}^{80} \\
+254364597537871570975284759118655801254332547323430302651150 {x_3}^{79}  \\ - 405743116891842974585208861357409219861363395354878349480024 {x_3}^{78} \\+612380460693768440871773749073945732272186880021261998877349 {x_3}^{77}  \\ - 875113963128364673685794098870412312567351427548775200004993 {x_3}^{76} \\+1185630843249395191547875465219341425290948445000653446938823 {x_3}^{75}  \\ - 1526156016359413728700771221981215402722504361600304503707071 {x_3}^{74} \\+1872140311833306055253112547604404320557478273540474816242205 {x_3}^{73}  \\ - 2197142616370071997752676286281654178936005418170556941864365 {x_3}^{72} \\+2477660626761403597845544417084070514622384049066587378036896 {x_3}^{71}  \\ - 2695279129437608065469390397221342766341704813704469316565954 {x_3}^{70} \\ 
 +2835433701981662882068588945651386104804048704393507162564893 {x_3}^{69}  \\ - 2884252164405306448117386421259048268986488851508923254464557 {x_3}^{68} \\
+2827426102309617415089093304345547856770953623242517510626736 {x_3}^{67}  \\ - 2653031260591475199015786134718569035816857413556039498720084 {x_3}^{66} \\
+2358669599996507371239530783357053968184783637965137255750024 {x_3}^{65}  \\
 - 1957846841895078664510239606394345202411065376438607044585244 {x_3}^{64} \\+1482575111628520660082626404685754443506000630228282037525432 {x_3}^{63}  \\
  - 977867340837642392768018881489052493366220469043864074527732 {x_3}^{62} \\+491766472172297444603170160453087191972251139268914030745220 {x_3}^{61}  \\ - 64289892666775037219425989232698561842251433476705352663704 {x_3}^{60}  \\ - 278945493495224974314286977888017233682365342439482314928900 {x_3}^{59} \\+526430447145328824729273526640420290370517101891583559600064 {x_3}^{58}  \\ - 678379950713096427271349581801112479268074670188010049765152 {x_3}^{57} \\+740804443402906260135865009076566078898436726677260064209660 {x_3}^{56}  \\ - 723437773277172187573344151721247865703874576057058786645840 {x_3}^{55} \\+641035988242594321731767126475011901128999801733089421531572 {x_3}^{54}  \\ - 512246964568992242392698713221172537445964128081045582227348 {x_3}^{53} \\+359435224624336029210539130176048914537249135075824824148656 {x_3}^{52}  \\ - 207440334442393198667039888458264495684556945767681966433896 {x_3}^{51} \\+75710441609060446524728615339027145373100155489714689483180 {x_3}^{50} \\+21604102072091727758542772876876306504838293754672061494810 {x_3}^{49}  \\ - 80248321816710686802873660978717285967176547645289112968438 {x_3}^{48} \\
+103727618132515579179402773133416119620715279258622028072932 {x_3}^{47}  \\ - 100868755076088189226682412281496766781111480676495550161140 {x_3}^{46} \\+82258536089640272288149509294477219539005077695487522749302 {x_3}^{45}  \\ - 57970183302251603419654961306043425453455813053820477148466 {x_3}^{44} \\
+34566800668606464343700458091411364463430345415525433898210 {x_3}^{43} \\ 
   - 16383644568171057304816744132324477326999822644552147091878 {x_3}^{42} \\
+4136127404028507683677730195509342689102464969628360294162 {x_3}^{41} \\+2448626876880142407515132852862076928913043337935695489050 {x_3}^{40} \\
 - 5091766004144148182025039566270136471217289325488871960720 {x_3}^{39} \\+5211357311516214470643681545481350157748897861517575943216 {x_3}^{38}  \\ - 4205038938600327672804921451310081054705566281185567057186 {x_3}^{37} \\+2880449505767857832367939657821857907905008739339817966542 {x_3}^{36}  \\ - 1722083524791071937275223955515869701707457252145375576992 {x_3}^{35} \\+883624776677143096193640273337726641926852894182302897924 {x_3}^{34}  \\ - 365844093248197813880318931085013840771534850915964839168 {x_3}^{33} \\+91100779597310425911980208315659416104663130431912801252 {x_3}^{32} \\+27847050547862877200320462747908632404588568320607377368 {x_3}^{31}  \\ - 62044289670560951692846051883485532537357245491675755676 {x_3}^{30} \\
 +58353887299451282295931377670775776235048712427538410548 {x_3}^{29}  
 \end{array}
    \end{equation*} 
 \begin{equation*} 
\begin{array}{l}
\\ - 42801139152078013889617819216603291344568555788701334096 {x_3}^{28} \\+27298429710753168330147448633665498685113651255938321004 {x_3}^{27}  \\ - 15754306291891017136056764020082795793908703387391423800 {x_3}^{26} \\
 +8385092507307451984906201504724689649995181449911987384 {x_3}^{25}  \\ - 4158959074188012109187108793403217630081388694570956516 {x_3}^{24} \\
 +1933988921439293195469245555839057789403000884600297488 {x_3}^{23}  \\ - 846254668982236036130963702540888323950320830947175956 {x_3}^{22} \\+349173361624441055757559121760687306723274054819456068 {x_3}^{21}  \\
 - 136002653968064858424211455167796115036642943021590992 {x_3}^{20} \\+50025097256033736148221913546357943308441010505896292 {x_3}^{19}  \\ - 17373295515964501643797136024825465698321434682820076 {x_3}^{18} \\+5693288809915569850120359913933561680515494371702321 {x_3}^{17}  \\
 - 1758699778656065182425515796649314364989143766487161 {x_3}^{16} \\+511400484439012231385964862157738050955388834378366 {x_3}^{15}  \\ - 139733354682245393321428617402456758035160677933092 {x_3}^{14} \\+35796740615382398885298967204849897416534790629933 {x_3}^{13}  \\ - 8574371869161530658987820760170185731254834290789 {x_3}^{12} \\+1913914136459751675483098693532003485746543254063 {x_3}^{11}  \\ - 396457437890074706831914824872045975693198088483 {x_3}^{10} \\+75819734083435665076464730202188368026690437353 {x_3}^{9}  \\ - 13299684783778380566640907764584095732685621109 {x_3}^{8} \\+2121856171516406265190177954088620136793636144 {x_3}^{7} \\ - 304497055879317382444174963304459359118141166 {x_3}^{6} \\+38715888461462869028868892744214611570993773 {x_3}^{5}  \\ - 4268028432749925091690300460592874748842233 {x_3}^{4} \\+394868683968983093689505159570909541213552 {x_3}^{3}  \\ - 29044225364799470649873570312140343108960 {x_3}^{2} \\+1526222320849667867688684374054814420480 {x_3} \\ - 43316614062737407568779520450405292288. 
     \end{array}
    \end{equation*}
}

 \smallskip
 
 Solving $h_1(x_3)=0$ numerically, we find five positive  and two negative solutions, which are given approximately by $x_3 \approx 0.18639713, x_3 \approx 0.34082479, x_3 \approx 1.1408077, x_3 \approx 1.4812096, x_3 \approx 2.3587740$ and  $x_3 \approx  -0.80052360, x_3 \approx -0.75956931$.  Moreover, real solutions of the system  $\{ g_0 =0, g_1=0, g_2=0, u_2 = u_3, g_4=0, g_5=0 \}$ with $ u_0 \, u_1 \,u_2 \, u_3 \, x_2 \, x_3 \neq 0$ have the form
 \begin{equation*} 
 \begin{array}{l}
 \{ u_0 \approx 0.19486610, \ u_1 \approx 0.15924707, \ u_2 = u_3 \approx 0.17659278, \ x_2 \approx 1.0016957, \ x_3 \approx 0.18639713 \},
 \\
\{ u_0 \approx 0.37393672, \ u_1 \approx  0.21919659,\ u_2 = u_3 \approx 1.1626641, \ x_2 \approx  1.0087519, \ x_3 \approx  0.34082479 \}, 
\\
\{ u_0 \approx 0.84212893, u_1 \approx 0.10872180,\ u_2 = u_3 \approx 0.178610960, \ x_2 \approx 0.58058297, \ x_3 \approx 1.1408077 \}, 
\\
\{ u_0 \approx 1.5046268, \ u_1 \approx 1.4017692, \ u_2 = u_3 \approx 0.23783921,\ x_2 \approx 1.0550442, \ x_3 \approx 1.4812096 \},  
\\
\{ u_0 \approx 2.3653572, \ u_1 \approx 0.16989789, \ u_2 = u_3 \approx 0.26106774, \ x_2 \approx 1.6915562, \ x_3 \approx 2.3587740 \} 
    \end{array}
    \end{equation*}
and 
 \begin{equation*} 
 \begin{array}{l}
 \{ u_0 \approx 1.3387553, \ u_1 \approx 0 .24195763, \ u_2 = u_3 \approx 0.34572988,\ x_2 \approx 3.8653526, \ x_3 \approx -0.80052360 \},
\\
\{ u_0 \approx  1.1975502, \ u_1\approx 0.69035070,  \ u_2 = u_3 \approx 0.37629335,  \ x_2 \approx 3.8406206, \ x_3 \approx -0.75956931 \}.
 \end{array}
    \end{equation*}
Therefore, we obtain five Einstein metrics which are non-naturally reductive  by Proposition \ref{prop5.5}. 
We also see that  these five metrics are non-isometric each other,   by computing the scale invariants  (cf.   \cite{Chry2}).

Consider now the case  $(x_3-1) (17 x_3 - 3) (105 {x_3}^2-180 x_3+43) =0$.
For  $105 {x_3}^2-180 x_3+43 = 0 $, we get that $ x_2= 1$, $u_0 = u_1 = x_3$ and   $35 x_3 + 43 u_2 - 60 = 0$, 
 as solutions of the system of equations $\{ g_0 =0, g_1=0, g_2=0, u_2 = u_3, g_4=0, g_5=0 \}$ with $ u_0 \, u_1 \,u_2 \, u_3 \, x_2 \, x_3 \neq 0$. 
It follows that the following values are real solutions of the homogeneous Einstein equation:
\begin{equation*} 
 \begin{array}{l}
\{ u_0 =  u_1  = x_3  = \displaystyle \frac{1}{105}
   \left(90-\sqrt{3585}\right),\ u_2 = u_3 = \displaystyle \frac{1}{129}
   \left(90+\sqrt{3585}\right),  \ x_2 =1 \}, \\ 
\\
\{ u_0 =  u_1  = x_3  = \displaystyle \frac{1}{105}
   \left(90+\sqrt{3585}\right),\ u_2 = u_3= \displaystyle \frac{1}{129}
   \left(90-\sqrt{3585}\right),  \ x_2 =1 \}.
\end{array}
\end{equation*} 
For $ x_3= 1$, we see that  $(x_2-1) (319 {x_2}^3-585 {x_2}^2+298 x_2-46) =0$, $u_0  =  u_2= x_2$ and $ 23 u_1-638 {x_2}^3+1808 {x_2}^2-1513 x_2+320=0$. In this case we obtain the following solutions:
 \begin{equation*} 
 \begin{array}{l}
\{ u_0  = u_2 = u_3 = x_2 \approx 0.32447061, \ u_1 \approx 0.10305040,  \ x_3 =1 \}, \\ 
\{  u_0  = u_2 = u_3  = x_2 \approx 0.40093736, \ u_1 \approx 1.6130682,  \ x_3 =1 \}, \\ 
\{  u_0  = u_2 = u_3 = x_2 \approx 1.1084478, \ u_1 \approx 0.19842410,  \ x_3 =1 \} 
\end{array}
\end{equation*} 
and $ u_0 = u_1 = u_2= u_3= x_2 = x_3 = 1$. 
Similarly, for $17 x_3 - 3=0$, we see that $x_2 = 1$ and $ u_0 = u_1 = u_2= u_3=  x_3 = 3/17$. 
According to  Proposition \ref{prop6.4}, these  values  induce only  naturally reductive Einstein metrics.

\smallskip 
 
 \noindent{\bf \underline{Case of $ 2 {u_2} {u_3} {x_2}^2+{u_2} {u_3}-{x_2}^2 =0$.}}
  
  We consider the polynomial ring $R= {\mathbb Q}[z,  u_0,  u_1,  u_2, u_3, x_2,  x_3] $ and an ideal $I$ generated by polynomials
$\{ g_0, \, g_1, \, g_2, \, 2 {u_2} {u_3} {x_2}^2+{u_2} {u_3}-{x_2}^2, \ g_4, \, g_5,  \,z \,  u_0 \, u_1 \,u_2 \, u_3 \, x_2 \, x_3-1\}$.   We take a lexicographic ordering $>$,  with $ z >  u_0 >  u_1>  u_2>  u_3  > x_2 > x_3$ for a monomial ordering on $R$. Then,  a  Gr\"obner basis for the ideal $I$ contains a  polynomial  of   $x_3$,  given by 
$(129 {x_3}^2-180 x_3+35) h_2(x_3)$, where
$h_2(x_3)$ is a polynomial  of degree 62 of the following form:
 { \small 
 \begin{equation*} 
\begin{array}{l} h_2(x_3)=
7176902311937267597352960000{x_3}^{62}  -63935286731053017507053568000{x_3}^{61}\\ +399843715184474955262341120000{x_3}^{60}  -1965764640832898839737014200320{x_3}^{59} \\ +8193254275314202221366521370624{x_3}^{58}  -29792334645086350198275186421248{x_3}^{57} \\ +96312668127950023620573089912976{x_3}^{56}  -279881292674846037180105828690048{x_3}^{55} \\ +736814302624650073499123557784400{x_3}^{54}  -1765609824944028233625757684850592{x_3}^{53} \\ +3858859512197506584743793325704936{x_3}^{52}  -7693673293317379815378916857414144{x_3}^{51} \\ +13978077490223112750713848520000136{x_3}^{50}  -23093289922724516044177470806375016{x_3}^{49} \\ +34590603147607745081071881647260329{x_3}^{48} -46777732301916799546957263457901784{x_3}^{47} \\ +56736923977077571150912980537328233{x_3}^{46}  -61045239892045312401412474185577596{x_3}^{45} \\ 
+57108713980664370406222711634860701{x_3}^{44}  -44533258600030824329782702723403172{x_3}^{43} \\ +25754099801447208802135893114691377{x_3}^{42}  -5506773889529575627395134456923560{x_3}^{41}  \\ 
-10900969369323449326468798316973849{x_3}^{40}+19761681445815159942981174871733088{x_3}^{39}  \\ -20188941574110618150578705304430161{x_3}^{38}+14202358308412638850212352987090788{x_3}^{37}  \\ -5742232611467996731383311668741917{x_3}^{36}  -1570190865515015457366830818567044{x_3}^{35} \\ +5648740711606386591873392674677879{x_3}^{34}  -5948880658615999334374796571219144{x_3}^{33} \\+3868106929529148762149605373828154{x_3}^{32}  -1258995580792601908577943150171216{x_3}^{31}  \\ -853879665000749089689937169600598{x_3}^{30}+1715750052119667014938302077428104{x_3}^{29} \\ -1499771919293409567273925312020942{x_3}^{28}+912888495260506476555714495569400{x_3}^{27} \\
-254675582893363617963644083458150{x_3}^{26}  -156524700943028756585375788114864{x_3}^{25} \\ +250591779632840858079012274593070{x_3}^{24}  -210754892342268589669423384800992{x_3}^{23} \\ +104847524519386925269074792690078{x_3}^{22}  -16046875591916438782483046815032{x_3}^{21}  \\ -19822089010062331168255611933090{x_3}^{20}+27934823512978080917038217731864{x_3}^{19}  \\ -18891540266011228264345732880890{x_3}^{18}+7806076572680194957966216721384{x_3}^{17}  \\ -1269024611811072417466760739747{x_3}^{16}  -1623952795000087557877271247768{x_3}^{15} \\ +1876612307555759003085287593757{x_3}^{14}  -1325495568850663913046822215564{x_3}^{13} \\ +722686827400045498790015516321{x_3}^{12} -329147016934629053928253191636{x_3}^{11} \\+129532305407715005005113350181{x_3}^{10}  -44529061623094804552954953960{x_3}^{9} \\ +13469000037841152898081791595{x_3}^{8} -3593330620975215538779860800{x_3}^{7} \\ +840650703590361114731290675{x_3}^{6}   -171221634223124842794559500{x_3}^{5} \\ +29979979441425878491638375{x_3}^{4}  -4380052206593637030292500{x_3}^{3} \\ 
+514712484479377144996875{x_3}^{2}  -45016029840810490875000 {x_3}+2277208974357528750000. 
     \end{array}
    \end{equation*}
}

Solving $h_2(x_3)=0$ numerically, we find four positive  and two negative solutions, which are given approximately by $x_3 \approx 1.0032864, x_3 \approx 0.25857941, x_3 \approx 1.3899161, x_3 \approx 1.4812096, x_3 \approx 1.0082270$ and  $x_3 \approx  -0.71677982, x_3 \approx -0.70718824$. In particular, real solutions of the system   $\{ g_0 =0, g_1=0, g_2=0, 2 {u_2} {u_3} {x_2}^2+{u_2} {u_3}-{x_2}^2=0, g_4=0, g_5=0 \}$ with $ u_0 \, u_1 \,u_2 \, u_3 \, x_2 \, x_3 \neq 0$ have the form
 \begin{equation*} 
 \begin{array}{l}
 \{ u0 \approx 0.35854650, u1 \approx 0.10372795, u2  \approx 0.42074677, u3  \approx 0.22748595, x2  \approx 0.34405535, x3  \approx 1.0032864 \},
 \\
  \{ u0 \approx 0.35854650, u1 \approx 0.10372795, u2  \approx  0.22748595, u3  \approx 0.42074677, x2  \approx 0.34405535, x3  \approx 1.0032864 \},
 \\
\{ u0 \approx 0.27724447, u1\approx 0.19525496, u2 \approx 1.4284934, u3 \approx 0.23400191, x2 \approx 1.0042307, x3 \approx 0.25857941 \}, 
 \\
\{ u0 \approx 0.27724447, u1\approx 0.19525496, u2 \approx 0.23400191, u3 \approx  1.4284934, x2 \approx 1.0042307, x3 \approx 0.25857941 \}, 
\\
\{ u0 \approx 1.5127009, u1 \approx 0.18053147, u2 \approx 1.3241010, u3 \approx 0.28487472, x2 \approx 1.2393058, x3 \approx 1.3899161 \}, 
\\
\{ u0 \approx 1.5127009, u1 \approx 0.18053147, u2 \approx 0.28487472, u3 \approx 1.3241010, x2 \approx 1.2393058, x3 \approx 1.3899161 \},
\\
\{ u0 \approx 0.47336119, u1 \approx 1.5881009, u2 \approx 0.25625661, u3 \approx 0.55478019, x2 \approx 0.44569962, x3 \approx 1.0082270 \},
\\
\{ u0 \approx 0.47336119, u1 \approx 1.5881009, u2 \approx 0.55478019, u3 \approx 0.25625661, x2 \approx 0.44569962, x3 \approx 1.0082270 \} 
    \end{array}
    \end{equation*}
and 
 \begin{equation*} 
 \begin{array}{l}
 \{ u0 \approx 1.0325160, u1 \approx 0.34059800, u2 \approx 1.1055012, u3 \approx 0.435626061, x2 \approx 3.6160749, x3 \approx -0.71677982 \},
\\
\{ u0 \approx 1.0325160, u1 \approx 0.34059800, u2 \approx 0.435626061, u3 \approx 1.1055012, x2 \approx 3.6160749, x3 \approx -0.71677982 \},
\\
\{ u0 \approx 1.0018570, u1 \approx 0.49052811, u2 \approx 1.0713232, u3 \approx 0.44960349, x2 \approx 3.6248195, x3 \approx -0.70718824 \},
\\
\{ u0 \approx 1.0018570, u1 \approx 0.49052811, u2 \approx 0.44960349, u3 \approx 1.0713232, x2 \approx 3.6248195, x3 \approx -0.70718824 \}.
 \end{array}
    \end{equation*}
Based now on Proposition \ref{prop6.4}, we conclude that these  eight Einstein metrics are non-naturally reductive.
Notice that metrics obtained by exchanging $u_2$ and $u_3$ are isometric. A final computation of the induced scale invariants allows us to deduce that  there are  
four Einstein metrics  which are non-isometric each other.

\smallskip

For  $129 {x_3}^2-180 x_3+35 = 0 $, we get that $ x_2= 1$, $u_0 = u_1 = x_3$,  $24 u_3 x_3 + 105 {u_3}^2 - 180 u_3 + 35 = 0$ 
and $8 x_3 + 35 u_2 + 35 u_3 - 60 =0 $ for solutions of the system of equations $\{ g_0 =0, g_1=0, g_2=0,  2 {u_2} {u_3} {x_2}^2+{u_2} {u_3}-{x_2}^2=0, g_4=0, g_5=0 \}$ with $ u_0 \, u_1 \,u_2 \, u_3 \, x_2 \, x_3 \neq 0$. 
Thus,  in this case   solutions of the homogeneous Einstein equation  are given by
\begin{equation*} 
 \begin{array}{l}
\{ u_0 =  u_1  = x_3  = \displaystyle \frac{1}{129}
   \left(90-\sqrt{3585}\right) =u_3,\ u_2 = \displaystyle \frac{1}{105}
   \left(90+\sqrt{3585}\right),  \ x_2 =1 \}, \\ 
\\
\{   u_0 =  u_1  = x_3  = \displaystyle \frac{1}{129}
   \left(90-\sqrt{3585}\right) =u_2,\ u_3 = \displaystyle \frac{1}{105}
   \left(90+\sqrt{3585}\right),  \ x_2 =1 \}, \\  \\ 
\{ u_0 =  u_1  = x_3 =  \displaystyle \frac{1}{129}
   \left(90+\sqrt{3585}\right) = u_3, \ u_2 = \displaystyle \frac{1}{105}
   \left(90-\sqrt{3585}\right),  \  x_2 =1 \}, \\   \\
\{  u_0 =  u_1  = x_3 =  \displaystyle \frac{1}{129}
   \left(90+\sqrt{3585}\right) = u_2 , \ u_3 = \displaystyle \frac{1}{105}
   \left(90-\sqrt{3585}\right),  \  x_2 =1 \}.
\end{array}
\end{equation*} 
By Proposition \ref{prop6.4}, these solutions
 give rise to   naturally reductive Einstein metrics.

 \end{document}